\documentclass[11pt,a4paper]{amsart}

\usepackage[foot]{amsaddr}

\usepackage{amsthm}
\usepackage{amsmath}
\usepackage{amssymb}
\usepackage{amsfonts}
\usepackage{bm}
\usepackage[top=3cm,bottom=3cm,left=2.5cm,right=2.5cm]{geometry}
\usepackage{enumitem}
\usepackage[dvipsnames]{xcolor}
\usepackage{graphicx}
\usepackage{multirow}
\usepackage{booktabs}
\usepackage{soul}
\usepackage{subfigure}
\usepackage[normalem]{ulem}

\newtheorem{theorem}{Theorem}
\newtheorem{lemma}[theorem]{Lemma}
\newtheorem{proposition}[theorem]{Proposition}

\newtheorem{assumption}{Assumption}
\newtheorem{definition}{Definition}
\newtheorem{remark}{Remark}
\newtheorem{example}{Example}

\numberwithin{equation}{section}

\newcommand{\angstrom}{\mbox{\normalfont\AA}}

\def\N{\mathbb{N}}
\def\Z{\mathbb{Z}}
\def\R{\mathbb{R}}
\def\bbC{\mathbb{C}}
\def\bbS{\mathbb{S}}

\def\br{{\bm r}}
\def\bk{{\bm k}}
\def\bl{{\bm l}}
\def\bbm{{\bm m}}
\def\bz{{\bm z}}

\def\bmu{{\bm \mu}}
\def\bn{{\bm n}}

\def\bL{{\bm L}}

\def\calB{\mathcal{B}}
\def\calU{\mathcal{U}}
\def\calP{\mathcal{P}}
\def\calN{\mathcal{N}}
\def\calR{\mathcal{R}}
\def\calM{\mathcal{M}}
\def\calE{\mathcal{E}}

\def\calF{\mathcal{F}}
\def\calC{\mathcal{C}}
\def\calL{\mathcal{L}}

\def\intD{\bar{D}}
\def\diagD{{\bm D}}

\def\Ib{I_b}

\def\Asym{\mathcal{A}}

\def\rcut{r_{\rm cut}}
\def\ro{{r_0}}
\def\fcut{f_{\rm cut}}
\def\rnn{r_{\rm nn}}

\def\<{\langle}
\def\>{\rangle}

\def\b{\big}

\def\Spec{{\bf Spec}}
\def\SpecV{{\bf Spec}^V}

\newcommand{\clospan}[2]{\overline{\operatorname{span}}_{#2} #1}

\newcommand\blfootnote[1]{%
  \begingroup
  \renewcommand\thefootnote{}\footnote{#1}%
  \addtocounter{footnote}{-1}%
  \endgroup
}

\newcommand{\revi}[1]{{ #1}}
\newcommand{\revii}[1]{{ #1}}

\title{Atomic Cluster Expansion: Completeness, Efficiency and Stability}

\author{Genevieve Dusson}
\address{Laboratoire de Math\'ematiques de Besan\c{c}on, UMR CNRS 6623,
Universit\'e Bourgogne Franche-Comt\'e,
16 route de Gray,
25030 Besan\c{c}on, France}
\email{genevieve.dusson@math.cnrs.fr}
\author{Markus Bachmayr}
\address{Johannes Gutenberg-Universitat Mainz, Institut f\"ur Mathematik, Staudingerweg 9
55128 Mainz, Germany}
\author{Gabor Csanyi}
\address{Engineering Laboratory,
University of Cambridge,
Trumpington Street,
Cambridge, CB2 1PZ,
United Kingdom}
\author{Ralf Drautz}
\address{ICAMS, Ruhr-Universit\"{a}t Bochum, Universit\"{a}tsstr. 150, 44801 Bochum, Germany}
\email{ralf.drautz@rub.de}
\author{Simon Etter}
\address{Department of Mathematics
National University of Singapore}
\email{ettersi@nus.edu.sg}
\author{Cas van der Oord}
\address{Engineering Laboratory,
University of Cambridge,
Trumpington Street,
Cambridge, CB2 1PZ,
United Kingdom}
\author{Christoph Ortner}
\address{Mathematics Institute, University of Warwick, Coventry CV4 7AL, United Kingdom}
\email{c.ortner@warwick.ac.uk}

\date{\today}

\begin{document}

\begin{abstract}
    The {\em Atomic Cluster Expansion} (Drautz, Phys. Rev. B 99, 2019) provides a framework to systematically derive polynomial basis functions
    for approximating isometry and permutation invariant functions, particularly
    with an eye to modelling properties of atomistic systems. 
    \revi{
    Our presentation extends the derivation by proposing a precomputation algorithm that yields immediate guarantees that a complete basis is obtained.
    }
    We provide a fast recursive algorithm for efficient evaluation and illustrate its performance in numerical tests. Finally, we discuss generalisations and open challenges, particularly from a numerical stability perspective, around basis optimisation and parameter estimation, paving the way towards a comprehensive analysis of the
    convergence to a high-fidelity reference model.
  \end{abstract}

\maketitle

\blfootnote{
    {\it Acknowledgements:} This work was partially supported by the French ``Investissements d'Avenir" program, project ISITE-BFC (contract ANR-15-IDEX-0003). CO acknowledges funding by the Leverhulme trust, RPG-2017-191.
}

\setcounter{tocdepth}{1}

\section{Introduction}
Many functions and functionals of interest for different scientific  domains exhibit symmetries.
In this research note, we are targeting the approximation of functions that are invariant under isometry and permutations of its variables.
This is in particular a feature of many-particle models in physics, chemistry and materials science, e.g. potential energy \revii{surfaces and landscapes}, or Hamiltonians. There has been a long history of constructing accurate approximations of such functions, based on empirical modelling combined with parameter estimation and more recently machine learning techniques. While the proposed representations and fitting procedures can be very different, one characteristic usually remains: the symmetries of the function of interest are included in the representation, hence are exactly preserved by the approximation.

We are particularly interested in the approximation of \revii{the potential energy of materials and molecules}, for which several methods relying on permutation and isometry invariant representations have been proposed~\cite{Braams2009-wi, Behler2007-ng, Bartok2010-mv, Bartok13, Thompson2015-af, Eickenberg-17, Shapeev2016-pd, Drautz2019-er, hansen-2013, FCHL-18, vonl+15ijqc,deepmd-18}.
Many of these representations  employ spherical harmonics, which are indeed natural objects to describe rotations.
This is the case for the Smooth Overlap of Atomic Positions (SOAP)~\cite{Bartok13}, where a rotation- and permutation-invariant environment descriptor is constructed as the power spectrum of a spherical harmonics expansion of the atomic density, and then is employed in a Gaussian process regression (GAP)~\cite{Bartok2010-mv}.
The Spectral Neighbor Analysis Potential (SNAP)~\cite{Thompson2015-af} builds on a similar (extended) descriptor called the {\em bispectrum}~\cite{Bartok2010-mv,Bartok13}, but employs a linear regression scheme instead of a Gaussian process framework.

An alternative approach is to directly generate a basis spanning an approximation space that is invariant under rotation and permutation.
For example, the Permutation-Invariant Potentials (PIPs)~\cite{Braams2009-wi} and its extension to materials, ``atomic PIPs'' (aPIPs) ~\cite{2019-regapips1}, rely on computational invariant theory to construct a polynomial basis that is invariant under permutations and rotations.
Another very successful construction of interatomic potentials via symmetric polynomials are the Moment Tensor Potentials (MTPs) by Shapeev~\cite{Shapeev2016-pd}. The latter relies on tensor contractions to obtain isometry-invariance and the density projection pioneered in \cite{Behler2007-ng,Bartok2010-mv} to obtain permutation invariance.
Finally, the construction which we take as our starting point in this article is called Atomic Cluster Expansion (ACE) proposed by Drautz~\cite{Drautz2019-er}. This method reformulates and significantly extends the SOAP construction
to obtain a complete set of invariant polynomials where the angular component is described by spherical harmonics.

Our present work serves multiple purposes: first, it serves as an introduction, for a numerical analysis audience, to this emerging field, emphasizing a particularly \revii{elegant, general and systematic} approach based on symmetric polynomials exploiting the properties of spherical harmonics. (See also \S~\ref{app:othersym} for an example of the non-spherical case.)
\revi{Secondly, we present a more detailed derivation of~\cite{Drautz2019-er}, to critically examine implicit underlying assumptions, and to clarify why the ACE construction indeed generates a spanning set. We further provide a numerical algorithm that can be run offline to finalise the construction and obtain a {\em complete} basis. 
}
\revii{Shapeev~\cite{Shapeev2016-pd} has also shown that the MTP construction yields a spanning set, but obtaining a basis {\em offline} in that setting is more difficult. Instead, he removes linearly dependent basis functions through an $l_0$ sparsification (see Appendix~\ref{sec:mtps}). Both Shapeev's work and our current work serve as a starting point for a rigorous approximation error analysis that we are pursuing in a separate work to explain the exceptional performance of these methods.
}

We then discuss a wide range of modifications, extensions and optimisations, and highlight related open challenges.
For example, in \S~\ref{sec:graph_representation} we propose an efficient {\em recursive representation} of the ACE basis in terms of a directed acyclic graph which has a significantly reduced evaluation cost. 
In \S~\ref{sec:reg} we discuss the inverse problem associated with estimating the model parameters, briefly examining its well-posedness (or lack thereof) and the need for sophisticated regularisation techniques, all aimed towards the construction of {\em stable} regression schemes. 
We illustrate computational aspects of these issues numerically; in particular, we present convergence tests in \S~\ref{sec:convergence_tests} for two training sets (Silicon and Tungsten), demonstrating the systematic convergence, low computational cost of the method, and the speed-up afforded by the recursive representation.
In \S~\ref{sec:extensions} we review several extensions and open problems including natural generalisations to {\em nonlinear} regression schemes; ill-conditioning of the original ACE construction and how this might be reversed through an explicit orthogonalisation; regularisation; and the application of the cluster expansion method to different symmetries.
Finally, to provide some broader context we briefly review several closely related schemes in Appendix~\ref{app:otherbases}.

To conclude this introduction we emphasize that our note should primarily be seen as a first step towards a numerical analysis of symmetric polynomials in the context of fitting interatomic potentials or more generally properties of molecular systems. Applications of these methods to the development of practical interatomic potentials will be pursued in separate works.

\subsection{Notation}
We briefly summarize key notation used throughout this paper:
\begin{itemize}
    \item $\N = \{0, 1, \dots\}$ : natural numbers, beginning with zero
    \item $\clospan{B}{X}$ : closure of ${\rm span} B$ with respect to the norm $\|\cdot\|_X$
    \item $R = ( \br_j )_{j = 1}^J$ : collection of, possibly relative, particle positions, $\br_j \in \R^3$
    \item $\br = r \hat{\br}$, i.e., $r = |\br|$ and $\hat\br = \br/r$
    \item $B_s = \{ \br \in \R^3 \,\b|\, | \br | < s \}$, $B_s^{\rm c} = \R^3 \setminus B_s$
    \item $\calP = \{P_0, P_1, \dots\}$ : radial basis functions
    \item $Y_l^m$ : complex spherical harmonics
    \item $\bbS^{d}$ : unit sphere in $d+1$ dimensions
    \item $\delta_{a} := \delta_{a0}$ \revi{is a Kronecker delta}, e.g., used in the form $\delta_{a+b} = \delta_{a,-b}$.
    \item If $a$ is a vector indexed by an index set $I$ then we denote its entries by either $a_i$ or $[a]_i, i \in I$. We use the latter notation e.g. when $a$ is given by a more complex expression. Similarly, for a matrix $A$ indexed by $I$ we will denote its entries by $A_{ij} = [A]_{ij}, i, j \in I$. The index set need not be a set of integers.
    \item If ${\bf a} = (a_1, a_2, \ldots, a_N)$ is a tuple and $\sigma\in S_N$ is a permutation, we denote by $\sigma({\bf a})$ or $\sigma {\bf a}$ the tuple $(a_{\sigma(1)}, \ldots, a_{\sigma(N)})$.
    \item We will sometimes use the labels `PI' for {\em permutation invariant}, `RI' for {\em rotation invariant} and `RPI' for {\em rotation and permutation invariant}.
    \item If $A$ is a set then $\# A$ denotes the number of set elements (possibly infinite).
\end{itemize}

\section{Potential Energy Surfaces}
\label{sec:pes}
We are concerned with the approximation of atomistic potential energy surfaces (PES). For simplicity of presentation, we only consider finite single-species clusters but extensions to infinite and multi-species configurations are straightforward; see \revi{Section}~\ref{sec:app:multi}.

The set of {\em admissible finite configurations} is given by
\[
  \calR := \bigcup_{J = 0}^\infty \calR_J, \qquad
  \calR_J := \b\{ \{\br_1, \dots, \br_J\} \subset \R^3  \b\};
\]
i.e., a configuration $R \in \calR_J$ is a set consisting of $J$ particle positions.

A potential energy surface is a mapping
\[
   E : \calR \to \R
\]
which is invariant under isometries (permutation invariance is already implicit in identifying configurations as sets)\revii{, that is, for an isometry $Q \in O(3)$, for any $(\br_1, \dots, \br_J) \in \R^{3J}$,
\[
E(\{Q\br_1, \dots, Q\br_J\}) = E(\{\br_1, \dots, \br_J\})
\]
} %
and observes a certain
locality of interaction\revii{, in the sense that the energy is additive for configurations composed of far away subsystems.
In fact, we will make this statement more precise (and even more restrictive) in Assumption~\ref{as:VN} just below.}

All {\em interatomic potential} models make various (often {\it ad hoc})
assumptions on the PES regarding low-rank structures and locality of
interactions. In general, one aims to represent a complex fully many-body PES
$E$ (exactly or approximately) as a combination of ``simple'' components, e.g.,
low-dimensional or low-rank. Here, we shall {\em assume} that $E$ can be written
in the form of a body-order expansion,
\begin{equation} \label{eq:cluster expansion}
\begin{split}
   E(\revii{\{\br_1, \dots, \br_J\}}) &= \sum_{\iota = 1}^J E_\iota\b(\{\br_j-\br_\iota\}_{j\neq\iota}\b), \\
   E_\iota\big(\{\br_{\iota j}\}_{j = 1}^{J-1}\big) &= \revii{V_0} + \sum_{N = 1}^\calN
         \sum_{j_1 < j_2 < \dots < j_N} V_N\b(\br_{\iota j_1}, \dots, \br_{\iota j_N}\b),
 \end{split}
\end{equation}
\revii{with $\br_{\iota j} = \br_j-\br_\iota$, $V_0 \in\R$ and $\calN \in \N$ being the maximal order of interaction. 
The functions $E_\iota$ which depend on a center atom $\br_\iota$ and its environment are usually called site energy functions, and the $V_N$ functions which depend on $N$ differences to a center atom $\iota$ are called $N$-body functions. }
We will discuss uniqueness (or lack thereof) of such a representation in \S~\ref{sec:regularisation}. 

Because the body-order expansion is applied to \revii{the site energy which is} an atom-centered property we call it an {\em atomic body-order expansion.} The concrete representation proposed in \cite{Drautz2019-er} uses symmetrized correlations to express terms in \eqref{eq:cluster expansion} and, motivated by the cluster expansion formalism in alloy theory \cite{Sanchez1984-iw}, is called {\em atomic cluster expansion} (ACE). We will also use the ACE name for subsequent modifications that build on \cite{Drautz2019-er}, such as those proposed in \S~\ref{sec:extensions}.

We will call the parameter $N$ in a term of the form $V_N(\br_{\iota j_1}, \dots, \br_{\iota j_N})$ the {\em interaction order}, or simply {\em order}, of that term. It can be identified as a {\em body-order} if an $N+1$-particle function involving $N$ neighbours of a centre-atom is interpreted as having body-order $N$.

\revii{Note that the arguments in $V_N$ are taken as (vectorial) differences with respect to a center atom, but for simplicity, we still denote them by $R = (\br_1, \dots, \br_N) \in \R^{3N}$.
Moreover,} we treat the $V_N$ as functions on $\R^{3N}$ rather than on
configurations, which makes it easier to reason about them, e.g., their
regularity. On the other hand we then need to place assumptions on  $V_N$ to
retain important physical symmetries. \revii{For} $R = (\br_1, \dots, \br_N) \in
\R^{3N}$, $Q \in \R^{3 \times 3}$ and $\sigma \in S_N$ a permutation, then we will write
\[
    QR := (Q \br_j)_{j = 1}^J, \qquad \text{and} \qquad
    \sigma R := (\br_{\sigma j})_{j = 1}^J.
\]
\begin{assumption}\label{as:VN}
We assume throughout that the $V_N$ satisfy the following conditions:
\begin{enumerate} \renewcommand{\labelenumi}{(\roman{enumi})}
  \item {\em Regularity: } $V_N \in C^{t}( \R^{3N} \setminus \{0\} )$; where  $t \geq 1$. The regularity parameter $t$ is fixed throughout the paper.
  \smallskip
  \item {\em Locality: } There exists $\rcut > 0$ such that
      $V_N(R) = 0$ if $\revii{\max_{1\le j \le N}} \revii{|\br_j|} \geq \rcut$;
  \smallskip
  \item {\em Isometry Invariance: }  $V_N(QR) = V_N(R)$ for all $Q \in {\rm O}(3)$; and
  \smallskip
  \item {\em Permutation Invariance: } $V_N(\sigma R) = V_N(R)$ for all $\sigma \in S_N$.
\end{enumerate}
\end{assumption}

\begin{remark}
A rigorous justification of these assumptions requires in particular quantifying the error committed by truncation of the interaction range and order and goes well beyond the scope of this work. We can only briefly comment on each assumption in turn:

{\it (i) Regularity} and {\it (ii) Locality} has been explored for some simple electronic structure models in \cite{2015-qmtb1,2019-tbloc0T,Drautz06}.
For more sophisticated models it is not clear whether, or how, these assumptions may fail. However, it is likely that $C^{1}$ or higher regularity are only achievable in configurations where no two particles collide. This requires a weakening of~(i) and leads to substantial additional difficulties for approximation theory.

The symmetries {\it (iii, iv)} are entirely natural and likely universal.

Regarding the truncation of the interaction order, since we place no bounds on the maximal order $\calN$ it is reasonable to assume that a broad class of PESs can be represented within this assumption to within arbitrary accuracy.
Indeed, the expansion \eqref{eq:cluster expansion} is an implicit ingredient in many regression schemes for potential energy surfaces~\cite{Bartok2010-mv,Shapeev2016-pd,Thompson2015-af,Braams2009-wi}.
It was first suggested in~\cite{Glielmo2018-lk} that it may be advantageous to directly target ``learning'' the $N$-body terms $V_N$.
\end{remark}

In the remainder of this paper we will present a framework, building on
\cite{Drautz2019-er, Bartok2010-mv, Shapeev2016-pd}, how to construct a broad
class of polynomial approximations capable of representing any order-$N$
potential $V_N$, while respecting the symmetry and cutoff requirements, and
which moreover can be efficiently evaluated even at high interaction order $\calN$.
Note that the cost of evaluating \revii{a site energy $E_\iota$} scales naively as $\binom{J}{\calN}$,
but it was shown in \cite{Drautz2019-er, Bartok2010-mv, Shapeev2016-pd} how this
cost can be significantly reduced.

\subsection{Potential spaces}
Since $N$-body potentials typically have a singularity at $r = 0$, we will not study convergence up to collision but instead specify some minimal distance $\ro > 0$, which will be fixed throughout this paper and only consider interatomic neighbourhoods $\{\br_1, \dots, \br_J\}$ for which $r_j \geq \ro$ for all $j$. Thus, given $\ro > 0$, we define the domains
\[
    \Omega_\ro := \{ \br \in \R^3 \,|\, |\br| > \ro \},
    \quad \text{and} \quad
    \Omega_\ro^N =
    \big\{ (\br_1, \dots, \br_N) \in \R^{3N} \, |\,
            \min r_j > \ro \big\}.
\]
The associated function spaces, motivated also by Assumption~\ref{as:VN}, are
\begin{align*}
    C^{s}_{\rcut}(\ro, \infty)
    &:= \big\{ f \in C^{s}(\ro, \infty) \,|\, f = 0 \text{ in } [\rcut, \infty) \big\},
    \\
    C^{s}_{\rcut}(\Omega_\ro^N) &:=
    \big\{ f \in C^{s}(\Omega_\ro^N) \,|\, f(\br_1, \dots, \br_N) = 0 \text{ if } r_j \geq \rcut \text{ for any } j \big\},
\end{align*}
equipped with the standard norms $\| \cdot \|_{C^{s}(\Omega_\ro^N)}, s \geq 0$. When it is safe from the context to do so, we will drop the domains from the definitions of the spaces and norms.

\subsection{Approximation by tensor products}
\label{sec:pes:approx}
Informally for now, we approximate a multi-variate function
$
    f : \R^{3N} \to \R
$
(in particular $V_N$) using a tensor product basis
\[
    \phi_{\bn\bl\bbm}
    \b( \revi{\br_1, \br_2, \ldots, \br_N})
    := \prod_{j = 1}^N \phi_{n_j l_j m_j}(\br_j),
    \qquad
    \phi_{nlm}(\br) := P_{n}(r) Y_{l}^{m}(\hat\br),
\]
\revii{with $\bn = (n_1, n_2, \dots, n_N)$, and similarly for $\bl$ and $\bbm$, }
where $P_n, n = 0, 1, 2, \dots$ are the {\em radial basis functions} defined below, while $Y_l^m, l = 0, 1, 2, \dots; m = -l, \dots, l$ are the standard complex spherical harmonics; see Appendix~\ref{sec:app:SH} for a brief introduction. In particular, for $\bl \in \N^N$, we have the restriction
\[
    \bbm \in \calM_\bl :=
    \big\{
        \bmu \in \Z^N \,|\, -l_\alpha \leq \mu_\alpha \leq l_\alpha
    \big\}.
\]
The choice of spherical harmonics for the angular component is crucial and will later enable us to explicitly incorporate ${\rm O}(3)$-symmetry into the approximation.

On the other hand, there is significant freedom in the choice of radial basis, which we denote by
\[
    \mathcal{P} := \big\{ P_n(r) \,\b|\, n = 1, 2, \dots \big\}.
\]
We will make the following standing assumption.

\begin{assumption} \label{as:P}
   The set $\calP$ is a linearly independent subset of
   $C^{t}_{\rcut}([\ro,\infty))$,
   where $t$ is the parameter from Assumption~\ref{as:VN}. Moreover,
   \[
      \clospan{\calP}{C^t}
      \supset
      \big\{ f \in C^{\infty}([r_0, \infty)) \,\b|\, f = 0 \text{ \rm in }
            [\rcut, \infty) \big\},
   \]
\end{assumption}

Intuitively, Assumption~\ref{as:P} states that due to rotation invariance, any pair potential $V_1(\br) = V_1(r)$ may be approximated to within arbitrary accuracy from $\operatorname{ span} \calP$. As a consequence we obtain that any $V_N$ satisfying the requirements of Assumption~\ref{as:VN} can be approximated using the tensor product basis
\[
    \Phi_N :=
    \big\{
        \phi_{\bn\bl\bbm} \,\big|\, \bn, \bl \in \N^N, \bbm \in \calM_\bl
    \big\}.
\]

\begin{proposition} \label{th:approx_nonsym}
    Suppose that $V_N$ satisfies Assumption~\ref{as:VN}(i, ii),
    and $\calP$ satisfies Assumption~\ref{as:P}; then,
    $V_N|_{\Omega_{\ro}^N} \in \clospan{\Phi_N}{C^t}$.
\end{proposition}

\begin{proof}
    We first show that $\Phi_1 = \{ \phi_{nlm} \}$ is dense with respect to the $C^t(\overline{\Omega}_{\ro})$-norm in
    \[ Y = \{ f \in C^{t}(\overline{\Omega}_{\ro}) \,\b|\, f = 0 \text{ in }
            B_{\rcut}^{\rm c} \big\} .  \]
    \revii{Denoting by $\operatorname{supp} f$ the support of a given function $f$, }
    it suffices to show that any $v \in \{ f \in C^\infty(\overline{\Omega}_{\ro}) \,\b|\, \operatorname{supp} f \subset B_{\rcut} \}$ can be approximated arbitrarily well by linear combinations of the basis functions, since this subset is dense in $Y$. Let $\kappa \in C^\infty(\overline{\Omega}_{\ro})$ be a radial function taking values in $[0,1]$ such that $\kappa = 1$ on $\operatorname{supp} v$ (that is, $\kappa  v = v$) and $\operatorname{supp} \kappa \subset \overline{B}_{\rcut}$.

    Let $\revii{\Ib} = [-b, b]^3$ with $b > \rcut$.
    By Whitney's extension theorem\revii{~\cite{Whitney1934-mf}}, there exists $\tilde v \in C^t(\revii{\Ib})$ such that $\tilde v|_{\revii{\Ib} \cap \overline{\Omega}_{\ro}} = v$. Note that approximation of $\tilde v$ by Bernstein polynomials on \revii{$\Ib$}, as in the corresponding proof of the Weierstrass approximation theorem, yields simultaneous approximation of $\tilde v$ and its continuous partial derivatives (see, e.g., \cite[Sec.~VI.6.3]{Davis}). Thus for any $\varepsilon>0$, we obtain a polynomial $p_\varepsilon$ such that
    \[
     \| v - p_\varepsilon \|_{C^t(\revii{\Ib} \cap \overline{\Omega}_{\ro})}
     \leq \| \tilde{v} - p_\varepsilon \|_{C^t(\revii{\Ib})}
     \leq \frac{\varepsilon}{ \| \kappa \|_{C^t}}.
     \]
     We can write $p_\varepsilon$ in the form
    \[
      p_\varepsilon(\br) = \sum_l r^l H_l(\hat\br),
    \]
    where $H_l$ is a homogeneous polynomial of degree $l$ for $l\geq 0$. Since $H_l|_{\mathbb{S}^2} \in \operatorname{span}\{ Y_l^m \colon m=-l,\ldots,l\}$ by \cite[Cor.~IV.2.2]{SteinWeiss}, we have univariate polynomials $q_l$ of degree $l$ such that
    \[
        p_\varepsilon(\br) = \sum_{l, m}  q_l(r) Y^m_l(\hat\br).
    \]
    Note next that one has the rough estimate
    \[
      \| v - \kappa p_\varepsilon \|_{C^t(\revii{\Ib} \cap \overline{\Omega}_{\ro})} = \| \kappa (v - p_\varepsilon) \|_{C^t( \revii{\Ib} \cap \overline{\Omega}_{\ro})} \leq \| \kappa \|_{C^t} \| v - p_\varepsilon \|_{C^t(\revii{\Ib}\cap \overline{\Omega}_{\ro})}  \leq \varepsilon.
    \]
    Thus finite sums of the form
    \[
     \kappa(r) p_\varepsilon(\br) = \sum_{l, m}  \kappa(r) \, q_l(r)\, Y^m_l(\hat\br),
    \]
    extended to $\mathbb{R}^3 \setminus \revii{\Ib}$ by zero, are dense in $Y$.
    By Assumption \ref{as:P}, the factors $\kappa  q_l \in C^\infty([\ro,\infty))$, which vanish on $[\rcut,\infty)$, can in turn be approximated to any accuracy in $C^t(\overline{\Omega}_\ro)$-norm. Altogether, this shows that $\Phi_1 = \{ \phi_{nlm} \}$ is dense in $Y$.

    As a consequence, the product basis $\Phi_N =\{ \phi_{\bn\bl\bbm} \}$ is dense in $\bigotimes^N_{i=1} Y$. This product space, which is a subspace of $C^t_{\mathrm{mix}} (\overline{\Omega}_{\ro}^N) = \bigotimes_{i=1}^N C^t(\overline{\Omega}_{\ro})$ (see, e.g., \cite[Ex.~4.80]{Hac12}), is in turn dense in $C^t(\overline{\Omega}_\ro^N)$. Since $V_N|_{\Omega_{\ro}^N} \in C^t(\Omega_\ro^N)$ by our assumptions, we arrive at the assertion.
\end{proof}

\section{Symmetric Polynomials}
\label{sec:symm}
We have seen in Proposition~\ref{th:approx_nonsym} that we can approximate an $N$-body function $V_N$ from the tensor product space ${\rm span}\{\phi_{\bn\bl\bbm}\}$. We will now show that (1) if $V_N$ satisfies the symmetries of Assumption~\ref{as:VN}(iii, iv) then the approximant may be assumed to inherit these symmetries without loss of accuracy; and (2) we will modify a construction from \cite{Drautz2019-er} to construct an explicit basis that respects these symmetries.

\revi{To motivate our presentation, which is somewhat different from \cite{Drautz2019-er}, we observe that there is a very general strategy to construct a basis respecting the desired symmetries at least in principle:
\begin{enumerate}[wide]
    \item Define the normalised Haar measure $H$ on the compact group \revii{ $G := \<S_N \cup {\rm O}(3)\>$} obtained by joining the permutation and isometry groups.
    \item Compute the symmetrised functions
    \[
        \phi_{\bn\bl\bbm}^{\rm sym}(R) := \int_{g \in G} \phi_{\bn\bl\bbm}(gR) \, H(dg)
    \]
    which may now be linearly dependent.
    \item Construct a basis of
    \[
        {\rm span}\big\{ \phi_{\bn\bl\bbm}^{\rm sym} \,|\,
                        \bn, \bl \in \N^N, \bbm \in \calM_\bl \big\}.
    \]
\end{enumerate}
Note that the symmetrised functions in (2) can be relatively easily constructed analytically while a basis of these symmetrised functions (that is, step (3)), is more difficult to obtain. In fact, this last part of the basis construction will be done numerically.
In what follows we will detail this construction when the symmetry group includes permutations, and isometries (cf. Assumption~\ref{as:VN}).
}
We will see that this leads to an explicit but computationally inefficient basis. We will then revisit a technique employed in \cite{Bartok2010-mv,Shapeev2016-pd,Drautz2019-er} to transform this basis into one that is {\em computationally efficient}. In preparation for this, we make the following observation:

Let $G$ be a group acting on $\R^d$, $\Omega \subset \R^d$ invariant under $G$\revi{,} and $X$ a subspace of functions $f : \Omega \to \R$ such that $f \circ g = f$ for all $g \in G$. Further, let a norm $\|\cdot\| : X \to \R$ be invariant under $G$, i.e.,
\[
    \|f\| = \|f \circ g\| \qquad \text{for all } g \in G.
\]
Next, let $\tilde{f} \in  X$ be some approximation to $f$, not necessarily  respecting the $G$ symmetries, and let
\[
    \tilde{f}^{\rm sym} := \int_G \tilde{f} \circ g \, H(dg)
\]
be the relevant symmetrised function then, since $f$ is invariant under $g \in G$ we have
\begin{align}
    \notag
    \big\| f - \tilde{f}^{\rm sym} \big\|
    &=
    \bigg\| \int_G (f - \tilde{f}) \circ g \, H(dg) \bigg\| \\
    \notag
    &\leq \int_G \| (f - \tilde{f}) \circ g \| \,H(dg) \\
    \label{eq:symmetrised approximation}
    &= \int_G \| f - \tilde{f} \| H(dg) =  \| f - \tilde{f} \|.
\end{align}
That is, the approximation error committed  by  $\tilde{f}^{\rm sym}$ is no larger than that of $\tilde{f}$. In the following sections we will use this observation implicitly several times.

\subsection{Permutation invariance}
\label{sec:symm:perm}
In addition to the requirements of regularity and locality encoded in $V_N \in C^t_{\rcut}(\Omega_\ro^N)$, let us now also assume that $V_N$ is permutation-invariant, i.e., that it satisfies Assumption~\ref{as:VN}(i, ii, iv).  Let $\tilde{V}_{N} \in {\rm span}\Phi_N$ be an approximation to $V_N$, and denote the symmetrised approximation by
\[
    \tilde{V}_N^{\rm perm}(R)
    :=
    \frac{1}{N!} \sum_{\sigma \in S_N} \tilde{V}_N(\sigma R).
\]
which we know from \eqref{eq:symmetrised approximation} is {\em at least as
accurate} as the original approximation $\tilde{V}_N$. We therefore assume from
now on that $\tilde{V}_N = \tilde{V}_N^{\rm perm}$, i.e., it is already
permutation symmetric.

Writing
\[
    \tilde{V}_N = \sum_{\bn, \bl, \bbm} c_{\bn\bl\bbm} \phi_{\bn\bl\bbm},
\]
then the linear independence of the $\phi_{\bn\bl\bbm}$ and the permutation symmetry $\tilde{V}_N = \tilde{V}_N \circ \sigma$ implies that
$
    c_{\bn\bl\bbm} = c_{\sigma\bn,\sigma\bl,\sigma\bbm}.
$
We can therefore alternatively write
\begin{equation} \label{eq:tilVN_perminv}
    \tilde{V}_N
    =
    \sum_{(\bn, \bl, \bbm)~{\rm ordered}}
    c_{\bn\bl\bbm} \sum_{\sigma \in S_N} \phi_{\bn\bl\bbm} \circ \sigma,
\end{equation}
with possibly {\em different} coefficients $c_{\bn\bl\bbm}$. By $\sum_{(\bn, \bl, \bbm)~{\rm ordered}}$ we mean that we sum over all {\em ordered} triples of tuples $(\bn, \bl, \bbm)$, according to the following definition, which we adopt throughout.

\begin{definition}
    We say that a tuple $({\bm a}^{(p)})_{p=1}^P \in (\Z^N)^P$ is {\em ordered} if the vector of tuples
    \[
        \Big( (a^{(p)}_1)_{p  = 1}^P, (a^{(p)}_2)_{p  = 1}^P,
             \dots, (a^{(p)}_N)_{p  = 1}^P  \Big)
    \]
    is lexicographically ordered.
\end{definition}

We emphasize  that any total ordering convention can be used, but we have found lexicographical ordering particularly convenient and intuitive.

\subsection{Invariance under point reflections}
\label{sec:symm:refl}
Next, we add point reflection symmetry to our approximation; that is, we assume that the potential $V_N$ satisfies
\[
    V_N(R) = V_N(-R), \quad \text{or, equivalently,} \quad
    V_N = V_N \circ J,
\]
where $J R = - R$. Combined rotation and point reflection symmetry are equivalent to ${\rm O}(3)$, i.e., symmetry under all isometries. Treating the point reflection case separately allows us an elementary demonstration how imposing invariance under isometries on the approximation can further reduce the number of admissible basis functions. We assume again without loss of accuracy that our approximation $\tilde{V}_N$ inherits the reflection symmetry.

Recalling from \eqref{eq:app:inv_Ylm} that $Y_l^m \circ J = (-1)^l Y_l^m$, which implies $\phi_{nlm} \circ J = (-1)^l \phi_{nlm}$  we can therefore write
\begin{align*}
    \tilde{V}_N(R)
    &=
    \sum_{\bn, \bl, \bbm} c_{\bn\bl\bbm} {\textstyle \frac12} \big(\phi_{\bn\bl\bbm} + \phi_{\bn\bl\bbm} \circ J \big)
    \\
    &=
    \sum_{\bn, \bl, \bbm} c_{\bn\bl\bbm} {\textstyle \frac12}(1 + (-1)^{\sum \bl}) \phi_{\bn\bl\bbm}.
\end{align*}
Thus, all basis functions $\phi_{\bn\bl\bbm}$ for which $\sum \bl$ is odd, vanish under this operation. That is, we only need to retain $(\bn, \bl, \bbm)$ tuples for which $\sum \bl$ is even. The resulting basis functions already respect reflection symmetry.

In summary, we have so far shown that we can approximate $V_N$ by symmetrised tensor products of the form
\begin{equation} \label{eq:start_rot}
    \tilde{V}_N
    =
    \sum_{\substack{(\bn, \bl, \bbm) {\rm ordered} \\
                    \sum \bl \text{ even}}}
    c_{\bn\bl\bbm} \sum_{\sigma \in S_N} \phi_{\bn\bl\bbm} \circ \sigma.
\end{equation}

\subsection{Rotation invariance: numerical construction}
\label{sec:symm:rot}
Finally, suppose that $V_N$ satisfies all the conditions of Assumption~\ref{as:VN}, including now also the rotation invariance, inherited by $\tilde{V}_N$. We may therefore rewrite it as
\[
    \tilde{V}_N
    =
    \sum_{\substack{(\bn, \bl, \bbm) {\rm ordered} \\
                    \sum \bl \text{ even}}}
    c_{\bn\bl\bbm} \sum_{\sigma \in S_N} \int_{\rm SO(3)}
    \big(\phi_{\bn\bl\bbm} \circ \sigma\big)(Q R) \,dQ,
\]
where the Haar integral $\int_{\rm SO(3)}$ is defined in Appendix~\ref{sec:app:SH}.

For a practical implementation it is necessary that we evaluate the integral over ${\rm SO}(3)$ explicitly. This is the key step where the spherical harmonics enter. Recall first that
\[
    \phi_{\bn\bl\bbm}(\br_1, \dots, \br_N)
    = \prod_{\alpha = 1}^N P_{n_\alpha}(r_\alpha) Y_{l_\alpha}^{m_\alpha}(\hat\br_\alpha),
\]
that is we only need to perform the integration over products of
$Y_{l_\alpha}^{m_\alpha}$ but can ignore the radial components, which are  already rotation-invariant.

\revii{Products of spherical harmonics arise in the coupling of angular momenta in the quantum mechanics literature, where they have been discussed extensively \cite{Drautz2019-er, Yutsis1962-yc, Brink68,Varshalovich88}. The product of several spherical harmonics can be expressed as linear combinations of single spherical harmonics, with the help of the generalized Clebsch--Gordan coefficients; see Appendix~\ref{sec:app:SH} for a brief summary.
Thus, rotationally invariant basis functions are formed by specific linear combinations of these products of spherical harmonics, 
which, when expressed as linear combinations of single spherical harmonics $Y_l^m$ involve only the rotationally-invariant spherical harmonic $Y_0^0$. 
In other words, the multiple angular momenta need to couple to zero total angular momentum.}

{We will first present a semi-numerical construction \revii{of a rotation-invariant basis}, which is an immediate consequence of properties of the spherical harmonics, and which is straightforward to implement and to generalise.
In the next section, we will then review an explicit algebraic construction based on reductions of the rotation group \cite{Yutsis1962-yc}, which is also the basis of~\cite{Drautz2019-er}. Both approaches are suitable for a practical implementation.}

Fix $\bl \in \N^N, \bbm \in \calM_\bl$ and \revii{$\hat R = (\hat\br_1, \dots, \hat\br_N)$}. The representation of rotated spherical harmonics in terms of the Wigner D-matrices,
\begin{equation}
    \label{eq:multi-wignerD-SH}
     Y_l^m(Q \revii{\hat\br}) = \sum_{\mu = -l}^l D^{l}_{\mu m}(Q) Y_l^\mu(\revii{\hat\br})
    \qquad \forall \revii{\hat\br \in \bbS^2}, Q \in {\rm SO}(3)
\end{equation}
(see \eqref{eq:app:wigner_D_matrix} for more details) immediately yields
\begin{align} 
    \notag
    Y_{\bl}^{\bbm}(Q \hat R) \;
    &= \sum_{\bmu \in \calM_\bl}
    \revii{D^\bl_{\bmu\bbm}(Q)} Y_\bl^\bmu(\hat R)
    \qquad \text{for all $Q \in SO(3)$,} \quad \text{where} \\
    \label{eq:def_Dmatrix}
\revii{D^\bl_{\bmu\bbm}(Q)} &= \prod_{\alpha = 1}^N  D_{\mu_\alpha m_\alpha}^{l_\alpha}(Q).
\end{align}
\revii{Note that with this definition of $D^\bl_{\bmu\bbm}(Q)$, one has to consider $(D^\bl_{\bmu\bbm}(Q))^T$ to mimic a matrix-vector product in~\eqref{eq:multi-wignerD-SH}.
}
Integrating with respect to $Q$ yields a rotation-invariant {\em  spanning set} $\{  b_{\bl\bbm}\}$ defined by
\begin{align}
    \notag
    b_{\bl\bbm}(\hat R) &:=
    \sum_{\bmu \in \calM_\bl} \revii{\intD^\bl_{\bmu\bbm}} Y_{\bl}^{\bmu}(\hat R),
    \quad \text{where} \\
    \label{eq:defn_intD}
    \revii{\intD^\bl_{\bmu\bbm}} &= \int_{\rm SO(3)} \revii{D^\bl_{\bmu\bbm}(Q)} \,  dQ.
\end{align}
\revii{These integrated coefficients $\revii{\intD^\bl_{\bmu\bbm}}$ can be efficiently computed via the recursion formula~\eqref{eq:Drecursion_appendix} involving Clebsch--Gordan coefficients presented in Appendix~\ref{sec:app:Dl_recursion}.
}

The following Lemma shows how to convert this spanning set into a {\em basis}.

\begin{lemma}
    \label{th:symm:orth_RI_basis}
    Suppose that $\tilde\calU^\bl_i = \revii{(\tilde\calU^\bl_{\bmu i}})_{\bmu \in \calM_\bl}$, $i = 1, \dots, \tilde{n}_\bl$ \revi{with $\tilde{n}_\bl = {\rm rank} \intD^\bl$} are orthonormal, i.e., $(\tilde\calU^\bl_i)^* \tilde\calU^\bl_{i'} = \delta_{ii'}$ and
    $\revii{ {\rm range} [\tilde\calU^\bl] = {\rm range} [\intD^\bl]}$, then the functions
    \[
        b_{\bl i} \revii{(\hat R)} := \sum_{\bmu \in \calM_\bl} \revii{ \tilde\calU^\bl_{\bmu i}} Y_\bl^\bmu(\revii{\hat R}),
        \qquad i = 1, \dots, \tilde{n}_\bl
    \]
    form an orthonormal basis of ${\rm span} \{ b_{\bl\bbm} \,|\, \bbm \in \calM_\bl \}$.

    In particular, we have that
    \[
        \big\{ b_{\bl i} \,\big|\, \bl \in \N^N, i = 1, \dots, \tilde{n}_\bl \big\}
    \]
    is an orthonormal basis of $\{ f \in L^2( (\bbS^2)^N) \,|\, f \text{ is rotation-invariant} \}$.
\end{lemma}
\begin{proof}
    Using the fact that the tensor products $Y_\bl^\bbm$ are orthonormal, we
    have
    \[
        \< b_{\bl i}, b_{\bl i'} \>
        =
        \sum_{\bmu, \bmu'} \tilde\calU^\bl_{\bmu i} \big(\tilde\calU^\bl_{\bmu' i'}\big)^* 
        \underset{ = \delta_{\bmu\bmu'}}{
            \underbrace{\< Y_\bl^{\bmu}, Y_{\bl}^{\bmu'} \>}
        }
        =
        \sum_{\bmu} \tilde\calU^\bl_{\bmu i} \big(\tilde\calU^\bl_{\bmu i'}\big)^*
        = \delta_{ii'}.
    \]
    This establishes orthonormality.

    To see that the basis is complete, we note that the range condition implies that $\revii{\intD^\bl_{\bmu\bbm} = \sum_i \alpha_i \tilde\calU^\bl_{\bmu i}}$, for some $(\alpha_i)$, and hence
    \[
        b_{\bl\bbm} = \sum_{\bmu} \revii{\intD^\bl_{\bmu\bbm}} Y_{\bl}^{\bmu}
        = \sum_i \alpha_i \sum_{\bmu} \tilde\calU^\bl_{\bmu i}Y_{\bl}^{\bmu},
    \]
    that is, $b_{\bl\bbm} \in {\rm span} \{ b_{\bl i} : i = 1, \dots, \tilde{n}_\bl \}$.

    The second statement is an immediate consequence.
\end{proof}

In \revii{practice} we can obtain the new coefficients $\tilde\calU^\bl$ via \revi{a singular value decomposition (SVD)}, which also provides a numerically stable estimate of the rank of $\intD^\bl$. The results of the next section show that $b_{\bl\bbm} = 0$ unless $\sum \bbm = 0$, that is, the SVD can be performed on a much smaller matrix in practice.

\subsection{Rotation symmetry: explicit construction}
\label{sec:rotinv:Drautz}
We now show that the numerical construction in the previous section corresponds to a {classical explicit SVD that is built on a reduction of the rotation group~\cite{Drautz2019-er, Yutsis1962-yc, Brink68, Varshalovich88}.} Here, we will include some additional details so we can clearly observe orthogonality of  the  rotation-invariant basis functions, as well as estimate the number of basis functions. In particular the latter is a key advantage of this alternative, but we will see that it still leads to some difficult challenges.

In the following, we review \revii{the argument that rotationally invariant basis functions are formed with specific linear combinations of products of spherical harmonics coupling to a zero total angular momentum} purely in the language of linear algebra. By explicitly constructing the SVD mentioned at the end of the previous section we show that we can explicitly build all linear combinations of products of spherical harmonics that are rotationally invariant, and this way generate a basis of rotation-invariant functions.

   \begin{lemma} \label{th:lemma_gencg_tildeD}
    For $N\ge 2$ and $\bl \in \N^N,$ let
    \begin{align*}
      \calL_\bl = \bigg\{ \bL = (L_2,L_3,\ldots,L_N)\in \N^{N-1} \, \bigg|
      \quad &
      |l_1-l_2| \le L_2 \le l_1+l_2, \\
      & \forall 3\le i\le N, \;
      |L_{i-1}-l_i| \le L_i \le {L_{i-1}+l_i}
      \bigg\},
      \end{align*}
      then the following statements are true:

      (i) Let $\revii{\bmu,\bbm} \in \calM_\bl$, then
      \begin{equation}
         \revii{D^\bl_{\bmu\bbm}(Q) =
         [\calC_\bl]_{\bmu,(\bL,M_N)} \;
         \diagD^\bl(Q) \;
         [\calC_\bl]^T_{\bbm,(\bL,M_N)},}
         \label{eq:DLmmuQ}
      \end{equation}
      with the generalized Clebsch-Gordan coefficients $\calC_\bl$ and the $\diagD^\bl(Q)$ defined as follows:
      \begin{align}
        \label{eq:Cmatrix}
         [\calC_\bl]_{\bbm,(\bL,M_N)} &=
         C_{l_1m_1l_2m_2}^{L_2M_2}C_{L_2M_2l_3m_3}^{L_3M_3} \ldots
         C_{L_{N-1}M_{N-1}l_Nm_N}^{L_NM_N}, \quad \text{ where}\\
         \notag
         \qquad \bL &= (L_2,\ldots,L_N) \in \calL_\bl,\quad  -L_N \le M_N \le L_N, \quad M_i = \sum_{j=1}^i m_i; \\
         \notag
        \text{and} \qquad \diagD^\bl(Q) &= {\rm diag} \left\{D^{L_N}(Q), \quad \bL = (L_2,L_3,\ldots,L_N) \in \calL_\bl \right\}.
    \end{align}

   (ii) Moreover,
   \begin{equation}
      \label{eq:DLmmuQ_integrated}
      \int_{SO(3)}\diagD^\bl(Q) \; dQ = {\rm diag}\big( \delta_{L_N} \mathbf{1} \big)_{\bL = (L_2,L_3,\ldots,L_N) \in \calL_\bl}.
   \end{equation}
\end{lemma}

\revi{
\begin{example}
For $\bl = (1,1,2),$ 
$
    \calL_\bl = \{ 
        (1,1), (2,0), (2,1), (2,2), (2,3)
    \},
$
and $\diagD^\bl(Q)$ is the block diagonal matrix 
\[
    \diagD^\bl(Q) = 
    \begin{pmatrix}
    D^1(Q) & 0 & 0 & 0 & 0 \\
    0 & D^0(Q) & 0 & 0 & 0\\
    0 & 0 & D^1(Q) & 0 & 0\\
    0 & 0 & 0 & D^2(Q) & 0 \\
    0 & 0 & 0 & 0 & D^3(Q) \\
    \end{pmatrix},
\]  
where the zeros have to be understood as zero matrices of matching size, and the Wigner matrices $D^i(Q)$ are of size $(2i+1)\times (2i+1)$. 
Moreover, 
\[
    \int_{SO(3)}\diagD^\bl(Q) \; dQ = 
    \begin{pmatrix}
    0_{3,3} & 0 & 0 & 0 & 0 \\
    0 & 1 & 0 & 0 & 0\\
    0 & 0 & 0_{3,3} & 0 & 0\\
    0 & 0 & 0 & 0_{5,5} & 0 \\
    0 & 0 & 0 & 0 & 0_{7,7} \\
    \end{pmatrix},
\]  
where the sizes of the block matrices are only made specific on the diagonal. Then one can easily deduce that $\intD^{(1,1,2)}_{\bmu\bbm}$ defined in~\eqref{eq:defn_intD} is of rank one.
\end{example}
}

\begin{remark}
   From the  representation theory of $SO(3)$ (see e.g. \cite{Yutsis1962-yc}) it follows that
   \[
      \displaystyle \prod_{i=1}^N(2l_i+1)
   =  \sum_{L_2 = |l_1-l_2|}^{l_1+l_2} \sum_{L_3 = |L_2-l_3|}^{L_2+l_3}\hspace{-.2cm} \ldots
   \hspace{-.2cm}
   \sum_{L_N = |L_{N-1}-l_N|}^{L_{N+1}+l_N} (2L_N+1),
   \]
   hence the dimensions of $D^\bl(Q)$ and $\diagD^\bl(Q)$ match.
\end{remark}

\begin{proof}
   (i) The statement \eqref{eq:DLmmuQ} directly follows from \eqref{eq:app:recursion_Dmu_full}, which uses the recursion formula on the product of Wigner-D matrices \eqref{eq:app:Dlkm-recursion}.

   (ii)
   Equation \eqref{eq:DLmmuQ_integrated} can be obtained using property \eqref{eq:app:wigner_D_matrix} of the Wigner-D matrices.
\end{proof}

The expression obtained in \eqref{eq:DLmmuQ_integrated} suggests defining the following subset of $\calL_\bl$.
\begin{equation}
   \calL_\bl^0 = \big\{ \bL \in \calL_\bl \,\big| \, L_N = 0
   \big\}.
   \label{eq:ll0def}
\end{equation}
To proceed we recall the definition of \revii{$\intD^\bl_{\bmu\bbm}$} from \eqref{eq:defn_intD}; and moreover define
\begin{align*}
    \calM^0_\bl &:= \big\{ \bbm \in \calM_\bl \,\big|\, {\textstyle \sum} \bbm = 0 \big\},
\end{align*}

We can now expose structures in the generalized Clebsch--Gordan coefficients, defined in~\eqref{eq:Cmatrix}, and the generalized Wigner-D matrices defined in~\eqref{eq:def_Dmatrix}.

\begin{lemma} \label{lem:CG_coef_properties}
   (i)  If $\bL \in \calL_\bl^0$ and $\sum \bbm \neq 0$, then
   \revii{$[\calC_\bl]_{\bbm,(\bL,\sum \bbm)} = 0$}.

   (ii)  For all $\bL,\bL' \in \calL_\bl^0$, we have
   $   \displaystyle\sum_{\bbm \in \calM_\bl}
   [\calC_\bl]_{\bbm,(\bL,0)}^T
   [\calC_\bl]_{\bbm,(\bL',0)} = \delta_{\bL,\bL'}.$

   (iii) For all $\bbm,\bmu\in \calM^0_\bl,$ we have
   $
   \revii{\intD^\bl_{\bmu\bbm}
      = \sum_{\bL \in \calL_\bl^0}
      [\calC_\bl]_{\bmu,(\bL,0)}  [\calC_\bl]^T_{\bbm,(\bL,0)}.}
   $

   (iv) If $\sum \bbm \neq 0$ or  $\sum \bmu \neq 0$, then
   $\revii{\intD^\bl_{\bmu\bbm}} = 0$.
\end{lemma}

\begin{proof}
   (i) For $\bL \in \calL_\bl^0$, $L_N = 0$, hence the last factor in the generalized Clebsch--Gordan coefficients \eqref{eq:Cmatrix} is non-zero only if $M_N = \sum \bbm = 0$.

   (ii) To show this, we recursively use the orthogonality property of Clebsch--Gordan coefficients,
   \[
      \sum_{m_1 m_2} C_{l_1 m_1 l_2 m_2}^{L M} C_{l_1 m_1 l_2 m_2}^{L' M'} = \delta_{L L'} \delta_{M M'}.
   \]
   Hence,
   \begin{align*}
       &  \hspace{-1cm}
       [\calC_\bl]_{.,(\bL,0)}^T
      [\calC_\bl]_{.,(\bL',0)}
      =
      \sum_{\bbm}
      [\calC_\bl]_{\bbm,(\bL,0)}^T
      [\calC_\bl]_{\bbm,(\bL',0)} \\
      &=
      \sum_{m_1,m_2,\ldots,m_N}
      C_{l_1m_1l_2m_2}^{L_2M_2}C_{L_2M_2l_3m_3}^{L_3M_3} \ldots
      C_{L_{N-1}M_{N-1}l_Nm_N}^{L_N0} \\
      & \qquad \qquad \qquad \times
      C_{l_1m_1l_2m_2}^{L_2'M_2}C_{L_2'M_2l_3m_3}^{L_3'M_3} \ldots
      C_{L_{N-1}'M_{N-1}l_Nm_N}^{L_N'0} \\
      &
      =\sum_{m_1,m_2} C_{l_1m_1l_2m_2}^{L_2M_2}
      C_{l_1m_1l_2m_2}^{L_2'M_2}  \\
      & \quad
      \qquad \times\sum_{m_3,\ldots,m_N}
      C_{L_2M_2l_3m_3}^{L_3M_3} \ldots
      C_{L_{N-1}M_{N-1}l_Nm_N}^{L_N0}
      C_{L_2'M_2l_3m_3}^{L_3'M_3} \ldots
      C_{L_{N-1}'M_{N-1}l_Nm_N}^{L_N'0}
      \\
      & = \delta_{L_2,L_2'} \sum_{m_3,\ldots,m_N}
      C_{L_2M_2l_3m_3}^{L_3M_3} \ldots
      C_{L_{N-1}M_{N-1}l_Nm_N}^{L_N0}
      C_{L_2'M_2l_3m_3}^{L_3'M_3} \ldots
      C_{L_{N-1}'M_{N-1}l_Nm_N}^{L_N'0}  \\
      & = \delta_{L_2,L_2'} \delta_{L_3,L_3'}
      \ldots \delta_{L_N,L_N'}.
   \end{align*}

   (iii) From \eqref{eq:DLmmuQ_integrated}, we obtain
   \[
      \revii{\intD^\bl_{\bmu\bbm}} = \int_{SO(3)} \revii{D^\bl_{\bmu\bbm}(Q)} \; dQ
      = \revii{[\calC_\bl]_{\bmu,(\bL,0)} \;
      {\bf 1}_{\calL_\bl^0}  \;
      [\calC_\bl]^T_{\bbm,(\bL,0)}
      = \sum_{\bL \in \calL_\bl^0}
      [\calC_\bl]_{\bmu,(\bL,0)}  [\calC_\bl]^T_{\bbm,(\bL,0)}}.
   \]

   (iv) The result follows immediately from (i) and (iii).
\end{proof}

Ignoring the permutation-invariance for the moment, we have shown that the functions
\[
    b_{\bl\bbm}(\revii{\hat R}) := \sum_{\bmu \in \calM_\bl^0}
    \revii{\intD^\bl_{\bmu\bbm}}
     Y_\bl^\bmu(\revii{\hat R}),
    \quad \bl \in \N^N, \bbm \in \calM_\bl^0
\]
span the space of rotation-invariant functions on $(\bbS^2)^N$. However, they are not linearly independent. Since $b_{\bl\bbm}, b_{\bl'\bbm'}$ for $\bl \neq \bl'$ are obviously independent (they are orthogonal in $L^2(\bbS^N)$), we can focus on each subset $\{b_{\bl\bbm}|\bbm \in \calM^0_\bl\}$, for which the following theorem gives its dimension and an expression of possible orthogonal rotation-invariant basis functions.

\begin{proposition} \label{th:counting_RI_coeffs}
   Let $\bl \in \N^N$, then the following statements are true.

   (i) If $N=1$,
   ${\rm dim}~{\rm span}~\big\{ b_{\bl\bbm} \,|\, \bbm \in \calM_\bl^0 \big\} = 1$ if $\bl = (0)$ and  $0$ otherwise.

   (ii) If $N \geq 2$,
   \begin{equation} \label{eq:symm:rankD_vs_dim}
       {\rm dim}~{\rm span}~\big\{ b_{\bl\bbm} \,|\, \bbm \in \calM_\bl^0 \big\} = \# \calL_\bl^0.
   \end{equation}

   (iii) \revii{${\rm range} \big([\calC_\bl]_{\bmu,(\bL,0)}\big)_{\bmu\in\calM_\bl^0,\bL\in\calL_\bl^0}  = {\rm range} \big( \intD^\bl_{\bmu, \bbm} \big)_{\bmu\in\calM_\bl^0,\bbm\in\calM_\bl^0}$.}
\end{proposition}

\begin{proof}
{\it (i)} For $N=1$, we can directly use the property of Wigner-D matrices \eqref{eq:app:wigner_D_matrix} to obtain the result.

{\it (ii, iii)} Lemma \ref{lem:CG_coef_properties}(iii) provides an diagonalization of $\intD^\bl$, as the generalized Clebsch--Gordan coefficients, $\calC_\bl$, are orthonormal. Hence the dimension of ${\rm span}~\big\{ b_{\bl\bbm} \,|\, \bbm \in \calM_\bl^0 \big\}$ is equal to the rank of $\intD^\bl$, which according to Lemma~\ref{th:lemma_gencg_tildeD} is equal to $\#\{ \calL_\bl^0 \}$; and we easily obtain (iii).
\end{proof}

Finally, we can define orthogonal rotation-invariant basis functions indexed by $\bL\in\calL_\bl^0$ as
\[
   b_{\bl \bL}\revii{(\hat R)} := \sum_{\bmu \in \revii{\calM_\bl^0}}
   [\calC_\bl]_{\bmu, (\bL,0)}
   Y_\bl^\bmu(\revii{\hat R}),
        \qquad \bL\in\calL_\bl^0.
\]

\begin{remark}
    The statements below follow from Proposition~\ref{th:counting_RI_coeffs}:
   \begin{itemize}[wide]  
   \item For $N = 2$, ${\rm dim}~{\rm span}~\big\{ b_{\bl\bbm} \,|\, \bbm \in \calM_\bl^0 \big\} = 1$ if
       $\bl = (l,l), l \in \N$; \\ 
        and ${\rm dim}~{\rm span}~\big\{ b_{\bl\bbm} \,|\, \bbm \in \calM_\bl^0 \big\} = 0$ otherwise.

   \item For $N = 3$, if  $|l_1 - l_2| \le l_3 \le l_1 + l_2,$ and $l_1+l_2+l_3$ is an even number, ${\rm dim}~{\rm span}~\big\{ b_{\bl\bbm} \,|\, \bbm \in \calM_\bl^0 \big\} = 1$, with
   \[
   b_{\bl (l_3,l_3)}\revii{(\hat R)} := \sum_{\bmu \in \revii{\calM_\bl^0}}
   [\calC_\bl]_{\bmu, (l_3,l_3,0)}
   Y_\bl^\bmu(\revii{\hat R}).
\]
Otherwise, ${\rm dim}~{\rm span}~\big\{ b_{\bl\bbm} \,|\, \bbm \in \calM_\bl^0 \big\} = 0$.
    \item    For $N \ge 4$, there are cases where ${\rm dim}~{\rm span}~\big\{ b_{\bl\bbm} \,|\, \bbm \in \calM_\bl^0 \big\} > 1$. For example, for $\bl = (1,1,1,1)$,
   ${\rm dim}~{\rm span}~\big\{ b_{\bl\bbm} \,|\, \bbm \in \calM_\bl^0 \big\} = 3$. We include a table with these numbers for a few other examples in Appendix~\ref{sec:app:dim_rot_inv}.
   \end{itemize}
\end{remark}

\subsection{Combining Rotation and Permutation Invariance}
\label{sec:symm:rot+perm}
It now remains to combine the rotation-invariance with permutation and reflection invariance. Directly applying our construction of the $b_{\bl \bL}$  basis to the permutation symmetric functions \eqref{eq:start_rot} yields basis functions
(we now revert to using $\bbm$ instead of $\bmu$)
\begin{align*}
    \tilde{\calB}_{\bn\bl \bL} := \sum_{\sigma \in S_N} \sum_{\bbm \in \calM^0_\bl}
                    [\calC_\bl]_{\bbm, (\bL,0)}
                    \phi_{\bn\bl\bbm} \circ \sigma, \quad
    &(\bn, \bl) \in \N^{2N} \text{ ordered}, \quad \sum \bl~\text{even},
    \quad {\bm L} \in \calL_\bl^0, \\[-5mm]
\end{align*}
which are permutation, reflection and rotation-invariant by construction.
For simplicity, we index the set $\calL_\bl^0$ from 1 to $\tilde{n}_\bl= \# \calL_\bl^0$ and we denote the basis functions by
\[
    \tilde{\calB}_{\bn\bl i} = \tilde{\calB}_{\bn\bl \bL^{(i)}} , \qquad i = 1, \dots, \tilde{n}_\bl,
\]
where $\bL^{(i)}$ is the $i^{\rm th}$ term in an enumeration of the elements  $\calL_\bl^0$.

Alternatively, and equivalently,  we may use the numerically constructed basis from  \S~\ref{sec:symm:rot}, which would lead to
\begin{align*}
    \tilde{\calB}_{\bn\bl i} := \sum_{\sigma \in S_N} \sum_{\bbm \in \calM^0_\bl}
                    \revii{\tilde\calU^\bl_{\bbm i}} \phi_{\bn\bl\bbm} \circ \sigma, \qquad
    & i = 1, \dots, \tilde{n}_\bl,
\end{align*}
where \revii{$\tilde\calU^\bl_{\bbm i}$} is constructed according to Lemma~\ref{th:symm:orth_RI_basis}.

While the set of all $\tilde{\calB}_{\bn\bl i}$ is a spanning set by construction, it turns out that, except for $N \leq 3$ (cf. Proposition~\ref{th:counting_RI_coeffs}), they are not linearly independent. Indeed, one can observe that the $\intD^\bl$ coefficients have certain symmetries, and after symmetrising the rotation invariant basis with respect to $S_N$, these symmetries give rise to additional linear dependence within a block of basis functions $\{ \tilde{\calB}_{\bn\bl i} | i = 1, \dots, \tilde{n}_\bl\}$.

To overcome this, one could study the symmetries of the generalised Clebsch--Gordan coefficients with respect to permutation of the indices based, e.g., on~\cite{Yutsis1962-yc}. One could also work on combining the representations of the rotations and the permutations following~\cite{Schmiedt2016-nw}.
However, so far, there does not seem to be explicit and unified formulas for all $N$, nor does it seem straightforward to obtain them.

For now, we proceed by a semi-numerical construction: We begin by algebraically evaluating the Gramian
\[
    G_{i,i'}^{\bn\bl} := \big\<\!\!\big\< \tilde{\calB}_{\bn\bl i}, \tilde{\calB}_{\bn\bl i'} \big\>\!\!\big\>,
\]
with respect to the abstract inner product,
\begin{equation} \label{eq:symm:abstract_ip}
    \<\!\< \phi_{\bn\bl\bbm}, \phi_{\bn'\bl'\bbm'} \>\!\> :=
    \delta_{\bn\bn'} \delta_{\bl\bl'} \delta_{\bbm\bbm'}.
\end{equation}
It is obvious that, if $\calP$ is linearly independent, then $\<\!\< \cdot, \cdot \>\!\>$ is an inner product on $\Phi_N$ (cf. \S~\ref{sec:pes:approx}). Moreover, we will show in \S~\ref{sec:radial} how to construct radial bases that are indeed {\em orthogonal} with respect to natural inner products.

At low and moderate interaction orders, $G^{\bn\bl}$ can be evaluated fairly \revii{quickly}.
\revii{Indeed, its computation only involves sums over permutations satisfying $(\bn,\bl,\bbm) = (\sigma \bn,\sigma \bl, \sigma \bbm')$ with $\bbm,\bbm' \in  \calM^0_\bl$, which stays computationally feasible at least up to interaction order $N = 10$.}
After diagonalising $G^{\bn\bl} = V \Sigma V^T$ we can then define a new set of coefficients
\begin{equation} \label{eq:defn_calU_coeffs}
    \revii{\calU^{\bn\bl}_{\bbm i}}
    :=
    \revii{\frac{1}{\sqrt{\Sigma_{ii}}}}\sum_{\alpha = 1}^{\tilde{n}_{\bl}} 
     [V_{\alpha i}]^* \revii{\tilde\calU^\bl_{\bbm\alpha}}, \qquad
    i = 1, \dots, n_{\bn\bl},
\end{equation}
where $n_{\bn\bl} = {\rm rank}(G^{\bn\bl})$ and $\revii{\tilde\calU^\bl_{\bbm\alpha}}, \tilde{n}_\bl$ are defined in Lemma~\ref{th:symm:orth_RI_basis}.
With this definition we obtain
\begin{equation} \label{eq:symm:defn_calB_Nparticles}
    \calB_{\bn\bl i} := \sum_{\bbm \in \revii{\calM_\bl^0}} \revii{\calU^{\bn\bl}_{\bbm i}} \sum_{\sigma \in S_N} \phi_{\bn \bl \bbm} \circ \sigma,
\end{equation}
which we collect into the basis
\begin{equation} \label{eq:firstbasis}
    {\bm \calB}_N := \big\{ \calB_{\bn\bl i} \,\big|\,
        (\bn, \bl) \in \N^{2N} \text{ ordered}, {\textstyle \sum} \bl \text{ even};
            i = 1, \dots, n_{\bn\bl} \big\}.
\end{equation}
This defines our first symmetric basis set.

\begin{theorem} \label{th:symm:approx_calB_basis}
    Fix $N \geq 1$ and let $\calP$ be a radial basis satisfying Assumption~\ref{as:P} with $\rcut > \ro > 0$, then ${\bm \calB}_N \subset C^{t}_{\rcut}(\Omega_\ro^N)$ is linearly independent and
    \[
        C^{t}_{\rcut}(\Omega_\ro^N) \subset \clospan{{\bm \calB}_N}{C^t}
    \]
    Moreover, ${\bm \calB}_N$ are orthonormal with respect to the inner product \eqref{eq:symm:abstract_ip}.
\end{theorem}
\begin{proof}
    This result is an immediate consequence of Proposition~\ref{th:approx_nonsym} and the construction of ${\bm \calB}_N$.
\end{proof}

Although we cannot (at present) explicitly construct the rotation and permutation-invariant basis, we may still ask whether it is possible to {\em predict its size}. That is, given $\bn, \bl$ tuples, we need to predict the rank of $G^{\bn\bl}$. Even this appears to be difficult in general, with tedious explicit calculations for specific cases (though easily performed numerically).

\revi{
Nevertheless, note that the dependency of $G^{\bn\bl}$ with respect to $\bn$ only appears through a sum over the permutations of size $N$ leaving $\bn$ and $\bl$ invariant, but where the elements in the sum do not depend on $\bn$.
Indeed, a straightforward calculation leads to
\[
     G_{i,i'}^{\bn\bl} = N! \sum_{\substack{\sigma \in S_N \\ \sigma(\bl) = \bl \\  \sigma(\bn) = \bn}}
      \sum_{\bbm,\bbm' \in \calM^0_\bl}
                    \tilde\calU^\bl_{\sigma(\bbm) i } \tilde\calU^\bl_{ \bbm' i'}.
\]
}
Hence, for each fixed $\bl$, the calculation of the number of rotation and permutation-invariant (RPI) basis functions needs only to be done for a finite number of indices $\bn$'s. We summarize these for the lowest degrees in  Table~\ref{tbl:more_dimensions}.

A preliminary result estimating the number of rotation and permutation-invariant (RPI) basis functions versus the number of rotation-invariant (RI) basis functions, for a single $(\bn, \bl)$ block, is the following:
\begin{proposition}
    \label{th:ri_vs_rpi_distinct}
    If $(\bn,\bl) \in \N^{2N}$ is such that all $(n_i,l_i)$ are pairwise distinct, then the number of permutation-invariant and permutation- and rotation-invariant basis functions match, i.e.,
    $\tilde{n}_\bl = n_{\bn\bl}$, with $\tilde{n}_\bl = \# \calL_\bl^0$  and $n_{\bn\bl}$ defined in~\eqref{eq:defn_calU_coeffs}.
\end{proposition}
\begin{proof}
   Let us first define the space of functions
   \[
        T_{\bn\bl} = {\rm span} \big\{ \phi_{\sigma(\bn)\sigma(\bl)\sigma(\bbm)} | \; \sigma \in S_N, \bbm \in \calM_\bl^0 \big\}.
   \]
   This space is closed under permutations and rotations of the variables.
   Moreover, if all $(n_i,l_i)$ are pairwise distinct, all $\phi_{\sigma(\bn)\sigma(\bl)\sigma(\bbm)}$ are linearly independent.
   This space is also generated by the following linearly independent functions, separated into symmetric and (partially) anti-symmetric functions.
   \begin{align*}
               T_{\bn\bl} = {\rm span} \Big\{ & \sum_{\sigma\in S_N} \phi_{\sigma(\bn)\sigma(\bl)\sigma(\bbm)}| \; \bbm \in \calM_\bl^0, \\ &
        \sum_{\sigma\in S_N\backslash \sigma_i} \phi_{\sigma(\bn)\sigma(\bl)\sigma(\bbm)}
        - (N!-1) \phi_{\sigma_i(\bn)\sigma_i(\bl)\sigma_i(\bbm)}
        | \; \sigma_i \in S_N\backslash (1) , \; \bbm \in \calM_\bl^0 \Big\}.
   \end{align*}
   Then, one can show that the action of rotation and permutation is stable over the symmetric and anti-symmetric functions. More precisely, for the symmetric functions, there holds for all $Q \in SO(3)$, $s \in S_N$,
   \[
   \sum_{\sigma\in S_N} \phi_{\sigma(\bn)\sigma(\bl)\sigma(\bbm)}(Q R_s)
    =
    \sum_{\bmu \in \calM_\bl^0}
    \revii{D^\bl_{\bmu\bbm}(Q)} \Big[ \sum_{\sigma\in S_N} \phi_{\sigma(\bn)\sigma(\bl)\sigma(\bbm)}(R) \Big]
   .
   \]
   For the antisymmetric functions, for $Q \in SO(3)$ and $s \in S_N$,
   \begin{align*}
      &  \sum_{\sigma\in S_N\backslash \sigma_i} \phi_{\sigma(\bn)\sigma(\bl)\sigma(\bbm)}(Q R_s)
       - (N!-1) \phi_{\sigma_i(\bn)\sigma_i(\bl)\sigma_i(\bbm)}(Q R_s) \\
    &=  \sum_{\sigma\in S_N\backslash \sigma_i} \phi_{s^{-1}\sigma(\bn)s^{-1}\sigma(\bl)s^{-1}\sigma(\bbm)}(Q R)
       - (N!-1) \phi_{s^{-1}\sigma_i(\bn)s^{-1}\sigma_i(\bl)s^{-1}\sigma_i(\bbm)}(Q R)  \\
       & = \sum_{\sigma\in S_N\backslash s^{-1}\sigma_i} \phi_{\sigma(\bn)\sigma(\bl)\sigma(\bbm)}(Q R)
       - (N!-1) \phi_{s^{-1}\sigma_i(\bn)s^{-1}\sigma_i(\bl)s^{-1}\sigma_i(\bbm)}(Q R) \\
       & =
       \sum_{\bmu \in \calM_\bl^0}
    \revii{D^\bl_{\bmu\bbm}(Q)} \left[
    \sum_{\sigma\in S_N\backslash s^{-1}\sigma_i} \phi_{\sigma(\bn)\sigma(\bl)\sigma(\bbm)}(R)
       - (N!-1) \phi_{s^{-1}\sigma_i(\bn)s^{-1}\sigma_i(\bl)s^{-1}\sigma_i(\bbm)}(R)
    \right].
   \end{align*}
    Therefore, the antisymmetric functions are stable under permutation and rotation.

    The permutation-invariant functions can be generated from summing these functions over permutations.
    Doing this summation, we obtain that the antisymmetric functions integrate to zero, whereas the relation over the symmetric functions does not change when summing over the permutations. Hence, we obtain
    \[
        T_{\bn\bl}^{PI} := {\rm span} \big\{ \sum_{\sigma \in S_N} \phi(R_\sigma) \; | \; \phi \in T_{\bn\bl} \big\}
        = {\rm span} \big\{
         \sum_{\sigma\in S_N} \phi_{\sigma(\bn)\sigma(\bl)\sigma(\bbm)}| \; \bbm \in \calM_\bl^0
        \big\}.
    \]
    The transformation under rotation of the functions in $T_{\bn\bl}^{PI}$
    being similar to the rotational case, we can then conclude that $ \text{\rm dim span} \big\{ \phi_{\bn\bl\bbm} | \; \bbm \in \calM_\bl^0 \big\} =  \text{\rm dim span} \big\{ b_{\bl\bbm} | \; \bbm \in \calM_\bl^0 \big\}$.
\end{proof}

\begin{remark}
    Asymptotically, as the polynomial degree tends to infinity, most $\bn$ tuples are strictly ordered and therefore most $(\bn,\bl)$ tuple pairs satisfy the condition of Proposition~\ref{th:ri_vs_rpi_distinct}. This implies that in this limit the number of RI and RPI basis functions is comparable. That is, we could estimate the size of the basis $\#{\bm \calB}_N$, restricted to some polynomial degree, provided we have an estimate of $\#\calL_\bL^0$ for all $\bL$. An expression for $\#\calL_\bL^0$ can indeed be obtained from the foregoing analysis but is difficult to analyze, and beyond the scope of the present paper.
\end{remark}

\def\lbrak{(}
\def\rbrak{)}

\begin{table}
    \begin{center} \small
    \begin{tabular}{cccc}
        \toprule
        $ \bl $ & \,\#RI\, & $ \bn $ & \#RPI \\
        \midrule[0.075em]
        \multirow{3}{*}{\lbrak 2, 2, 2, 2\rbrak}
                         & \multirow{3}{*}{5}  &  \lbrak 0, 0, 0, 1\rbrak &  1 \\
          &  &  \lbrak 0, 0, 0, 2\rbrak &  1 \\
          &  &  \lbrak 0, 0, 1, 2\rbrak &  3 \\
        \midrule[0.02em]
        \multirow{1}{*}{\lbrak 2, 2, 2, 4\rbrak}
                         & \multirow{1}{*}{3}  &  \lbrak 0, 0, 0, 0\rbrak &  1 \\
        \midrule[0.02em]
        \multirow{1}{*}{\lbrak 2, 2, 3, 3\rbrak}
                         & \multirow{1}{*}{5}  &  \lbrak 0, 0, 0, 0\rbrak &  3 \\
        \bottomrule
    \end{tabular}
    \quad
    \begin{tabular}{cccc}
        \toprule
        $ \bl $ & \,\#RI\, & $ \bn $ & \#RPI \\
        \midrule[0.075em]
        \multirow{3}{*}{\lbrak 1, 1, 2, 2, 2\rbrak}
                         & \multirow{3}{*}{9}  &  \lbrak 0, 0, 0, 1, 2\rbrak &  6 \\
          &  &  \lbrak 0, 1, 0, 0, 0\rbrak &  2 \\
          &  &  \lbrak 1, 1, 0, 0, 1\rbrak &  4 \\
        \midrule[0.02em]
        \multirow{1}{*}{\lbrak 1, 2, 2, 2, 3\rbrak}
                         & \multirow{1}{*}{12}  &  \lbrak 0, 0, 0, 0, 0\rbrak &  3 \\
        \midrule[0.02em]
        \multirow{1}{*}{\lbrak 2, 2, 2, 2, 2\rbrak}
                         & \multirow{1}{*}{16}  &  \lbrak 0, 0, 0, 0, 0\rbrak &  1 \\
        \bottomrule
    \end{tabular}
    \end{center}
    \medskip
    \caption{
        \label{tab:RIvsRPI_small}
    Representative examples of reduction in basis size by combining rotation with permutation-invariance; cf. \S~\ref{sec:symm:rot+perm}. For each $\bl \in \N^N$, $N = 4, 5$, the column `$\#{\rm RI}$' displays the number $\tilde{n}_{\bl}$ of rotation-invariant (RI) basis functions, i.e.,
    ${\rm dim~span}\{ b_{\bl\bbm} \,|\, \bbm \in \calM_\bl^0\}$.
    For each $(\bl, \bn)$ the column `$\#{\rm RPI}$' displays the number $n_{\bn\bl}$ of corresponding rotation- and permutation-invariant (RPI)    basis functions; cf. \eqref{eq:defn_calU_coeffs}.
    }
\end{table}

\begin{remark}
   When $\bn,\bl$ is such that all $(n_i,l_i)$ are {\em not} pairwise distinct, then the claim that the basis functions spanning $T_{\bn\bl}$ are linearly independent is not satisfied anymore, as was very recently shown in the slightly different context~\cite{Nigam2020-zj} of covariant descriptors. In particular in the preasymptotic regime of low polynomial degrees and high interaction order $N \geq 4$ this situation is typical and leads to a striking reduction in basis size, as Table~\ref{tab:RIvsRPI_small} shows.
   A more comprehensive table of basis function numbers is given in Appendix~\ref{sec:app:dim_rot_inv}. 
\end{remark}

\section{Efficient Evaluation}
\label{sec:proj}  
We already  commented at the beginning of \S~\ref{sec:symm} that the basis ${\bm \calB}_N$ we constructed throughout that section is not computationally efficient due to the $N!$ terms arising in the summation over all permutations.
A second limitation of the ${\bm \calB}_N$ basis is that we usually wish to evaluate the sum over all $N$-neighbour clusters (cf. \S~\ref{sec:pes})
\begin{equation} \label{eq:proj:extension_calB}
    \calB_{\bn\bl i}\big(\{\br_j\}_{j = 1}^J\big)
    :=
    \sum_{1 \leq j_1 < \dots < j_N \leq J}
    \calB_{\bn\bl i}(\br_{j_1}, \dots, \br_{j_N}),
\end{equation}
which brings an additional $\binom{J}{N}$ cost. For the remainder of this paper
we take \eqref{eq:proj:extension_calB} to be the {\em definition} of
$\calB_{\bn\bl i}$ when applied to an atomic neighbourhood
$\{\br_j\}_{j = 1}^J$. This is a consistent extension: When $N = J$ this
definition coincides with our previous one.

The purpose of the present section is to derive an alternative basis with cost
that scales linearly with $N$. The main ideas that we use here can in various
formats be found in~\cite{Behler2007-ng,Bartok2010-mv,Shapeev2016-pd} and in particular \cite{Drautz2019-er}.
We nevertheless give a full derivation, for the sake of
completeness, but also because we will indicate in \S~\ref{sec:proj:orth} how to modify this construction to efficiently compute the original orthogonal basis functions $\calB_{\bn\bl i}$.

In preparation, and arguing purely formally for now, an atomic neighbourhood $R = \{\br_j\}_{j = 1}^J$ can be represented as its {\em density},
\[
    \rho_R(\br) := \sum_{j = 1}^J \delta(\br - \br_j),
\]
which we can project onto the 1-particle basis,
\begin{equation} \label{eq:proj:defn_Anlm}
    A_{nlm}(R) := \< \phi_{nlm}, \rho_R \>
    =
    \sum_{j = 1}^J \phi_{nlm}(\br_j),
    \qquad
    n, l \in \N, m \in \{-l, \dots, l\}.
\end{equation}
This means that we can think of a site energy $V$ as a function defined on configuration space $\calR$, or on the space of measures $\{\rho\}$, or on the space of descriptors $\{ (A_{nlm}) \}$. But note that none of these incorporate the isometry invariance. 

If we symmetrize the density projection by averaging over ${\rm O}(3)$ then we lose all angular information and retain only the radial basis which is clearly insufficient to describe a general atom environment. The typical strategy to overcome this is to consider correlations: for $\bn = (n_\alpha)_{\alpha=1}^N, \bl = (l_\alpha)_{\alpha=1}^N, \bbm = (m_\alpha)_{\alpha=1}^N$ we define
\begin{equation} \label{eq:defn_correlations}
    A_{\bn\bl\bbm}(R) := \big\< \otimes_{\alpha = 1}^N \phi_{n_\alpha l_\alpha m_\alpha}, \otimes_{\alpha = 1}^\revii{N} \rho_R \big\>
    =  \prod_{\alpha = 1}^N A_{n_\alpha l_\alpha m_\alpha}(R).
\end{equation}
which can then be symmetry-adapted by averaging over ${\rm O}(3)$. E.g., with $N = 2$ this leads to SOAP~\cite{Bartok2010-mv}; see Appendix~\ref{app:otherbases} for more details. In the following we explore how these correlations are related to the symmetric polynomial bases ${\bm \calB}_N$ derived in the previous section.

\subsection{Symmetry adapted correlations}

\label{sec:ACEBasis}
An alternative way to write \eqref{eq:proj:extension_calB} is
\[
    \calB_{\bn\bl i}\big(\{\br_j\}_{j = 1}^J\big)
    =
    \frac{1}{N!}
    \sum_{j_1 \neq \dots \neq j_N}
    \calB_{\bn\bl i}(\br_{j_1}, \dots, \br_{j_N}),
\]
where $\sum_{j_1 \neq \dots \neq j_N}$ means summation over all $N$-tuples $(j_1, \dots, j_N) \in \{1, \dots, J\}^N$ for which $j_\alpha \neq j_{\alpha'}$ unless $\alpha = \alpha'$. A key observation~\cite{Drautz2019-er} was that we can also write
\begin{equation} \label{eq:proj:drautz_trick}
    \calB_{\bn\bl i}\big(\{\br_j\}_{j = 1}^J\big)
    =
    \frac{1}{N!}
    \sum_{j_1, \dots, j_N}
    \calB_{\bn\bl i}(\br_{j_1}, \dots, \br_{j_N})
    +
    W_{N-1}(\{\br_j\}_{j=1}^J),
\end{equation}
where $W_{N-1}$ is a polynomial with interaction-order $N-1$, i.e., it can be written as a sum over terms each of which depends on at most $N-1$ neighbours. The polynomial $W_{N-1}$ may be interpreted as balancing the unphysical self-interaction contributions of atoms that have been introduced by the modified summation. Its precise form is unimportant for now, as we simply drop it from the order-$N$ basis and absorb it into the order-$(N-1)$ basis.

In this way, we obtain a new symmetric  basis function
\begin{equation} \label{eq:eff:drautz_B_def1}
    B_{\bn\bl i}\big(\{\br_j\}_{j = 1}^J\big)
    :=
    \frac{1}{N!}
    \sum_{j_1, \dots, j_N}
    \calB_{\bn\bl i}(\br_{j_1}, \dots, \br_{j_N}),
\end{equation}
which we now manipulate to relate it to the density projections $A_{nlm}$ defined in \eqref{eq:proj:defn_Anlm} and the correlations \eqref{eq:defn_correlations}. Inserting the definition of $\calB_{\bn\bl i}(\br_{j_1}, \dots, \br_{j_N})$ from \eqref{eq:symm:defn_calB_Nparticles} we obtain
\begin{align*}
    B_{\bn\bl i}(\{\br_j\}_{j = 1}^J)
    &=
    \frac{1}{N!}
    \sum_{j_1, \dots, j_N}
    \sum_{\bbm \in \calM_\bl} 
    \revii{\calU^{\bn\bl}_{\bbm i}}
    \sum_{\sigma \in S_N} \phi_{\bn \bl \bbm}\big(\br_{j_{\sigma 1}},
        \dots, \br_{j_{\sigma N}} \big)
    \\
    &=
    \sum_{\bbm \in \calM_\bl} \revii{\calU^{\bn\bl}_{\bbm i}}
    \sum_{j_1, \dots, j_N}
    \phi_{\bn \bl \bbm}\big(\br_{j_1},
        \dots, \br_{j_N} \big)
    \\
    &=
    \sum_{\bbm \in \calM_\bl} \revii{\calU^{\bn\bl}_{\bbm i}}
    \sum_{j_1, \dots, j_N = 1}^J
    \prod_{\alpha = 1}^N
    \phi_{n_\alpha l_\alpha m_\alpha}\big(\br_{j_\alpha}\big)
    \\
    &=
    \sum_{\bbm \in \calM_\bl} \revii{\calU^{\bn\bl}_{\bbm i}}
    \prod_{\alpha = 1}^N
    \sum_{j = 1}^J
    \phi_{n_\alpha l_\alpha m_\alpha}\big(\br_{j}\big).
\end{align*}
Thus, inserting the definition \eqref{eq:defn_correlations} we obtain the alternative expression
\begin{equation} \label{eq:proj:B_kli_efficient}
    B_{\bn\bl i}(\{\br_j\}_{j = 1}^J)
    :=
    \sum_{\bbm \in \calM_\bl} 
    \revii{\calU^{\bn\bl}_{\bbm i}}
    A_{\bn\bl\bbm}(\{\br_j\}_{j = 1}^J),
\end{equation}
which avoids both the $N!$ cost for symmetrising the basis as well as the
$\binom{J}{N}$ cost of summation over all order $N$ clusters within an atomic
neighbourhood. 

We denote the resulting basis by
\begin{equation}
    \label{eq:B_basis}
       {\bm B}_N := \big\{
        B_{\bn\bl i} \,\big|\,
        (\bn, \bl) \in \N^{2N} \text{ ordered},
        {\textstyle \sum} \bl \text{ even},
        i = 1, \dots, n_{\bn\bl} \big\},
\end{equation}
and obtain the following result, which can loosely also be stated as follows: ``the symmetry-adapted $N$-correlations, $N \in \N,$ form a complete basis of symmetric polynomials.''

\begin{theorem} \label{th:density_acetype_basis}
    Fix $N \geq 1$ and let $\calP$ be a radial basis satisfying Assumption~\ref{as:P}, then $\bigcup_{n = 1}^{N} {\bm B}_n \subset C^t_{\rcut}(\Omega_\ro^N)$ is linearly independent and
    \[
        C^t_{\rcut}(\Omega_\ro^N) \subset \clospan{\bigcup_{n = 1}^{N} {\bm B}_n}{C^t}.
    \]
\end{theorem}
\begin{proof}
    This result is a direct corollary of Theorem~\ref{th:symm:approx_calB_basis} and of \eqref{eq:proj:drautz_trick}.
\end{proof}

In light of the results of this section we make the following formal definition. 

\begin{definition}
    An {\em atomic cluster expansion} (ACE) is an expansion of an atomic property in terms of the basis $\bigcup_{N = 1}^\infty {\bm B}_N$, of symmetry-adapted $N$-correlations~\cite{Drautz2019-er}. More generally we will use the label ACE to describe approximations and expansions building on this basis.
\end{definition}

\subsection{Basis evaluation cost}
A concrete finite basis is specified by a finite set of $(\bn, \bl)$ tuples
$\Spec \subset \bigcup_{N = 1}^\infty \N^N \times \N^N$; for more details see \S~\ref{sec:implementation:basisconstruction}. This results in the basis
\[
    {\bm B} := \big\{ B_{\bn\bl i} \b|  
        (\bn, \bl) \in \Spec, i = 1, \dots, n_{\bn\bl} \b\}.
\]
The basis functions $B_{\bn\bl i}$ are then computed from the correlations $A_{\bn\bl\bbm}$ via a sparse matrix-vector operation, which is fast and efficient. We therefore focus on the cost of evaluating the purely permutation-invariant basis. 

Given a RPI basis specification $\Spec$, we can define the resulting permutation-invariant basis
\[
    {\bm A} := 
    \b\{
        A_{\bn\bl\bbm} \,:\, 
        (\bn, \bl) \in \Spec, \quad 
        |m_\alpha| \leq l_\alpha, \quad 
        (\bn, \bl, \bbm) \text{ ordered}
    \b\}.
\]
The size of ${\bm A}$ can be readily estimated from the choice of $\Spec$, but to simplify the presentation we provide cost estimates directly in terms of ${\bm A}$. 

To evaluate ${\bm A}$ we must first evaluate the corresponding one-particle basis,  
\[
    {\bm A}_1 := \big\{ A_{nlm} \,:\, 
        \exists (\bn, \bl) \in \Spec, \alpha \in \N \text{ s.t. }
        (n, l) = (n_\alpha, l_\alpha), 
        |m_\alpha| \leq l_\alpha \big\};
\]
that is, all $A_{nlm}$ occuring in the evaluation of ${\bm A}$.
Through recursive evaluation of the radial as well as spherical harmonics basis the cost will be approximately 
\[
  {\rm COST}({\bm A}_1) \approx \big[ \max n \times (\max l)^2 \big] \times N_{\rm at},
\]
where $N_{\rm at}$ is the number of atoms in the environment that is evaluated.
This cost scales as any three-dimensional tensor approximation. A significant bottleneck is that $N_{\rm at}$ scales cubically with the cutoff radius, which motivates ideas such as the multiscale basis suggested in \S~\ref{sec:radial:remarks}(7).

Once the $\{A_{nlm}\}$ have been evaluated the basis ${\bm A}$ can be obtained simply by multiplications, resulting in 
\[
    {\rm COST}({\bm A})  \approx 
    {\rm COST}({\bm A}_1)
    + 
    \sum_{(\bn, \bl) \in \Spec} {\rm len}(\bn),
\]
where ${\rm len}(\bn) := N$ for $\bn \in \N^N$. 
This additional cost scales linearly in the number of basis functions and linearly in the body-order, which shows that the basis functions  
\eqref{eq:proj:B_kli_efficient} already give a highly efficient representation of the space of symmetric polynomials.

\subsection{Recursive representation: towards optimal evaluation cost}
\label{sec:graph_representation}
\def\calK{\mathcal{K}}
Although the natural basis evaluation scheme outlined in the previous section is already fairly efficient, it can be improved further. \revii{Related constructions to the one we propose here are used in \cite{Shapeev2016-pd} for the efficient evaluation of MTPs and in~\cite{Nigam2020-zj} for the recursive construction of {\em covariant} correlations.}

To explain this evaluation scheme, it is convenient to enumerate the one-particle basis by \revii{${\bm A}_1 := \{ A_{n_k l_k m_k} : k = 1, \dots, K \}$}, then we can write 
\[
    A_{\bn\bl\bbm} = \prod_{\alpha = 1}^N A_{n_\alpha l_\alpha m_\alpha} 
        = \prod_{\alpha = 1}^N A_{k_\alpha} = A_{\bk},   
\]
where we identify $A_{k} = A_{n_k l_k m_k}$ and $\bk = (\bn, \bl, \bbm)$. With this notation we can write ${\bm A} = \{ A_{\bk} : \bk \in \calK \}$ where $\calK \subset \bigcup_{N=1}^\infty \N^N$.

Let $\bk \in \N^{N_1+N_2}, \bk^{(1)} \in \N^{N_1}, \bk^{(2)} \in \N^{N_2}$ such that $A_{\bk} = A_{\bk^{(1)}} A_{\bk^{(2)}}$, then we say that $\bk \equiv \bk^{(1)} \cup \bk^{(2)}$. If all $\bk \in \calK$ have such a decomposition, then the basis ${\bm A}$ can be stored in terms of a directed acyclic graph where each node $\bk \in \calK$ represents a basis function $A_{\bk}$, ${\rm len}(\bk) > 1$ with exactly two incoming edges $(\bk^{(1)}, \bk)$ and $(\bk^{(2)}, \bk)$. In this case the cost of evaluating the basis would be optimal, requiring only one arithmetic operation for evaluating each basis function.

In \revii{practice}, due to the condition $\sum_\alpha m_\alpha = 0$ not all $A_\bk, \bk \in \calK$ have such a decomposition. For example, if $\bl = (1,1,1,1)$, $\bbm = (1, 1, -1, -1)$ can be written as a simple product, 
\[
    A_{\bk} = A_{(k_1, k_3)} A_{(k_2, k_4)} 
\]
\revii{where $(k_1, k_3) = (k_2, k_4)$ both stand for $\bl^{(j)} = (1,1), \bbm^{(j)} = (1,-1)$. In this case one arithmetic operation, instead of three are required. The performance improvement becomes increasingly significant with increasing interaction order.

On the other hand, basis functions with $\bl = (2,2, 1, 1), \bbm = (1,1,-1,-1)$ or with $\bl = (3,3,3,3), \bbm = (3,-1,-1,-1)$ cannot be split in this way, since any possible splitting either violates either the $\sum l_t = {\rm even}$ or the $\sum m_t = 0$ condition. 
In order to use the recursive evaluation idea for these cases, we must add {\em auxiliary} basis functions to ${\bm A}$, resulting in an extended permutation-invariant basis set 
\[
    {\bm A}^{\rm ext} \supset {\bm A}.
\]
It is easy to see that such an extended set can always be constructed. For example, for $\bl = (2,2, 1, 1), \bbm = (1,1,-1,-1)$ we would add the basis functions 
\[
   A_{(n_1, n_3), (2, 1), (1, -1)} \qquad \text{and} \qquad 
   A_{(n_2, n_4), (2, 1), (1, -1)}
\]
whose product will produce 
\begin{align*}
    A_{(n_1,n_2,n_3,n_4), (2,2,1,1), (1,1,-1-1)}
    &= 
    A_{n_1, 2, 1} A_{n_2, 2, 1} A_{n_3, 1, -1} A_{n_4, 1, -1} \\ 
    &= A_{(n_1, n_3), (2, 1), (1, -1)} A_{(n_2, n_4), (2, 1), (1, -1)}
\end{align*}
The computational cost of evaluating ${\bm A}$ in this way requires exactly $\# {\bm A}^{\rm ext}$ multiplications. The choice of extended basis set ${\bm A}^{\rm ext}$ is non-unique. Optimising this representation, to potentially  obtain an {\em optimal} evaluation complexity will be explored in a separate work. In \S~\ref{sec:implementation} we present a simple first heuristic and promising benchmark results.

The original RPI basis ${\bm B}$ can still be evaluated without changes (algorithmically, the sparse matrix multiplication ${\bm A}^{\rm ext} \to {\bm B}$ will simply contain additional empty rows).  }

\subsection{Extension to multi-component systems}
\label{sec:app:multi}
Multi-component systems are describ\-ed by position, atomic number pairs $(\br_j, z_j)$. The permutation invariance is now only with respect to like atoms, i.e., atoms with the same atomic number. Symmetric polynomial basis functions are now obtained by
\begin{align*}
    B^{(\zeta)}_{\bz\bn\bl i }(\{(\br_j, z_j)\})
    &
    :=
    \sum_{\bbm} \calU^{\bn\bl}_{i \bbm} A_{\bz\bn\bl\bbm}(\{\br_j\}, \{z_j\}), \\
    A_{\bz\bn\bl\bbm}
    &:= \prod_{\alpha = 1}^N
        A_{z_\alpha n_\alpha l_\alpha m_\alpha} \\
    A_{z n l m}
    &:=
    \sum_{j} \delta(z_j - z) \phi_{nlm}(\br_j).
\end{align*}
A natural way to interpret (and extend; see \S~\ref{sec:composition}) this expression is to think of $A_{znlm}$ as projecting a density in the $(z, \br)$ space to the basis $\phi_{\zeta nlm}(\br, z) := \delta(z - \zeta) \phi_{nlm}(\br)$. 

The $\calU^{\bn\bl}_{i \bbm}$ coupling coefficients are the same as in the single-species case, i.e., they are entirely independent of the species involved in the basis definition. The superscript $(\zeta)$ indicates that this is a basis function for a centre-atom of species $\zeta$.

The additional $z$-dependence for multi-component systems increases the dimensionality of the configuration space and can significantly reduce the effectiveness of the density projection. The practical consequences of this efficiency deterioration, and means to control it for example via sparsification or compositional structures, is yet to be explored in detail.

\section{Polynomial Radial Bases}
\label{sec:radial}
We have left the specification of the radial basis $\calP$ to the very end, in order to emphasize that very little needs to be assumed about it. Indeed, the choice of radial basis leaves significant freedom for optimisation.
Drautz~\cite{Drautz2019-er} proposes a concrete construction, which appears to work well in \revii{practice}, but leaves room for exploration and improvements.

\subsection{Orthogonal polynomials}
\label{sec:orth_radial}
We propose the following general procedure to obtain radial bases $\calP$, satisfying Assumption~\ref{as:P}, which are orthogonal with respect to a {\em user-defined} inner product, and further increase the design space for the radial basis:
\begin{enumerate}
    \item \revi{Choose $\rcut, \ro \geq 0$} specifying the domain $[\ro, \rcut]$ on which we require orthogonality;
    \item \revi{Choose} a smooth and smoothly invertible coordinate transformation $\xi : [r_0, \rcut] \to [a, b]$
    \item \revi{Choose} a cut-off function $\fcut \in C^t([0, \infty))$ such that $\fcut > 0$ in $[\ro, \rcut)$ and which vanishes in $[\rcut, \infty)$;
    \item \revi{Choose} an orthogonality measure $\rho$ on $[\ro, \rcut]$.
\end{enumerate}
We will write $x = \xi(r), r = \xi^{-1}(x)$. The measure $\rho$ gives rise to a measure $\rho_x$ on $[a, b]$ via $\rho_x = |\xi'| \rho$. We then specify a second measure
\[
    \tilde\rho_x(dx) := (\fcut \circ \xi^{-1})(x)^2 \rho_x(dx).
\]
There exists a unique sequence of polynomials $J_n(x), n \in \N$ such that (cf. \cite[Sec. 11.4]{Powell1981-bg})
\[
    J_0(x) = c_0, \quad J_1(x) = c_1 x, \quad
    \int_{a}^b J_{n}(x) J_{n'}(x) \tilde{\rho}_x(dx) = \delta_{nn'},
\]
where $c_0, c_1$ are normalisation factors, and moreover, the $J_n$ can be evaluated by a \revii{recurrence} relation
\begin{equation} \label{eq:rad:3ptrec}
    J_{n} = (s_{n} x + t_n) J_{n-1} + u_{n} J_{n-2},
\end{equation}
with explicit expressions for the coefficients $s_n, t_n, u_n \in \R$ which can be evaluated provided that integrals of polynomials with respect to the measure $\tilde\rho_x$ can be evaluated. This provides a fast and numerically stable means to evaluate the $J_n$. (See \cite[Sec. 11.4]{Powell1981-bg} for the details.)

After defining the $J_n(x)$ polynomials we define the radial basis to be
\begin{equation} \label{eq:rad:general_basis}
    \calP := \big\{ P_n := (J_n \circ \xi) \fcut  \,\big|\, n \in \N \big\}.
\end{equation}
We immediately obtain the following result.

\begin{proposition} \label{th:rad:general_rad_basis}
    The radial basis $\calP$ satisfies Assumption~\ref{as:P}, and in addition is orthogonal in $L^2(\rho)$; that is,
    \[
        \< P_n, P_{n'} \>_{L^2(\rho)} :=
        \int_{\ro}^{\rcut}  P_n(r) P_{n'}(r) \rho(dr) = \delta_{nn'}.
    \]
    In particular, with this choice of radial basis, the inner product $\<\!\< \cdot, \cdot \>\!\>$ in \eqref{eq:symm:abstract_ip} (with respect to which the basis $\calB_N$ is orthonormal) has the representation
    \[
        \<\!\< \phi_{\bn\bl\bbm}, \phi_{\bn'\bl'\bbm'} \>\!\> =
        \prod_{\alpha = 1}^N
        \< P_{n_\alpha}, P_{n_\alpha}' \>_{L^2(\rho)}
        \< Y_{l_\alpha}^{m_\alpha}, Y_{l_\alpha'}^{m_{\alpha}'} \>_{L^2(\bbS^2)}.
    \]
\end{proposition}
\begin{proof}
    The orthogonality properties follow by construction. We only need to show that ${\rm span}\calP$ is dense in $C^t_{\rcut}([\ro, \infty))$.

    Let $f \in C^t_{\rcut}([\ro, \infty))$ and define $f_\epsilon(x) := f(x + \epsilon)$, then $f_\epsilon = 0$ in $[\rcut-\epsilon, \infty)$ and it is easy to see that $f_\epsilon \to f$ in $C^t([\ro, \infty))$.
    Moreover, since $f_{\rm cut} \geq \delta > 0$ in $[\ro, \rcut-\epsilon]$ it follows that $g(x) := (f_\epsilon / f_{\rm cut}) \circ \xi^{-1} \in C^t([\ro, \rcut-\epsilon])$ and can therefore be approximated to within arbitrary accuracy by a polynomial $p(x)$.
    It follows readily that $p(\xi(r)) f_{\rm cut}(r)$ approximates $f_\epsilon$, and after invoking a diagonal argument we obtain the desired density.
\end{proof}

\subsection{The radial basis design space}
\label{sec:radial:remarks}
\begin{enumerate}[wide]
    \item The choice of transform is very general; typical choices are
    \[
        \xi(r) = (r/\rnn)^{-q}, \quad (1+r/\rnn)^{-q}, \quad  e^{-\lambda (r/\rnn)},
    \quad \dots,
    \]
    Extensive testing of such transforms for {\tt aPIPs}~\cite{2019-regapips1} we did not see significant differences in performance (in the sense of convergence of the \revi{root mean square error}). Our default choice is $\xi(r) = (1 + r/\rnn)^{-2}$.
    Still, it is conceivable that an optimisation of the coordinate transform, agressively adapting it to a given training set, may lead to measurable performance improvements. A potential concern to be explored is a loss of transferability.

    \item There is significant freedom in the choice of cut-off functions. We recommend
    \[
        \fcut(r) = \big|\xi(\rcut) - \xi(r)\big|^p,
    \]
    where $p \in \{2, 3, \dots\}$. This is attractive since the resulting products $P_n(r) = J_n(x) \fcut(r)$ is again a polynomial in $x$. In particular $p = 2$ is canonical in that it is the lowest power for which we have $C^{t}$-regularity across the cutoff (recall that $t \geq 1$).

    \item A canonical choice for $\rho_x$ is the uniform measure on $[\xi(\ro),\xi(\rcut)]$. In this case, the measure $\tilde{\rho}_x$ is of the form
    \[
        \tilde{\rho}_x(dx) = |\xi(\ro)-x|^\alpha |\xi(\rcut)-x|^\beta \, dx
    \]
    where $\alpha, \beta \geq 0$; cf. item (6) below, and the polynomials $J_n(x)$ are shifted and scaled {\em Jacobi polynomials} $J_n^{\alpha,\beta}$, for which there are analytic expressions for the recursion coefficients $s_n, t_n, u_n$ \cite[Sec. 12.3]{Powell1981-bg}.

    Another natural choice for $\rho_x$ is the radial distribution function on the training set, thus incorporating information about the least squares system into the basis set.

    \item The lower bound $\ro$ need not, and usually should not, be at $0$. Instead, one should typically choose $\ro$  close to the infimum of the support of the radial distribution function of the training set. In many of our tests $\ro \in [0.6 \rnn, 0.8\rnn]$ is typical. This is related to regularisation: the interval $[\ro, \rcut]$ specifies where the potential will be regular, and we should not expend degrees of freedom in a domain where no accuracy is required since no data is provided there; see \cite{2019-regapips1} for a more detailed discussion of this issue.

    \item Related to the previous point is the idea of a two-sided cutoff, which we introduced also in \cite{2019-regapips1}, e.g.,
    \[
        \fcut(x) = |\xi(\ro) - x|^p  |\xi(\rcut)-x|^p
    \]
    This applies a standard outer cutoff at $r = r_{\rm cut}$ and an inner cutoff at $r = r_{0}$ and prevents the basis functions from oscillating and/or diverging in the domain $r \in (0, \ro)$.

    \item In the computational chemistry domain it is common  to choose a one-particle basis
    \[
        \phi_{nlm}(\br) = P_{nl}(r) Y_l^m(\hat\br),
    \]
    i.e., to let the radial basis depend on $l$. This is motivated by the \revi{Linear Combination of Atomic Orbitals (LCAO)} basis where the radial components of electron wave functions strongly depend on the angular momentum $l$~\cite[Fig.1]{MadsenPRB2011}.
    From an approximation theoretic perspective this additional freedom does not change any of our results, but enables additional optimisations of the basis set.
    A canonical formulation~\cite[\S~C.1]{Drautz2019-er} is to define
   \[
       P_{nl}(r) = \sum_{k} c_{nlk} Q_k(r),
   \]
    where $\{Q_k\}_k$ is a standard radial basis, and to optimise the $c_{nlk}$. This is formally equivalent to defining a higher radial polynomial degree, but possibly leads to a {\em sparser} basis. Whether this does lead to significant savings has not been systematically explored.

    \item A further feature would be to allow the cutoff radius $\rcut$, or radii $(\ro, \rcut)$ for inner and outer cutoffs, to be independent for each $P_{nl}$ which would open the possibility of constructing a ``multi-scale'' basis suitable for modelling long-range interactions.
\end{enumerate}

The above items point to a wide range of potential optimisations of the radial basis $\calP$. We emphasize again that most of the freedom in the design of a permutation and isometry invariant basis set within our framework lies in the choice of $\calP$.
It will therefore be crucial in future work to investigate these options further. One might aim to optimise the choice of radial basis more aggressively through a nonlinear parameter optimisation procedure.
In addition to the ideas mentioned above, such as optimising the coordinate transformation, or choosing a radial basis of the form $P_{nl}$, possibly with independent cutoff parameters, this could lead to significant performance improvements.

\section{Parameter Estimation}
\label{sec:reg}
Suppose we have constructed a finite symmetric polynomial basis set ${\bm B} \subset \bigcup_{N = 1}^{\calN} {\bm B}_N$. Via \eqref{eq:proj:B_kli_efficient} every choice of parameters ${\bm c} = (c_B)_{B \in {\bm B}}$ defines a site potential (see also \S~\ref{sec:pes})
\begin{equation} \label{eq:lsq:V}
    V_{\bm c}(\{\br_{\iota j}\}) :=
    \sum_{B \in {\bm B}} c_B B(\{\br_{\iota j}\}),
\end{equation}
where $\{\br_{\iota j}\}_{j}$ denotes a collection of atom positions relative to a centre-site $\iota$, i.e., $\br_{\iota j} = \br_j - \br_\iota$.
This then defines a potential energy surface $E_{\bm c}$,
\begin{equation} \label{eq:lsq:pes}
   E_{\bm c}(\{\br_\iota\}) :=
   \sum_{\iota} \sum_{B \in {\bm B}} c_B B\big( \{ \br_{\iota j}\}_{r_{\iota j} < \rcut} \big),
\end{equation}
which is parameterised by ${\bm c}$. In the present section we will discuss techniques and challenges surrounding the estimation of the parameters ${\bm c}$.

\subsection{Linear Least Squares}
\label{sec:lsq}
\def\calJ{\mathcal{J}}
We are given a training set $\mathcal{R}$ containing atomic configurations $R \in \R^{3J}$ (possibly combined with a cell and a boundary condition, and where $J$ does not need to be the same over the configurations),
together with corresponding observations: total potential energies $\calE_R \in \R$, forces $\calF_R \in \R^{3J}$ and possibly other quantities such as virials, force constants, and so forth. Most commonly, these are obtained from a high-fidelity electronic structure model $E^{\rm ref}$, in which case $\calE_R = E^{\rm ref}(R), \calF_R = - \nabla E^{\rm ref}(R)$. We then minimize the quadratic cost function
\begin{equation} \label{eq:lsq_fcnl}
      \calJ[E, \calR] =
      \sum_{{R} \in \mathcal{R}} \Big(
         w_{R,E}^2 \left| E({R}) - \calE_{R} \right|^2
         + w_{R,F}^2 \left| -\nabla  E(R) - \calF_{R} \right|^2
         \Big),
\end{equation}
with respect to the new potential $E$,  where $w_{E, R}$, $w_{F, R}$ are weights that may depend on the configurations $R$ as well as the observations $\calE_R, \calF_R$.

In the case \eqref{eq:lsq:pes} the parameterised potential energy surface $E_{\bm c}$ energy and forces are {\em linear} functionals of the parameters ${\bm c}$. Therefore, the minimisation of $\calJ[E_{\bm c}, \calR]$  with respect to ${\bm c}$ can be rephrased as a linear least-squares problem
\begin{equation}
   \label{eq:LSQ_system}
   \min_{\bm c} \|\Psi{\bm c} - {\bm y}\|_2^2,
\end{equation}
where ${\bm y} \in \R^{N_{\rm obs}}$ contains the weighted observations $\calE_R, \calF_R$, and $\Psi \in \R^{N_{\rm obs} \times N_{\rm basis}}$, with $N_{\rm basis}  = \# {\bm B}$ being the number of basis functions. This is then solved using a QR factorisation, or a rank-revealing QR factorisation for a simple form of regularisation.

\revii{
    We note that, in practice, the least squares formulation should include also a regularisation term. Since we are primarily concerned with {\it convergence} of the parameterisation we have ignored this here. However, in Sec.~\ref{sec:regularisation} we briefly comment on the ill-posedness of the inverse problem and connect to regularisation.
}

\subsection{Convergence in \revi{root mean square error (RMSE)}}
\label{sec:convergence_loss}
\def\RMSE{{\rm RMSE}}
One immediately expects that uniform convergence of the potential approximations implies also convergence in RMSE.
Nevertheless, it is still interesting to formalise this statement, for example with an eye to more quantitative future results, or to situations where we relinquish uniform convergence but still insist on convergence in RMSE.
For the sake of simplicity, we consider only energy and forces (but no virials) and only cluster configurations $R$ but no periodic boundary conditions.

Given a potential energy surface $E$ and a training set $\calR$, the RMSE is defined by
\begin{align*}
    \RMSE[E, \calR] &:= \Big( Z_{\calR}^{-1} \calJ[E, \calR] \Big)^{1/2},
    \qquad \text{where} \\
    Z_{\calR} &:= \sum_{R \in \calR} \big(1 + 3\# R\big).
\end{align*}
Here, $1 + 3\# R$ denotes the number of scalar observations obtained from the configuration~$R$, one total energy and $3 \# R$ forces.

Such a scaling is meaningful if the weights satisfy the scaling
\begin{equation}  \label{eq:conv_rmse:weights}
    w_{R, E} \leq \frac{\bar{w}_E}{\# R}
    \quad \text{and} \quad
    w_{R, F} \leq \bar{w}_F,
\end{equation}
where $\bar{w}_E, \bar{w}_F$ are fixed constants independent of $R$.

We consider an approximation by potential energy surfaces $E^{(\nu)}_{\bm c}$ of the form \eqref{eq:lsq:pes}, with basis  ${\bm B} = {\bm B}^{(\nu)}, \nu \in \N$. We assume the sequence of basis sets ${\bm B}^{(\nu)}$ satisfies the following properties:
\begin{enumerate}
   \item[(B.1)] {\em Convergence:} ${\bm B}^{(\nu)} \uparrow \bigcup_{N = 1}^{\calN} {\bm B}_N$ as $\nu \to \infty$.
   \item[(B.2)] {\em Completeness:} The full interaction order-$N$ basis sets ${\bm B}_N$ satisfy the conditions of Theorem~\ref{th:density_acetype_basis}.
\end{enumerate}

Convergence of the RMSE for a fixed training set is not particularly illuminating, hence we consider the convergence of the RMSE for a sequence of increasing training sets $\calR^{(\nu)}, \nu \in \N$ such that $\#\calR^{(\nu)} \uparrow \infty$ as $\nu \to \infty$.
The key point is that we allow $\# \calR^{(\nu)} \gg \# {\bm B}^{(\nu)}$ as $\nu \to \infty.$ We assume that all $\calR^{(\nu)}, \nu \in \N$ satisfy the following restrictions:
\begin{enumerate}
   \item[(R.1)] {\em Uniform separation:} there exists $\ro > 0$ such that
\[
   |\br - \br'| \geq \ro \qquad \forall \br \neq \br' \in R, \quad R \in \calR^{(\nu)}.
\]
   \item[(R.2)] {\em Model consistency:} the energies $\calE_R$ and forces $\calF_R$, $R \in \calR^{(\nu)}$, are obtained by evaluating a potential energy surface satisfying \eqref{eq:cluster expansion} and Assumption~\ref{as:VN}.
   \item[(R.3)] {\em Weight scaling:} The weights $w_{R, E}, w_{R, F}$, $R \in \calR^{(\nu)}$ satisfy \eqref{eq:conv_rmse:weights} where $\bar{w}_E, \bar{w}_F$ are independent of $\nu$.
\end{enumerate}

\begin{proposition}
   Suppose that conditions (R.1), (R.2), (R.3) as well as (B.1), (B.2) are satisfied; then
   \[
      \lim_{\nu \to \infty}
      \inf_{\bf c} \RMSE\big[ E^{(\nu)}_{\bf c}, \calR^{(\nu)} \big]
       = 0.
   \]
\end{proposition}
\begin{proof}
   The result is elementary, hence we only provide a brief sketch. Under assumptions (B.1) and (B.2), Theorem~\ref{th:density_acetype_basis} (density of  the basis) implies that there exist
   parameters ${\bf c} = {\bf c}^{(\nu)}$ such that the resulting order-$N$
   potentials as well as their gradients converge uniformly to the order-$N$
   potentials in the exact model. The separation condition (R.1) implies that
   each site energy and each force can involve at most finitely many terms with
   a uniform upper bound. Combined with the scaling of the weights, it now
   follows easily that all weighted observations converge uniformly. The
   normalisation $Z_{\calR^{(\nu)}}^{-1}$ of the RMSE then guarantees that
   $\RMSE[ E^{(\nu)}_{{\bf c}^{(\nu)}}, \calR^{(\nu)} ] \to 0,$
   and in particular the same convergence holds for the optimised RMSE.
\end{proof}

\begin{remark}
While (R.1) and (R.3) are natural and  mild restrictions, (R.2) is more severe since the observations $\calE_R, \calF_R$ are typically obtained from electronic structure models that clearly do not satisfy it.
This issue goes back to the discussion in \S~\ref{sec:pes} whether one can {\em in principle} approximate an electronic structure model PES using the body-order expansion \eqref{eq:cluster expansion}.
If this can be established, then we would need to add $\calN \to \infty$ and $\rcut \to \infty$ as $\nu \to \infty$ in order to obtain a generalised result that doesn't require the assumption (R.2).
\end{remark}

\section{Implementation and Performance}
\label{sec:implementation}
The following paragraphs provide a brief guidance on the practical implementation of the class of symmetric polynomials bases we constructed, and also serve as a summary of the concrete steps required to construct and evaluate the basis, fit the parameters and construct a site energy potential. 
An implementation based on these notes is provided in an open source Julia package {\tt ACE.jl}~\cite{ACEjl}. 
All tests we report are performed on an {\tt Intel(R) Core(TM) i7-7820HQ CPU @ 2.90GHz}, with {\tt macOS (x86-64-apple-darwin19.5.0)} operating system, {\tt Julia Version} {\tt 1.5.0-rc1}, and {\tt ACE.jl Version 0.6.9}.

\revii{
    The purpose of this section is not to document the details of our implementation~\cite{ACEjl}, nor of the specific choices of certain model parameters we make. Instead, we only summarize the choices and main steps that must be considered in constructing and implementing a basis, fitting the parameters, and eventually evaluating the resulting site potential. Some of the model parameter choices are left deliberately vague as they will differ significantly across different modelling scenarios (e.g., metal, covalent systems, molecules, \dots). Exploring the possibility of a systematic basis construction will be the focus of future work.
}

\subsection{Construction of the symmetric basis}
\label{sec:implementation:basisconstruction}
Within the framework of this paper, a symmetric polynomial basis is specified by the following three steps.
\begin{enumerate}[wide]
    \item Choose a radial basis $\calP = \{P_n\}$\revii{, following the discussion Sec.~\ref{sec:radial}}.
    \item Specify a finite set of tuples
      \[
            \Spec \subset \{ (\bn, \bl) \in  \N^N \times \N^N : N \geq 0 \},
      \]
    where each $(\bn, \bl)$ specifies one or more \revii{many-body} basis functions.
     Their interaction order $N$ is implicitly given by its length. A canonical approach to choosing $\Spec$ is to specify a degree $d$ and construct all tuples $(\bn, \bl)$ up to some maximum $\calN$ such that
     \[
         {\rm deg}(\bn, \bl) \leq d.
     \]
     with a suitable {\em degree function} ${\rm deg}$, e.g., the tensor degree, ${\rm deg}(\bn, \bl) = \max_\alpha \max(n_\alpha, l_\alpha)$, the total degree ${\rm deg}(\bn, \bl) = \sum_\alpha n_\alpha + l_\alpha$, or weighted versions thereof. More general choices are of course possible and interesting.

    \item Construct the coupling coefficients: for each $(\bn, \bl) \in \Spec$ compute the coefficients $\calU_{i\bbm}^{\bn\bl}$, $\bbm \in \calM_\bl^0, i = 1, \dots, n_{\bn\bl}$. \revii{This can be performed, either through the numerical construction~\S~\ref{sec:symm:rot} or the explicit construction~\S~\ref{sec:rotinv:Drautz}, followed by the additional step incorporating the permutation symmetry.}
    Then the resulting symmetric polynomial basis functions $B_{\bn\bl i}$, $(\bn, \bl) \in \Spec, i = 1, \dots, n_{\bn\bl}$, are given by \eqref{eq:proj:B_kli_efficient}.
\end{enumerate}

\subsection{Evaluation of the symmetric basis}
\label{sec:eval_basis}
Given a finite basis set as constructed in \S~\ref{sec:implementation:basisconstruction}, we now outline its evaluation to obtain \revii{total potential energy. Other properties can be computed analogously,} for example forces or virials through differentiation with respect to positions and cell shape.

Given a configuration $\{\br_\iota\}$, the total potential energy is defined as $E = \sum_\iota E_\iota$ where each site energy $E_\iota$ is represented by the symmetric polynomial basis. Thus, the resulting basis for total energy is given by
\[
    B^E := \sum_\iota B^{(\iota)},
\]
where $B^{(\iota)}$ is the symmetric polynomial basis evaluated at site $\iota$ and is computed as follows:

\begin{enumerate}
    \item Determine all atoms at positions $\br_j$ such that $r_{\iota j} = |\br_{\iota j}| < \rcut$, where $\br_{\iota j} = \br_j - \br_\iota$.
    \item Evaluate the one-particle basis,
    \[
        A_{nlm}^{(\iota)} := \sum_j \phi_{nlm}(\br_{\iota j})
    \]
    \item Evaluate the permutation-invariant many-body basis,
    \[
        A_{\bn\bl\bbm}^{(\iota)} := \prod_{\alpha = 1}^{N_{\bn\bl}} A_{n_\alpha l_\alpha m_\alpha}^{(\iota)},
        \qquad
        (\bn, \bl) \in \Spec,  \quad \bbm \in \calM_\bl^{0},
    \]
    where $\bn, \bl \in \N^{N_{\bn\bl}}$. \revii{This can be performed using either the naive formula or the recursive representation described in \S~\ref{sec:graph_representation}.}
    \item Evaluate the  RPI basis 
    \[
        B^{(\iota)}_{\bn \bl i}  := \sum_{\bbm \in \calM_{\bl}^0} \calU_{i\bbm}^{\bn \bl} A_{\bn \bl \bbm}^{(\iota)},
        \qquad (\bn, \bl) \in \Spec, \quad i = 1, \dots, n_{\bn\bl}.
    \]
    \revii{This can be conveniently implemented as a sparse matrix-vector multiplication.}
\end{enumerate}

\subsection{Construction of the site potential}
\label{app:implementation:V}
After the basis $\{B_{\bn\bl i}\}$ has been constructed the coefficients ${\bf c} = (c_{\bn\bl i})$ must be determined, for example, using the procedures discussed in \S~\ref{sec:reg}. \revii{We will not give further details on this step since it simply involves converting the loss minimization into a linear least squares problem, which is then solved using the QR factorisation or the iterative LSQR method.}

\revii{Once the parameters have been determined, a site potential is then defined by}
\[
    V\big(\{ \br_{\iota j}\}\big) = V_{\bf c}\big(\{ \br_{\iota j}\}\big)
    = \sum_{ (\bn, \bl) \in \Spec} \sum_{i = 1}^{n_{\bn\bl}}
            c_{\bn\bl i} B_{\bn \bl i}(\{\br_{\iota j}\}),
\]
which can be simplified and optimised as follows:

\begin{enumerate}[wide]
    \item We combine the basis coefficients with the coupling coefficients,
    \[
        a_{\bn\bl\bbm}' := \sum_{i = 1}^{n_{\bn\bl}} c_{\bn\bl i} \calU_{i\bbm}^{\bn\bl},
        \qquad (\bn, \bl) \in \Spec, \quad \bbm \in \calM_\bl^0.
    \]
     to obtain an equivalent represention of $V$,
    \[
        V\big(\{ \br_{\iota j}\}\big)
        = \sum_{ (\bn, \bl) \in \Spec} \sum_{\bbm \in \calM_\bl^0}
                a_{\bn\bl \bbm}' A_{\bn\bl\bbm}\big(\{\br_{\iota j}\}\big).
    \]

    \item Let $\calM_{\bn\bl\bbm}$ denote the set of all $\bbm' \in \calM_\bl^0$ for which there exists a permutation $\pi$ such that $\pi(\bn) = \bn, \pi(\bl) =  \bl$ and $\pi(\bbm') = \bbm$.
    Then $A_{\bn\bl\bbm} = A_{\bn\bl\bbm'}$ for all $\bbm' \in \calM_{\bn\bl\bbm}$.
    Thus, we can further reduce the representation of $V$ by defining
        \begin{align*}
            a_{\bn\bl\bbm} &:=
                \sum_{\bbm' \in \calM_{\bn\bl\bbm}} a_{\bn\bl\bbm'}',
            \qquad \text{for } (\bn, \bl, \bbm) \in \SpecV, \qquad \text{where}  \\ 
        \SpecV &:= \big\{ (\bn, \bl, \bbm) :
                (\bn, \bl) \in \Spec, \bbm \in \calM_{\bl}^0,
                \text{ and $(\bn, \bl, \bbm)$ is  ordered} \big\}.
    \end{align*}
    We then obtain the final representation of the site potential,
    \begin{equation} \label{eq:evalV}
        V\big(\{ \br_{\iota j}\}\big)
        = \sum_{ (\bn, \bl, \bbm) \in \SpecV}
                a_{\bn\bl \bbm} A_{\bn\bl\bbm}\big(\{\br_{\iota j}\}\big).
    \end{equation}
    A particular advantage of this representation is that it is unique. Permutation symmetry is ensured by the usage of the $A_{\bn\bl\bbm}(\{\br_{\iota j}\})$ functions, which form a basis of permutation-invariant polynomials, while O(3) symmetry is now encoded in the coefficients $a_{\bn\bl \bbm}$.
\end{enumerate}

The site potential $V$ is straightforward to evaluate from \eqref{eq:evalV} after first following the steps (1--3) in \S~\ref{sec:eval_basis}, followed by evaluating \eqref{eq:evalV}. \revii{Step (4) is no longer needed since the symmetry is now encoded in the coefficients $a_{\bn\bl \bbm}$. }

\subsection{Fast gradients via adjoints }
\label{sec:fast_gradients}

\subsubsection{Standard evaluation algorithm}
\label{sec:alg:standardeval}
The standard evaluation of \eqref{eq:evalV} is unremarkable, hence we focus on the implementation of fast gradients to obtain forces and virials. This is a straightforward application of backward differentiation (the adjoint method~\cite{Strang2007-nj}). The backward differentiation formula for the standard representation can be found in~\cite{Drautz2019-er}. Here we include an algorithm for the sake of completeness. We conveniently incorporate parts of the backward pass into the forward pass which simplifies implementation and saves storage.

\revii{In the following algorithm the symbols $W_{nlm}$ are the ``adjoints'' that must be computed to eventually assemble the gradient. We refer to~\cite{Strang2007-nj} for further details. }

\bigskip \noindent {\bf Standard Adjoint Gradient Algorithm}
\begin{enumerate}
    \item For all $n, l, m$: 
    \item \qquad $A_{nlm} \leftarrow \sum_j \phi_\revii{nlm}(\br_j)$ 
    \item \qquad \revii{$W_{nlm} \leftarrow 0$}
    \item For all $(\bn, \bl, \bbm) \in \SpecV$:
    \item \qquad $V \leftarrow V + c_{\bn\bl\bbm} \prod_{\alpha=1}^N A_{n_\alpha l_\alpha m_\alpha}$  
    \item \qquad For $\alpha = 1, \dots, N$
    \item \qquad \qquad $W_{n_\alpha l_\alpha m_\alpha} \leftarrow 
            \revii{W_{n_\alpha l_\alpha m_\alpha}} + 
            a_{\bn\bl\bbm} \prod_{\beta \neq \alpha} A_{n_\beta l_\beta m_\beta}$.
    \smallskip 
    \item For all $n, l, m$ and neighbour indices $j$:
    \item \qquad $\nabla_{\br_{\iota j}} V \leftarrow  \nabla_{\br_{\iota j}} V 
            + \revii{W_{nlm}} \nabla \phi_{nlm}(\br_{\iota j})$
\end{enumerate}

\subsubsection{Recursive evaluation algorithm}
\label{sec:alg:grapheval}
\revii{In line with \S~\ref{sec:graph_representation} we now switch to the compressed index notation where $k \equiv (n,l,m)$ and $\bk = (\bn, \bl, \bbm)$.}
To evaluate \eqref{eq:evalV} using the recursive representation introduced in \S~\ref{sec:graph_representation} we first need to expand the PI basis ${\bm A} = \{  A_{\bk} = A_{\bn\bl\bbm} : \bk = (\bn, \bl, \bbm) \in \SpecV  \}$ with auxiliary nodes to an extended basis ${\bm A}^{\rm ext}$ so that each $A_{\bk} \in {\bm A}^{\rm ext}$ can be decomposed into $A_{\bk} = A_{\bk'} A_{\bk''}$ where ${\rm len}(\bk'), {\rm len}(\bk'') < {\rm len}(\bk)$. \revii{Recall that these auxiliary basis functions are assigned the coefficient $a_\bk = 0$.}

Such an extension is always possible, for example by simply inserting the products
\[
    A_{k_1} A_{k_2}, 
    \quad 
    A_{k_1 k_2} A_{k_3}, 
    \quad 
    A_{k_1 k_2 k_3} A_{k_4},
    \dots  
\]
to the basis. In \revii{practice} we search all possible partitions and choose the one that minimizes the number of new basis functions to be inserted into the extended basis. An in-depth exploration of this heuristic, and of alternative constructions, will be pursued in a separate work.

For the following algorithm we use the notation 
\[
    \bk', \bk'' \Rightarrow \bk, c_\bk  
\]
to indicate that the contribution $a_\bk A_\bk$ to the site potential $V$ is to be computed via the operation $V \leftarrow V + a_\bk A_\bk$ with $A_\bk \leftarrow A_{\bk'} A_{\bk''}$. 

\revii{As in the previous section, we denote the one-particle basis adjoints by $W_k \equiv W_{nlm}$, but now also require many-body basis adjoints $W_{\bk} \equiv W_{\bn\bl\bbm}$.}

\bigskip \noindent {\bf Adjoint Gradient Algorithm: Graph Representation}
\begin{enumerate}
    \item {\bf Forward Pass:}
    \item For all $k = (n, l, m)$: 
    \item \qquad $A_k \leftarrow \sum_j \phi_k(\br_j)$ 
    \item \qquad $V \leftarrow V + a_k A_k$ 
    \item \qquad $\revii{W_k} \leftarrow a_k$
    \item For all operations $(\bk', \bk'') \Rightarrow (\bk, a_\bk)$ with increasing ${\rm len}(\bk)$:
    \item \qquad $A_{\bk} \leftarrow A_{\bk'} A_{\bk''}$
    \item \qquad $V \leftarrow V + a_\bk A_\bk$  
    \item \qquad $\revii{W_\bk} \leftarrow \revii{W_\bk} + a_\bk$ 
    \smallskip 
    \item {\bf Backward Pass:}
    \item For all operations $(\bk', \bk'') \Rightarrow (\bk, a_\bk)$ with decreasing ${\rm len}(\bk)$:
    \item \qquad $\revii{W_{\bk'}} \leftarrow \revii{W_{\bk'}} + a_\bk \revii{W_{\bk''}}$
    \item \qquad $\revii{W_{\bk''}} \leftarrow \revii{W_{\bk''}} + a_\bk \revii{W_{\bk'}}$ 
    \item For all $k = (n, l, m)$ and neighbour indices $j$:
    \item \qquad $\nabla_{\br_{\iota j}} V \leftarrow  \nabla_{\br_{\iota j}} V 
            + \revii{W_k} \nabla \phi_{k}(\br_{\iota j})$
\end{enumerate}

\subsubsection{Performance comparison}
A performance comparison of the standard and recursive evaluation is presented in Figure~\ref{fig:graphspeedup}.  
Only the evaluation of the correlations $A_{\bk}$ are included in the timing, but not the evaluation of the one-particle basis $A_k$ which is the same for both methods. The test is representative insofar as similar efforts have been made in optimising each evaluation code.

The conclusion of this test is that the speedup of the recursive representation is indeed significant due to the fact the very few artificial basis functions need to be added. An explanation of this observation requires a better understanding of algebraic (in-)dependence of the basis functions, e.g., possibly leveraging the observations made in \cite{Nigam2020-zj}. Depending on the computational cost of evaluating the one-particle basis the practical speed-up at low interaction orders may be less pronounced; see the next section for tests in a more realistic setting.

\begin{figure}
    \includegraphics[height=7cm]{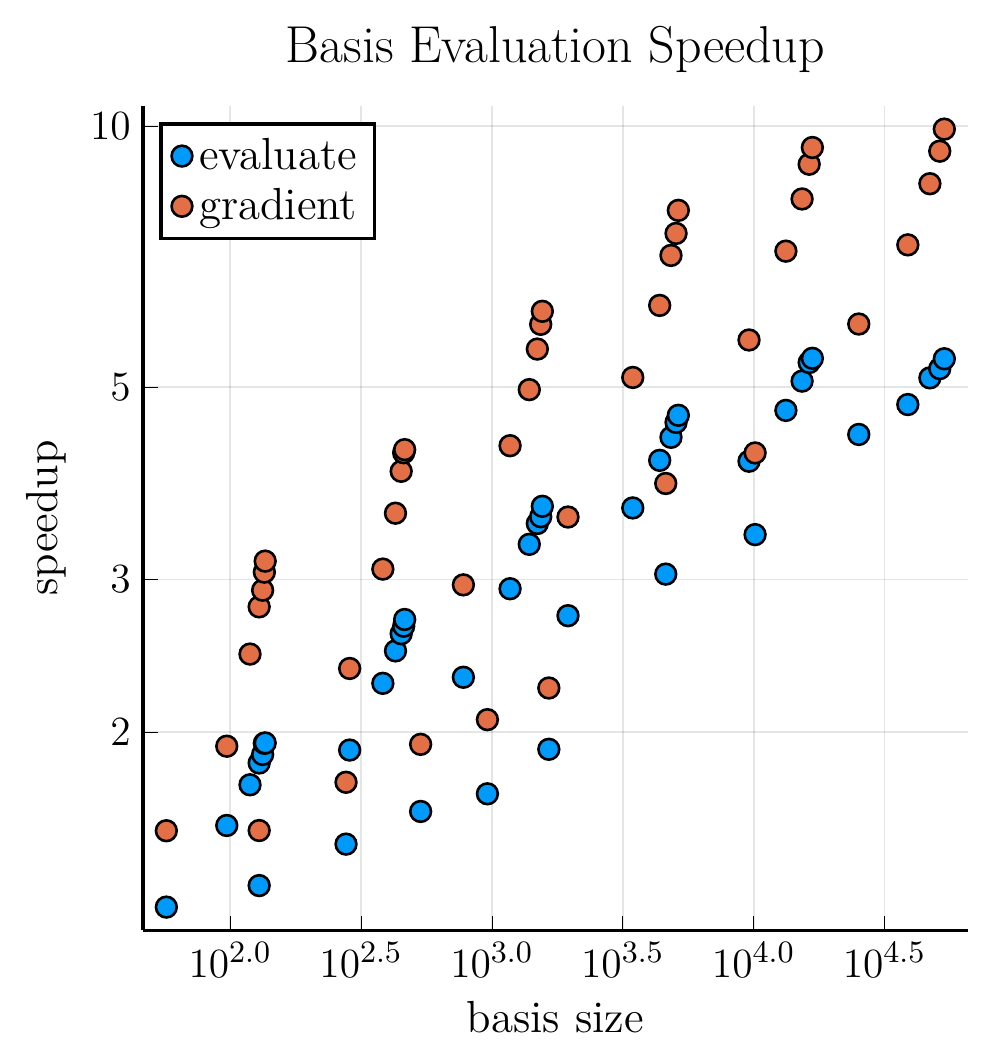}\qquad
    \includegraphics[height=7cm]{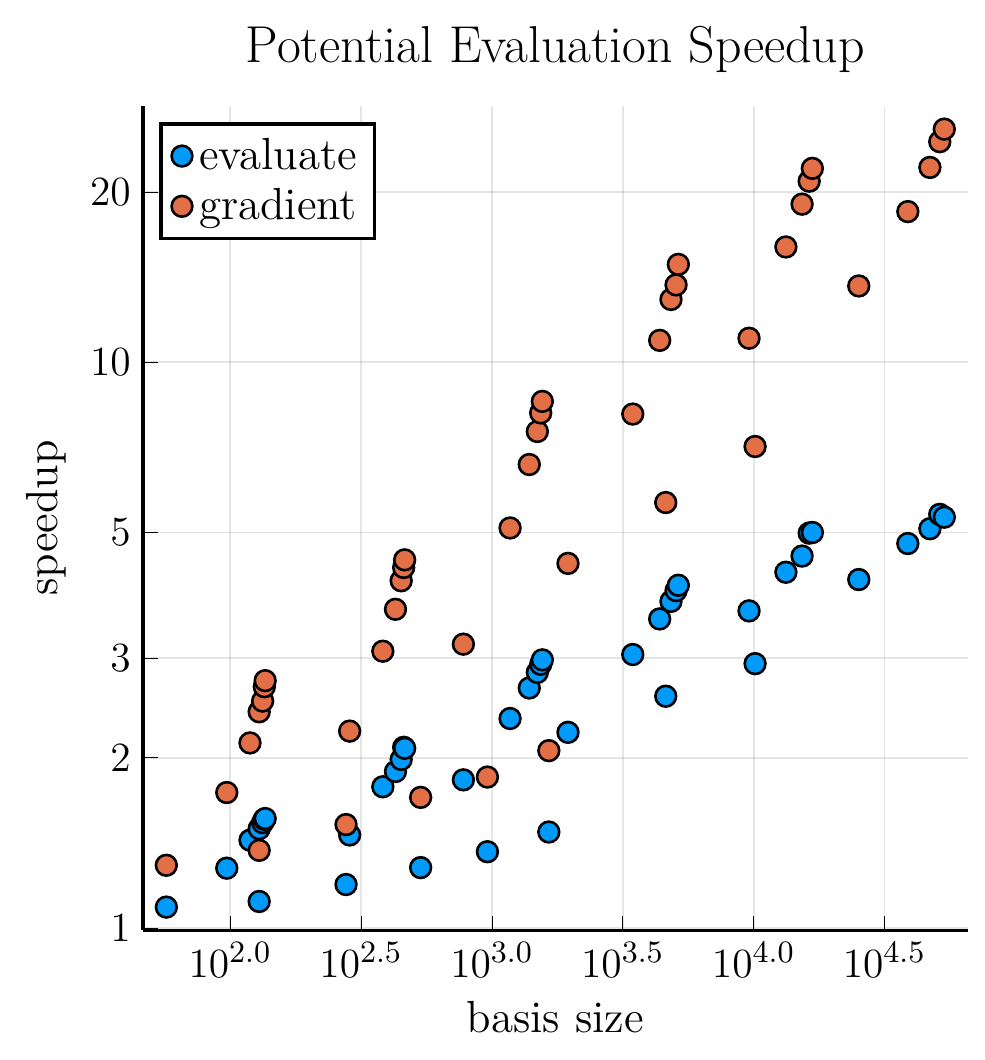}
    \caption{
    Speedup of recursive versus standard evaluator for the ACE basis (left) and potentials (right). For increasing interaction orders $N$ and total polynomial degrees (cf. \S~\ref{sec:implementation:basisconstruction}) an RPI basis is generated via \eqref{eq:B_basis} and then supplied with random coefficients to obtain a potential \eqref{eq:evalV}.}
    \label{fig:graphspeedup}
\end{figure}

\subsection{Convergence Tests}
\label{sec:convergence_tests}
The analysis of the previous sections establishes that the class of symmetric polynomials we propose provides a systemically improvable construction and can {\em in principle} represent any many-body PES to arbitrary accuracy. To provide more quantitative information supporting these results we performed convergence tests using two previously published training sets for, respectively, Si \cite{Bartok2018-fk} and W \cite{Szlachta2014-qr}. \revii{These convergence tests were done by fitting ACE bases with increasing polynomial degree, interaction order and cutoff to the previously described training databases. The radial basis was constructed according to \S~\ref{sec:orth_radial} with the following choices: The distance transform used was $\xi(r) = (1+r/\rnn)^{-2}$ where $\rnn$ was chosen to be the nearest neighbour distance for Si and W respectively. For both elements the recommended cutoff function $\fcut(r) = (\xi(\rcut) - \xi(r))^2$ was used. The outer cutoff $\rcut$ for Si was 5.5 \AA and for W 5.0 \AA, the inner bound for the orthogonality measure chosen as $r_0 = 0.7\rnn$.

To construct the many-body bases for the benchmarks, we chose polynomial degrees 8, 12, 14, 16, 18, 20 and 22 and interaction orders $N$ ranging from 2 to 7 for Si and from 2 to 5 for W. We used a {\em weighted total degree} of a basis function defined by ${\rm deg}(\bn, \bl, \bbm) = \sum_i n_i + 2 l_i$, which leads to higher degrees in the radial than in the angular components.
}

The Si training set contains a wide range of configurations, including surfaces, bulk structures, defects and high temperature liquid \cite{Bartok2018-fk}. It was assembled to create a GAP model which successfully describes a wide range of properties such as elastic constants, brittle fracture, surface formation energies and random structure search. Figure \ref{fig:si_conv} shows the convergence of our symmetric polynomial parameterisation with increasing interaction order and polynomial degree. For $N$=7 an RMSE set error of 2.33 meV/atom is achieved on the entire Si database. This fit had an evaluation time of 0.66 ms/atom per force evaluation. 

\begin{figure}
    \includegraphics[width=0.9\textwidth]{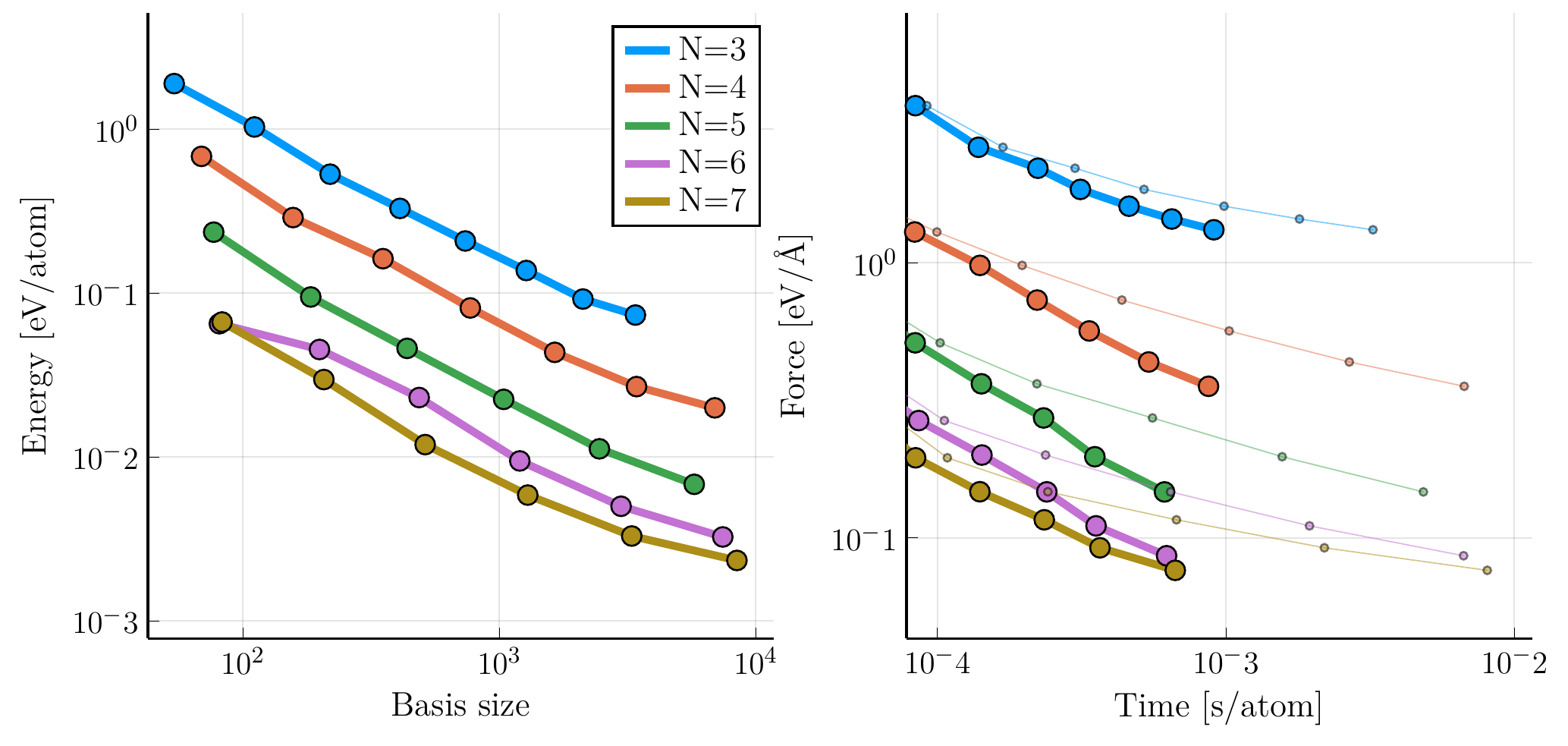}
    \caption{RMSE Convergence test on Si training set for increasing interaction order $N$ and polynomial degree, fixing the cutoff at 5.5 $\angstrom$. Left: convergence of energy RMSE error against increasing polynomial degree. Right: convergence of the force RMSE against the costs of a force evaluation using the recursive evaluator (thick lines) and the standard evaluator (thin lines). 
    }
    \label{fig:si_conv}
\end{figure}

The W database consists of a more narrow range of configurations including only bulk configurations and surfaces. These convergence test results are shown in Figure \ref{fig:w_conv}.
For $N$=3 an RMSE set error of 1.43 meV/atom was achieved with force evaluation timing at 1.12 ms/atom. Increasing the interaction order to $N$=5 resulted in comparable RMSE errors at 1.55 meV/atom but with better performance at 0.55 ms/atom per force evaluation. To explain why it possible that the RMSE {\em increases} with increasing interaction order we note that we are comparing at a fixed number of basis functions. Thus, a challenging compromise has to be made between increasing polynomial degree or interaction order.

\begin{figure}
    \includegraphics[width=0.9\textwidth]{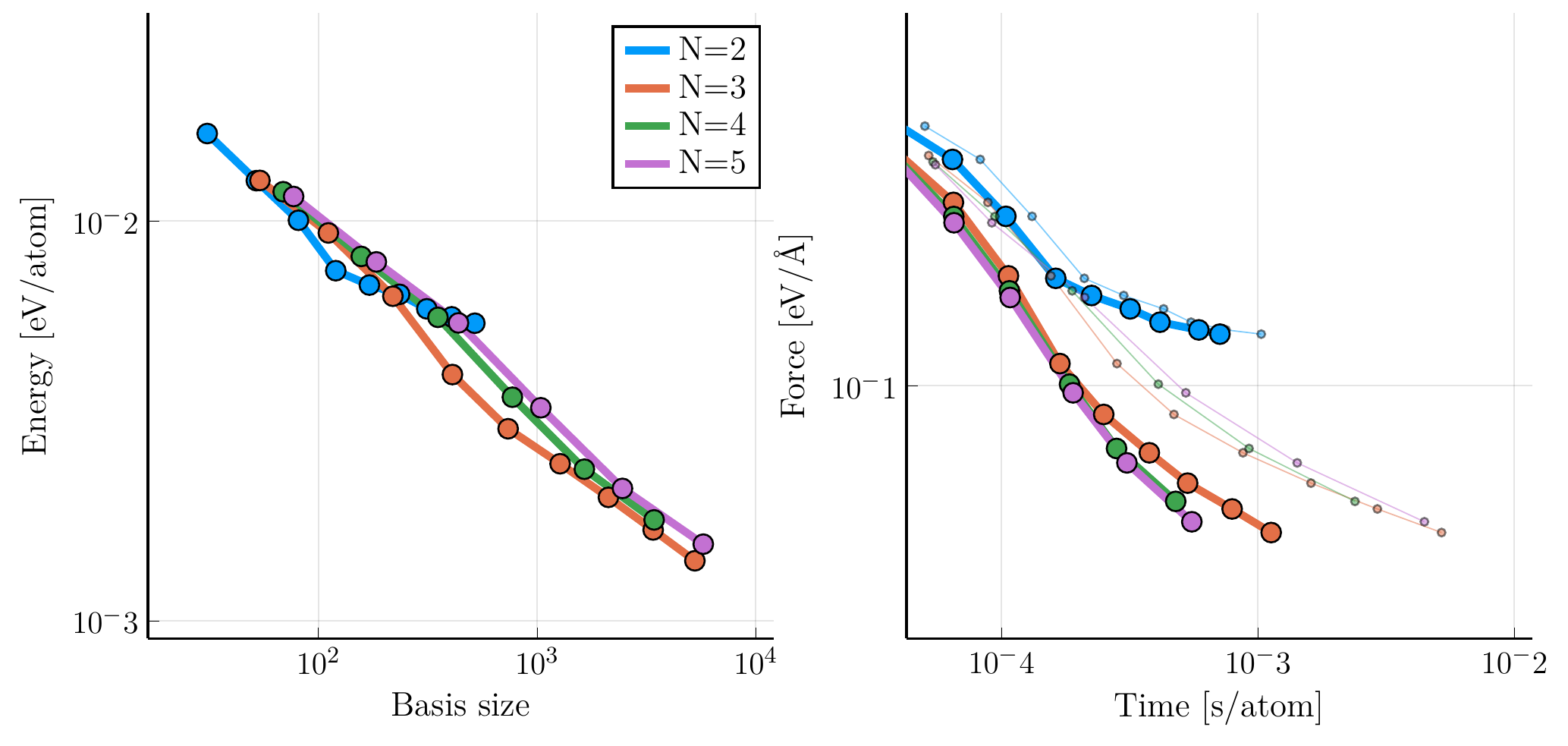}
    \caption{RMSE Convergence test on W training set for increasing interaction order $N$ and polynomial degree, fixing the cutoff at 5.0 $\angstrom$. Left: convergence of energy RMSE error against increasing polynomial degree. Right: convergence of the force RMSE against the costs of a force evaluation using the recursive evaluator (thick lines) and the standard evaluator (thin lines). 
    }
    \label{fig:w_conv}
\end{figure}

Comparing the convergence tests it can be concluded close to meV convergence is achieved at a lower interaction order for W compared to Si. One possible explanation for this is the lack of liquid configurations in the W database compared to the Si database. Moreover we observe that in this more realistic setting the recursive evaluator increases performance by roughly an order of magnitude compared to the standard evaluator. 

These timings are significantly slower than classical interatomic potentials but competitive with leading ML type potentials using similarly rich parameterisations. \revii{It should also be noted that the timings do not seem to scale with respect to the interaction order $N$ due to the projection of the ACE basis on the atomic neighbourhood density.}

\subsection{Convergence in cutoff radius}
\label{sec:cutoff_convergence_tests}
In addition to the polynomial degree and interaction order explored in the previous section, the ACE has a third approximation parameter, the cutoff radius $\rcut$. We investigate RMSE error convergence with respect to $\rcut$ for the W training set and interaction order of $N$=3. 
In Figure \ref{fig:w_rcut} we show the results for $\rcut = 4.0\angstrom, 5.0\angstrom, 6.0\angstrom$ and increasing total polynomial degree. 
Especially in the W test, we can see that relatively small bases have lower set RMSE error for small $\rcut$ distances, while large bases show improved RMSE set accuracy for larger $\rcut$ distances.
A possible explanation for this effect is that the increase in $\rcut$ increases the complexity of the fit (or, the size of the domain of definition) thus requiring a larger basis to fit accurately. 
It remains an open question how to optimally exploit the balance between cutoff, basis size and interaction order. 

\begin{figure}
    \includegraphics[height=6.7cm]{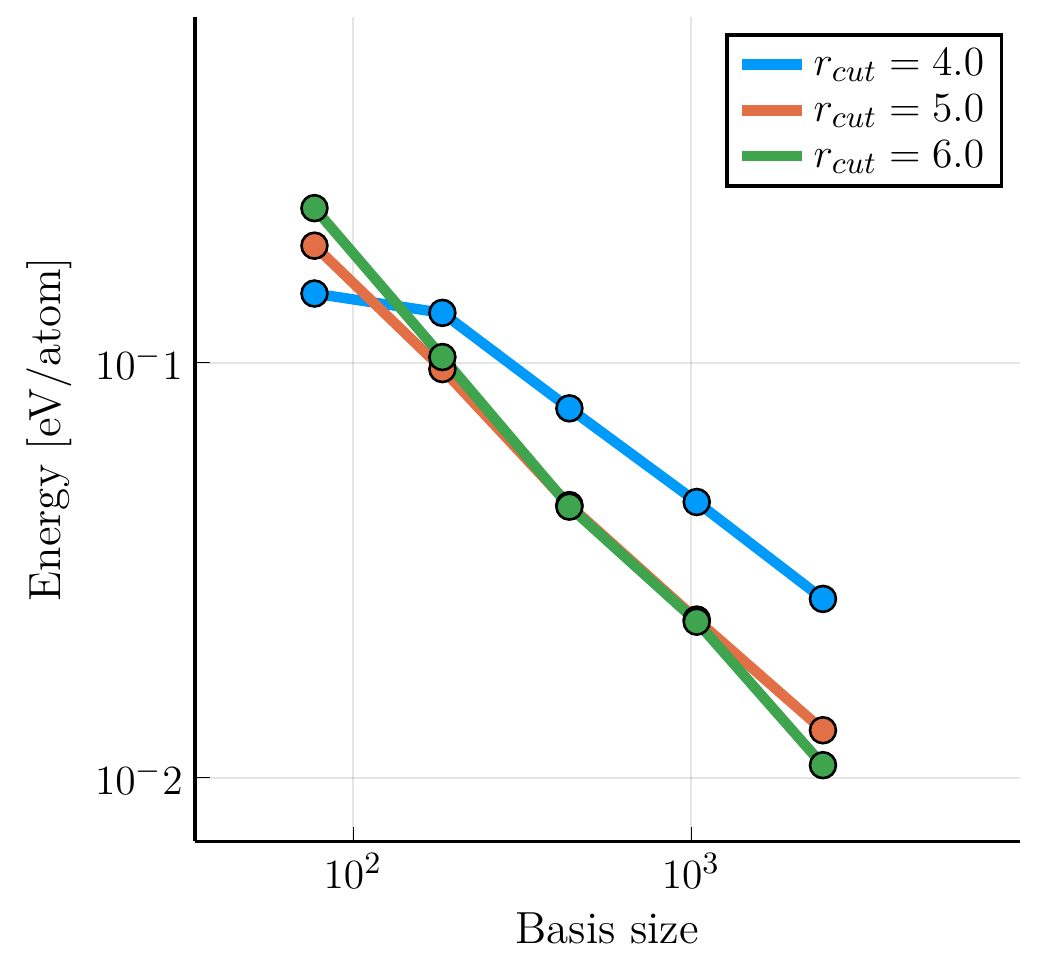}
    \includegraphics[height=6.7cm]{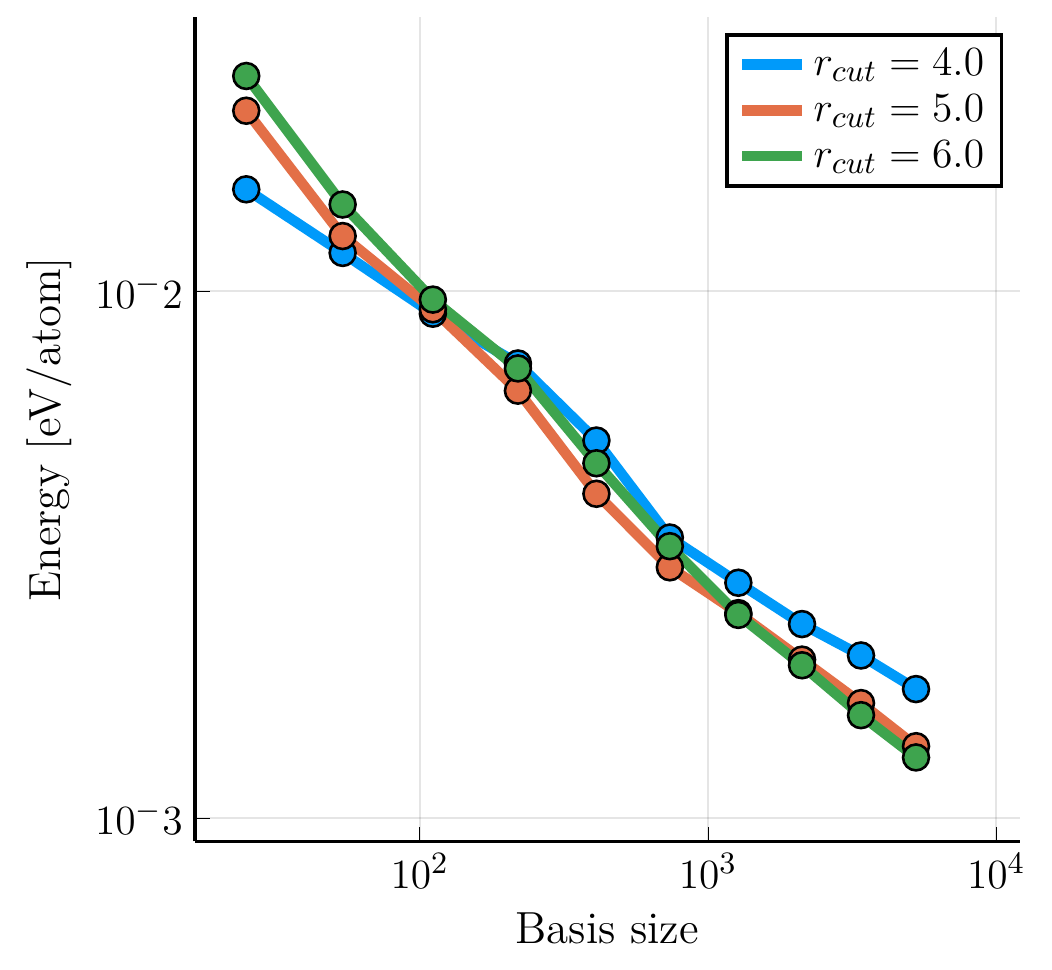}\qquad
    \caption{RMSE Convergence for varying $\rcut$ equal to 4.0 $\angstrom$, 5.0 $\angstrom$, 6.0 $\angstrom$ for W (left) and Si (right). A cross-over is observed for increasing basis size relative to $\rcut$.  }
    \label{fig:w_rcut}
\end{figure}

\section{Variations, Extensions and Remarks}
\label{sec:extensions}
We collect some natural variations and generalisations of the ideas presented in the foregoing sections to highlight how they can be applied in a more general fashion and to provide a road-map for future work.

\subsection{Combination, Composition and Self-consistency}
\label{sec:combcomp}
While the symmetric polynomial potentials introduced in the foregoing sections provide a general representation and approximation scheme for interatomic potentials, it will often be beneficial, and {\em in practice} maybe even necessary, to consider more general functional forms leading to nonlinear representations. We present three general constructions of this kind, which provide a modeller with a broad range of tools to significantly increase the ``design space''. 

\subsubsection{Combinations:} A general idea with considerable generality was proposed in~\cite{Drautz2019-er} is to admit site potentials consisting of {\em nonlinenar combinations}
\[
    V(R) = \mathcal{F}\big( \rho^{(1)}, \dots, \rho^{(P)} \big), 
\]
of atom-centered features  (called ``densities'' in~\cite{Drautz2019-er}) $\rho^{(p)} = \rho^{(p)}(R), p = 1, \dots, P$. The features $\rho^{(p)}$ will normally be permutation and isometry invariant and can therefore be represented in terms of the ACE basis. A prototypical example inspired by the Finnis-Sinclair and embedded atom models is 
\[
    V(R) = \rho^{(1)} + \sqrt{\rho^{(2)}}.
\]
Many generalisations are of course possible, including power laws, 
\[
    V(R) = \sum_{p = 1}^P \big| \rho^{(p)} \big|^{\alpha_p}, 
\]
or general polynomial expansions of $\mathcal{F}$ with parameters either determined {\it a priori} by insight into the chemistry of bonding, or by fitting to a training set. 

\subsubsection{Composition:} 
\label{sec:composition}
Inspired by classical interatomic potentials such as the Tersoff potential~\cite{Tersoff1986-rz} and the EDIP potential~\cite{Bazant1997-qb}, as well as nonlinear approximation schemes such as artificial neural networks, it is equally interesting to explore the {\em composition} of atom-centered features. For example, we can envision a site potential of the form 
\[
    E_i = V(\{\br_{ij}\}_j) = \mathcal{G}(s_0, \{(s_j, \br_{ij})\}_j ),
\]
where each $s_j$ is an atom-$j$-centered symmetric feature (possibly vectorial), such as a neighbour-count, a more general atomic neighbourhood statistic, a charge or charge density moments, and so forth. The function $\mathcal{G}$ can again be modelled as a symmetric polynomial but now in the expanded $(s, \br)$ space. Permutation symmetry is introduced by exactly the same mechanism as before, while rotation symmetry will now only be applied to the $\br$-component. Note though that the density must now include the centered atom described by $(s_0, {\bm 0})$, which requires only minor modifications to the framework outlined in the foregoing sections.

The features $s_j$ could be given analytically, or could again by presented as a symmetric polynomial and fitted to a training set, or several training sets to achieve transferrability. Applying the composition concept again to the $s_j$ feature, we may choose to write $s_i = \mathcal{S}(\{ (t_j, \br_{ij})\}_j)$.

An alternative way to write composition is to have $E_i$ depend only on a single symmetric feature centered at the site $i$, i.e., 
\[
    E_i = V(t_i; \{\br_j\}_j).
\]
This case can also be incorporated into the general framework by interpreting the centre-site as an atom of an auxiliary species, different from the neighbours $\{\br_j\}$ and analogous ideas to \S~\ref{sec:app:multi} can now be applied.

\subsubsection{Self-consistency}
A particularly interesting possibility opened up by the compositional structure is to have {\em self-consistent features}, i.e., the feature vector ${\bm s} := (s_i)_i$ could be defined through a fixed point map 
\[
    s_i = \mathcal{S}\big(s_i, \{ (s_j, \br_j)\}_{j}\big),
\]
which seems particularly natural for modelling a charge density. The mapping $\mathcal{S}$ can again be written as a symmetric polynomial and fitted to a training set or multiple training sets for transferrability.

\subsection{Evaluation of the orthogonal basis}
\label{sec:proj:orth}
We have seen that the computationally efficient basis ${\bm B}_N$
has the downside of introducing unphysical self-interactions that destroy the ``pure'' interaction order of the ${\bm \calB}_N$ basis. This means that for any practical parameterization, modifications of expansion coefficients of order $N$ need to be accompanied by corresponding changes of expansion coefficients $N' < N$ to balance the unphysical self-interactions.
This suggests that the ${\bm B}_N$ basis is ill-conditioned with respect to the canonical inner product, which we also confirm numerically in Table~\ref{tbl:ill_conditioning_ace}, leading to difficulties for the regularisation of the least-squares problem~\eqref{eq:LSQ_system}.

\begin{table}
    \begin{center} \small
    \begin{tabular}{cclllll}
        \toprule
        \multicolumn{2}{r}{total degree} &
             $4$ &       $6$ &       $8$ &      $10$ &      $12$  \\
        \midrule[0.075em]
         $N = 2$ &  all & $1.4 \times 10^{3}$ & $4.5 \times 10^{3}$ & $9.8 \times 10^{3}$ & $2.2 \times 10^{4}$ & $4.4 \times 10^{4}$ \\
                 & pure & $1.1 \times 10^{2}$ & $4.7 \times 10^{2}$ & $9.9 \times 10^{2}$ & $4.2 \times 10^{3}$ & $1.3 \times 10^{4}$ \\
         \midrule[0.02em]
         $N = 3$ &  all & $1.4 \times 10^{5}$ & $6.3 \times 10^{5}$ & $2.2 \times 10^{6}$ & $7.0 \times 10^{6}$ & $3.5 \times 10^{7}$ \\
                 & pure & $9.2 \times 10^{2}$ & $6.1 \times 10^{3}$ & $2.0 \times 10^{4}$ & $9.1 \times 10^{4}$ & $6.7 \times 10^{5}$ \\
         \midrule[0.02em]
         $N = 4$ &  all & $2.3 \times 10^{7}$ & $1.3 \times 10^{8}$ & $4.8 \times 10^{8}$ & $2.4 \times 10^{9}$ & $6.0 \times 10^{10}$ \\
                 & pure & $5.1 \times 10^{3}$ & $1.1 \times 10^{5}$ & $9.9 \times 10^{5}$ & $5.6 \times 10^{6}$ & $5.0 \times 10^{7}$ \\
         \midrule[0.02em]
         $N = 5$ &  all & $9.4 \times 10^{9}$ & $5.5 \times 10^{10}$ & $2.2 \times 10^{11}$ & $3.1 \times 10^{12}$ & $1.2 \times 10^{14}$ \\
                 & pure & $1.8 \times 10^{4}$ & $7.2 \times 10^{5}$ & $1.1 \times 10^{7}$ & $1.4 \times 10^{8}$ & $1.3 \times 10^{9}$ \\
         \midrule[0.02em]
         $N = 6$ &  all & $7.0 \times 10^{12}$ & $5.2 \times 10^{13}$ & $3.9 \times 10^{14}$ & $5.0 \times 10^{15}$ & $1.1 \times 10^{17}$ \\
                 & pure & $6.9 \times 10^{4}$ & $3.7 \times 10^{6}$ & $1.2 \times 10^{8}$ & $2.1 \times 10^{9}$ & $4.3 \times 10^{10}$ \\
         \bottomrule
    \end{tabular}
    \end{center}
    \vspace{2mm}
    \caption{Ill-conditioning of the ${\bm B}_N$ basis (symmetry-adapted correlations) due to self-interactions: ``$N$'' denotes the interaction order, ``all'' indicates that all orders are combined in the basis, while ``pure'' indicates that only basis functions with exact order $N$ are taken. For each basis the Gramian $G$ with respect to the canonical inner product is computed approximately by random samples, and the resulting condition numbers $\kappa(G) = \|G\| \| G^{-1}\|$ are displayed in the table. For the ``pure'' basis $\calB_N$ we have $\kappa(G) = 1$. }
    \label{tbl:ill_conditioning_ace}
\end{table}

For numerical stability of regression \S~\ref{sec:lsq}, and in particular for the purpose of regularisation and sparsification, it is therefore interesting to explore the possibility of obtaining an orthogonal basis.
Of course, the canonical symmetrised basis ${\bm \calB}_N$ retains the pure interaction order as well as orthogonality, but is seemingly too expensive to evaluate due to the explicit sum over permutations and clusters.

It turns out that there is a recursion formula for ${\bm \calB}_N$ basis functions which is not as efficient as the three point
recursion for univariate orthogonal polynomials, but possibly sufficient for our
purposes. This is due to the fact that the basis evaluation only occurs during
the training phase, while the final fitted potential will be represented
differently; cf. \S~\ref{app:implementation:V}. We present here a brief summary of the result, but leave the details for \cite{Etter2020}.

To simplify the notation we identify multi-indices $k = (n, l, m)$, $k_i = (n_i, l_i, m_i)$ and analogously for tuples $\bk = (\bn, \bl, \bbm)$. We focus only on the PI basis,
\begin{equation} \label{eq:defn_Asym}
\begin{split}
    \Asym_{\bn\bl\bbm} &= \Asym_\bk
         = \sum_{j_1 \neq \dots \neq j_N}
         \phi_{k_1}(\br_{j_1}) \cdots \phi_{k_N}(\br_{j_N}).
\end{split}
\end{equation}
from which the RPI basis functions $\calB_{\bn\bl i}$ can of course be constructed as described in \S~\ref{sec:symm}.

The result proven in \cite{Etter2020} is that there exist coefficients $\calP_{k_\beta k_{N+1}}^\kappa$ such that
\begin{align} \label{eq:Asym_recursion}
    &\Asym_{(\bk,k_{N+1})}  =
    \Asym_{\bk} A_{k_{N+1}}
    -
    \sum_{\beta = 1}^N
    \sum_{\kappa \in \N} \calP_{k_\beta k_{N+1}}^\kappa
    \Asym_{{\bk}[\beta, \kappa]},
    \\
    \notag
    &\bk[\beta, \kappa] := (k_1', \dots k_N')
    \quad \text{where} \quad
    k_\alpha' = \begin{cases}
        k_\alpha, &  \alpha \neq \beta  \\
        \kappa, &  \alpha = \beta.
    \end{cases}
\end{align}
The possibility of computing the $A_{\bn\bl\bbm}$ basis from the $\Asym_{\bn\bl\bbm}$ basis follows of course from the fact that both are complete. The key point of \eqref{eq:Asym_recursion} is that it provides a relatively efficient recursion.

We briefly hint at how \eqref{eq:Asym_recursion} is proven by considering the $N=2$ case:
\begin{align*}
    \notag
    A_{k_1 k_2}
    &=
    \sum_{j_1, j_2} \phi_{k_1}(\br_{j_1}) \phi_{k_2}(\br_{j_2})
    \\
    \notag
    &=
    \sum_{j_1 \neq j_2} \phi_{k_1}(\br_{j_1}) \phi_{k_2}(\br_{j_2})
    +
    \sum_{j} \phi_{k_1}(\br_j) \phi_{k_2}(\br_j)
    \\
    &=
    \Asym_{k_1 k_2} + \sum_j\phi_{k_1}(\br_j) \phi_{k_2}(\br_j),
\end{align*}
After expanding the product $\phi_{k_1}(\br) \phi_{k_2}(\br) =
\sum_\kappa \calP_{k_1 k_2}^\kappa \phi_\kappa(\br)$ we obtain \eqref{eq:Asym_recursion}. The recursion step for the general case is similar.

\begin{remark}
    \begin{enumerate}
    \item When the radial basis, including the cutoff, is purely polynomial then only finitely many coefficient $\calP_{k_\beta k_{N+1}}^\kappa$ are non-zero. In general, this expansion must be truncated which would then lead not to an orthogonal basis but a well-conditioned basis.

    \item The computational cost of \eqref{eq:Asym_recursion} is determined by the sparsity of the expansion coefficients $\calP^\kappa_{k_1 k_2}$ and this depends on the choice of polynomial basis; see \S~\ref{sec:radial}. From the perspective of \eqref{eq:Asym_recursion} the Chebyshev polynomials appear to be ideal since they have only two terms in the product formula.

    \item The basis functions $\calB_{\bn\bl i}$ are orthogonal within a fixed interaction order $N$, but not across different $N$.
    It is not clear yet how to overcome this limitation, however it may also be natural and sufficient to construct robust regression schemes; cf.~\S~\ref{sec:regularisation}.
    \end{enumerate}
\end{remark}


\subsection{Ill-posedness and Regularisation}
\label{sec:regularisation}
Except for the one-particle case, the body-order expansion representation \eqref{eq:cluster expansion} of a PES is not unique. For example, if we write a 3-body PES (interaction order 2) as
\[
    E(\{\br_j\}) = \sum_i \sum_{j_1 < j_2} V_2(\br_{j_1 i}, \br_{j_2 i}),
\]
then we can immediately rewrite this as
\begin{align*}
    E(\{\br_j\}) &=
    \!\!\sum_{j_1 < j_2 < j_3} \!\!\Big\{
       V_2(\br_{j_2 j_1}, \br_{j_3 j_1}) +
       V_2(\br_{j_1 j_2}, \br_{j_3 j_2}) +
       V_2(\br_{j_1 j_3}, \br_{j_2 j_3}) \Big\} \\
    &=: \!\!\sum_{j_1 < j_2 < j_3} \!\! W_3(\br_{j_1}, \br_{j_2}, \br_{j_3}),
\end{align*}
Here, $W_3$ is fully permutation invariant, while $V_3$ need not be: indeed we can define an alternative 3-body potential
\[
    \tilde{V}_2(\br_1, \br_2) = {\textstyle \frac{1}{3}}  W_3({\bm 0}, \br_1, \br_2),
\]
which gives rise to the same $W_3$ and hence the same PES $E$.

However, in general, we can define
\[
    W_N(\br_1, \dots, \br_N) :=
    \sum_{i = 1}^{N} V_{N-1}\big( (\br_j)_{j \neq i} \big).
\]
Let $R = \{ \br_j \}_{j = 1}^N$ be an arbitrary configuration consisting of $N$ particles, then
\[
    E(R) = \sum_{n = 1}^{N-1} \sum_{j_1 < \dots < j_n} W_n\big( (\br_{j_i})_i \big)
    + W_N(R).
\]
Thus, arguing inductively, we can determine all $W_N$ uniquely by evaluating $E$ on suitable configurations, that is, the $W_N$ are unique \cite{Drautz04}.

In \revii{practice}, however, we usually ``train'' ${\bf V} := (V_N)_{N = 0}^{\calN}$ on a data set that need not contain {\em any} small cluster configurations of this kind but only total energies and forces from condensed states in different phases and at different temperatures and pressures,
and possibly liquid phases. In this case, while ${\bf V}$  (or, rather, the corresponding ${\bf W} = (W_N)$) may be uniquely determined in principle, it is not obvious whether it is determined by the training set.

That is, suppose that the condensed state training set is sufficiently rich can we then uniquely determine the object of interest ${\bf V}$? 
We leave a detailed discussion of this underlying inverse problem, which is of fundamental interest to ``learning'' PESs, for future work. We merely note that the well-known uniqueness result for the pair interaction given the radial distribution function \cite{Henderson1974-ai} is an example of such an inversion, albeit an ill-conditioned one.

A second, possibly more severe source of ill-posedness of the least-squares system \eqref{eq:LSQ_system} is that practical training sets do not cover the entire space of $N$-body clusters even approximately. This is discussed in much more detail in~\cite{2019-regapips1} and a range of regularisation mechanisms are proposed in response. It is convenient and natural to treat each $V_N$ as an independent function and therefore choose norms of the form
\[
    \| {\bf V}\|^2 = \sum_{N = 1}^{\calN} \| V_N \|_{X_N}^2,
\]
where $X_N$ denotes a suitable function space. Natural choices are Sobolev norms and Sobolev semi-norms, which enforce a desired degree of smoothness.

In that vein,  \cite{2019-regapips1} proposed a generalised Tychonov
regularisation with a discretised $H^2$-type seminorm,
\[
    \| V_N \|_{\Delta}^2 := c_N \int_{\Omega_{N,\ro, \rcut}} |\Delta V_N|^2
    \approx \frac{1}{\# \mathcal{X}_N} \sum_{R \in \mathcal{X}_N}
        |\Delta V_N(R)|^2,
\]
where $\mathcal{X}_N$ is a quasi-Monte Carlo (low-discrepancy, Sobol) sequence.
This was used to construct highly transferrable (regular) symmetric polynomial potentials based on the aPIP formalism, despite significant gaps in the training set.
We refer to \cite{2019-regapips1} for further details.
While natural, this is by no means the only way to regularise and adding any {\it a priori} physical information through regularisation will in general further improve the fits.

\subsection{Cylindrical symmetry}
\label{app:othersym}
\def\bo{{\bm 0}}
\def\be{{\bm e}}
Throughout this paper we considered functions with full ${\rm O}(3)$ symmetry, but in many scenarios other (possibly simpler) symmetries arise which can be treated with the same overarching ideas. As a canonical example we consider a bond between two atoms at positions $\pm\frac12 \br$, surrounded by a {\em bond environment} $\{\br_j\}$. We wish to model a pair potential whose ``strength'' is determined by that environment, i.e., a function $V^{\rm env}(\br; \{\br_j\})$.

We begin by representing the environment in cylindrical coordinates aligned with the bond,
\[
    \br_j = r_j \cos \theta_j \be_x + r_j \sin\theta_j \be_y + z_j \be_z,
\]
where $(\be_x, \be_y, \be_z)$ are an orthonormal frame with $\be_z \propto \br$. The choice of $\be_x, \be_y$ is non-unique, but we will see that any choice that gives an orthonormal frame suffices. Finally, let $r := |\br|$ denote the length of the bond. In this representation, the function $V^{\rm env}$ is invariant under permutations of the environment, ${\rm O}(2)$ symmetries of the environment about the $\be_z$ axis as well as planar reflection along the $\be_z$ axis. The remaining rotations are already accounted for in the choice of coordinate system.

To construct a basis set inheriting these symmetries we begin, similarly as in the main part of the paper, by specifying four univariate basis sets,
\[
    \calP^0, \quad \calP^r, \quad \calP^z, \quad \calP^\theta,
\]
where $\calP^0, \calP^r, \calP^z$ are variants of the transformed polynomial basis introduced in \S~\ref{sec:radial} (but for $\calP^z$ the domain must be an interval $[-\rcut^z, \rcut^z]$ rather than $[0, \rcut^z]$) and the angular basis are the trigonometric polynomials,
\[
    \calP^\theta = \big\{
    P_k^\theta(\theta) := e^{i k \cdot \theta} \, |\, k \in \Z \big\}.
\]
We assume in the following that
\[
    P_k^z(-z) = (-1)^k P_k^z(z),
\]
which is natural if the basis is constructed through the procedure described in \S~\ref{sec:radial} and the orthogonality measure $\rho$ are symmetric about the origin. For example the $J^{\alpha,\alpha}$ Jacobi polynomials satisfy this condition.

For $\bk = (k^r, k^\theta, k^z) \in \mathbb{K}_1 := \N \times \Z \times \N$ we define the one-particle tensor product basis
\[
    \phi_{\bk}(\br_j) := P_{k^r}^r(r_j) P_{k^\theta}^\theta(\theta_j) P_{k^z}^z(z_j).
\]
For ${\bm K} := (\bk_1, \dots, \bk_N)$, $\bk_\alpha = (k_\alpha^r, k_\alpha^\theta, k_\alpha^z)$, the canonical pair-bond tensor product basis, involving $N$ neighbours, can then be written as
\[
    \phi_{n,{\bm K}}(r, \br_1, \dots, \br_{N})
    :=
    P_n(r)
    \prod_{\alpha = 1}^N
    \phi_{\bk_\alpha}(\br_\alpha).
\]

We now wish to symmetrise it with respect to the symmetries mentioned above: permutation of the environment atoms $\{\br_j\}_{j = 1}^N$, rotation about the $\be_z$ axis and reflection through the $(\be_x, \be_y)$ plane. In abstract terms we simply follow the procedure of integrating the basis with respect to the Haar measure, described in \S~\ref{sec:symm}, but now the symmetry groups are much simpler. A fairly straightforward calculation yields the following symmetric orthonormal basis:
\[
    \calB_{n, {\bm K}}(r, \br_1, \dots, \br_{N})
    :=
    \frac{1}{N!} \sum_{\sigma \in S_N}
    {\rm Re}\big[\phi_{n,{\bm K}}(r, \br_{\sigma 1}, \dots, \br_{\sigma N})\big],
    \qquad
    \text{for }
    n \in \N, {\bm K} \in \mathbb{K}_N^{\rm sym},
\]
where
\begin{align*}
    \mathbb{K}_N^{\rm sym} :=
    \Big\{  {\bm K} \in \mathbb{K}_1^N \,\big|\,
            \text{ordered,} \quad
            {\textstyle \sum_{\alpha = 1}^N} k_\alpha^\theta = 0, \quad
            {\textstyle \sum_{\alpha = 1}^N} k_\alpha^z \text{ is even}, \\
            \text{and the first non-zero $k_\alpha^\theta > 0$}
            \Big\}.
\end{align*}
The condition $\sum k_\alpha^\theta = 0$ enforces rotation symmetry about the $z$-axis, $\sum k_\alpha^z$ being even ensures reflection symmetry along the $z$-axis. Taking the real part ${\rm Re}[\phi_{n, \bf K}]$ imposes symmetry with respect to the reflection $(x, y, z) \mapsto (x, -y, z)$ and the final condition in the definition of  $\mathbb{K}_N^{\rm sym}$ merely avoids using the two identical basis functions characterised by $\pm {\bm k}^\theta$.

Next, we can adapt the density trick of \S~\ref{sec:ACEBasis} to obtain the alternative basis functions, with improved computational cost,
\begin{align*}
    B_{n, {\bm K}}\big(r, \{\br_j\}\big)
    &:= P_n(r) {\rm Re}\big[ A_{\bm K}^{\rm cyl}(\{\br_j\} )\big], \\
    A_{\bm K}^{\rm cyl}\big( \{\br_j\} \big)
    &:= \prod_{\alpha = 1}^N A^{\rm cyl}_{\bk_\alpha}\big(\{\br_j\}\big), \\
    A_{\bk_\alpha}^{\rm cyl}\big(\{\br_j\}\big)
    &:= \sum_{j} \phi_{\bk_\alpha}(\br_j),
\end{align*}
which we may call the {\em bond cluster expansion}.

Analogous comments apply regarding the significant gain in computational cost, the loss of orthogonality, conditioning of the basis, parameter estimation, well-posedness of the inverse problem, regularisation, and so forth.

\section{Conclusions and Outlook}
\label{sec:conclusions}
There is a rich literature emerging on the construction of interatomic potentials, and other molecular models, using machine learning methodology: general function approximation schemes combined with parameter estimation on large training sets obtained from electronic structure models.
The primary purpose of the present paper was to provide an entry-point for mathematicians, in particular numerical analysts, to this field. To that end, we reviewed, in full detail, the construction of an isometry and permutation invariant polynomial basis~\cite{Drautz2019-er}, based on the spherical harmonics expansions. We consider this to be one of the fundamental and most promising approaches to ``learning'' symmetric functions such as interatomic potentials. There are moreover close connections to other succesful schemes (cf. Appendix~\ref{app:otherbases}).

We demonstrated that, with some care, this construction leads to a {\em complete basis} that is highly efficient to evaluate. This creates the foundation for a numerical analysis framework to study the approximation properties as well as the associated parameter estimation schemes.
Throughout the paper we pointed out a variety of interesting challenges to further improve upon the basis construction as well as its theoretical foundation.

\appendix

\section{Spherical Harmonics}
\label{sec:app:SH}
In terms of the polar coordinate system
\[
    \br = r (\cos\varphi \sin\theta, \sin\varphi \sin\theta, \cos\theta)
\]
the complex spherical harmonics $Y_l^m : \bbS^2 \to \bbC$, $l \in \N, m = -l, \dots, l$ are given by
\[
    Y_l^m(\hat\br) := P_l^m(\cos\theta) e^{i m \varphi}.
\]
The normalisation factor, which is irrelevant for our purposes and is hence ignored, varies across different application domains. The associated Legendre polynomials $P_l^m$ are defined by (there are many equivalent definitions)
\[
    P_l^m(x) := \frac{(-1)^m}{2^l l!} (1-x^2)^{m/2} \frac{d^{l+m}}{dx^{l+m}} \big( x^2 - 1 \big)^l,
\]
however, we will never use this expression in \revii{practice}, but instead follow a numerically stable implementation described in \cite{Limpanuparb2014-tv}, which we slightly adapt to our purposes in \S~\ref{sec:Ylm:numstab}.

The two key properties of spherical harmonics that we will employ are (1) that they form an orthogonal basis of $L^2(\bbS^2)$; and (2) rotations $Y_l^m(Q \br)$ can be conveniently expressed and manipulated. To explain these manipulations we first need to define the Wigner D-matrices and Clebsch-Gordan coefficients. These are easiest to define {\em implicitly} through their action on spherical harmonics.

First, we define integration over ${\rm SO}(3)$,
\[
    \int_{{\rm SO}(3)} f(Q) \,dQ
\]
to be integration with respect to the unique normalised Haar measure on the group of rotations. This can be conveniently expressed analytically in terms of Euler angles, but we will never require this representation.

\begin{lemma}[Wigner-D matrices]
    Let $Q \in {\rm SO}(3)$, then there exist $D^l_{\mu m}(Q) \in \bbC$,
    $\mu, m = -l, \dots, l$, measurable with respect to the normalised Haar measure on ${\rm SO}(3)$,
    such that
    \begin{equation}
        \label{eq:app:rot_Ylm}
        Y_l^m(Q\br) = \sum_{\mu = -l}^l D_{\mu m}^l(Q) Y_l^\mu(\br).
    \end{equation}
\end{lemma}

\begin{lemma}[Clebsch-Gordan coefficients] There exist
   $C_{l_1m_1l_2m_2}^{\lambda \mu} \in \R$,
   $-l_i \leq m_i \leq l_i;
     |l_1-l_2| \leq \lambda \leq l_1+l_2;
     -\lambda \leq \mu \leq \lambda$
   such that
   \[
      Y_\lambda^\mu
      = \sum_{m_1 = -l_1}^{l_1} \sum_{m_2 = -l_2}^{l_2}
      C_{l_1m_1l_2m_2}^{\lambda\mu} Y_{l_1}^{m_1} Y_{l_2}^{m_2},
   \]
   where $C_{l_1m_1l_2m_2}^{\lambda\mu}$ = 0 if $m_1+m_2 \neq \mu$.
\end{lemma}

It is convenient, and common \revii{practice}, to extend the definition of the Wigner matrices $D^l_{\mu m}$ and Clebsch-Gordan coefficients  $C_{l_1m_1l_2m_2}^{\lambda \mu}$ to all indices $l, \mu, m; l_i, m_i, \lambda, \mu$ by zero, and to simply write, e.g., $\sum_{\lambda, \mu}$, encoding the finite-ness of the sum in the coefficients.

We proceed to summarise further standard results related to the above definitions that we employ throughout this article, to be found, e.g., in \cite{Byerly,Yutsis1962-yc}.

\begin{lemma} \label{th:SH_properties}
The following identities are true:
\begin{align}
    \label{eq:app:inv_Ylm}
    Y_l^m(-\br) &= (-1)^l Y_l^m(\br),
    \\
    \label{eq:app:prod_D}
    \revii{D_{\mu_1 m_1}^l D_{\mu_2 m_2}^{l_2}}
    &=
    \sum_{\lambda = |l_1-l_2|}^{l_1+l_2} C_{l_1 m_1 l_2 m_2}^{\lambda (m_1+m_2)} C_{l_1 \mu_1 l_2 \mu_2}^{\lambda (\mu_1+\mu_2)} D_{(m_1+m_2)(\mu_1+\mu_2)}^\lambda
    \\
      \label{eq:app:wigner_D_matrix}
      \displaystyle\int_{SO(3)} \revii{D_{\mu m}^{l}(Q)} \; dQ & = \delta_{l0} \delta_{m0} \delta_{\mu 0} \\
    \label{eq:app:wigner-D-orth}
    \int_{{\rm SO}(3)} \revii{D_{\mu_1,m_1}^{l_1}(Q) D_{\mu_2,m_2}^{l_2}(Q)} dQ &=
    \frac{8 \pi^2}{2l_1+1} (-1)^{m_1-\mu_1} \delta_{m_1+m_2}\delta_{\mu_1+\mu_2}\delta_{l_1,l_2} \\
    \label{eq:app:prod_Ylm}
    Y_{\ell_1}^{m_1}Y_{\ell_2}^{m_2}
    & = \sum_{\lambda, \mu}
    \tilde{C}_{l_1m_1l_2m_2}^{\lambda \mu} Y_\lambda^\mu,
\end{align}
with $\tilde{C}_{l_1m_1l_2m_2}^{\lambda \mu} :=
\sqrt{\frac{(2 l_1 + 1) (2 l_2 + 1)}{4 \pi (2 \lambda + 1)}}
    C_{l_10l_20}^{\lambda 0} C_{l_1m_1l_2m_2}^{\lambda \mu}$.
\end{lemma}

\subsection{Numerically stable evaluation of $Y_l^m, \nabla Y_l^m$}
\label{sec:Ylm:numstab}
Our implementation of the $Y_l^m$ basis functions directly uses the methods proposed in \cite{Limpanuparb2014-tv}. However, unlike \cite{Limpanuparb2014-tv} we also need numerically stable gradients $\nabla Y_l^m$, which requires some extra care.

Following~\cite{Limpanuparb2014-tv} we use the following recursion to evaluate the associated Legendre functions:
\begin{equation}  \label{eq:recursion:Plm}
\begin{split}
    P_0^0 &= \sqrt{\frac{1}{2\pi}} \\
    P_1^0 &= \sqrt{3} \cos \theta P_0^0 \\
    P_1^1 &= - \frac{3}{2} \sin\theta P_0^0 \\
    P_l^m &= A_l^m \big( \cos\theta P_{l-1}^m + B_l^m P_{l-2}^m \big), \quad m = 0, \dots, l-2 \\
    P_l^{l-1} &= C_l^{l-1} \cos\theta P_{l-1}^{l-1} \\
    P_l^l &= C_l^l \sin\theta P_{l-1}^{l-1}
\end{split}
\end{equation}
The constants $A_l^m, B_l^m, C_l^m$ are precomputed.

One numerical issue arises that is not addressed in \cite{Limpanuparb2014-tv}: the numerically stable evaluation of gradients $\nabla Y_l^m$ must be done with care due to the removable singularity of the analytic expression at the poles where $\sin\theta = 0$:
\[
\nabla Y_l^m = \frac{im P_l^m e^{im\varphi}}{r \sin \theta}
    \begin{bmatrix}
        - \sin\varphi  \\
        \cos\varphi  \\
        0
    \end{bmatrix}
    + \frac{\partial_\theta P_l^m e^{im \varphi}}{r}
    \begin{bmatrix}
        \cos\varphi \cos\theta \\
        \sin\varphi \cos\theta \\
        -\sin\theta
    \end{bmatrix}
\]
The $(\sin\theta)^{-1}$ singularity can be removed as follows:

{\it Case 1: } If $m = 0$ then $P_l^m$ does not depend on $\varphi$ hence the first term with the singularity just does not occur.

{\it Case 2: } If $m \neq 0$ then the recursions shown above show that $P_l^m$ is divisible by $\sin\theta$. Thus, for evaluating $\nabla Y_l^m$ we evaluate $P_l^m / \sin\theta$ instead of $P_l^m$ itself and this then immediately leads to an evaluation of $\nabla Y_l^m$ that is numerically stable as $\sin\theta \to 0$.

\subsection{Recursion for the $\revii{D^\bl_{\bmu\bbm}(Q)}$ coefficients}
\label{sec:app:Dl_recursion}
Recall the definition of $\revii{D^\bl_{\bmu\bbm}(Q)}$ in~\eqref{eq:def_Dmatrix}
\begin{equation} \label{eq:defn_calDlmmu}
    D^\bl_{\bmu\bbm}(Q) :=
        \prod_{\alpha = 1}^N D_{\mu_\alpha m_\alpha}^{l_\alpha}(Q),
\end{equation}
where $D^{l_\alpha}_{\mu_\alpha m_\alpha}(Q)$ are the Wigner D-matrices.

We will compute the $D^\bl_{\bmu\bbm}(Q)$ recursively. Assume therefore that we have computed all $D^\bl_{\bmu\bbm}(Q)$ with $\bl \in \N^{N-1}$ and now let $\bl \in \N^{N}, \bbm, \bmu \in \calM_{\bl}$. Applying the product formula~\eqref{eq:app:prod_D} for the D-matrices we obtain
\begin{align*}
   D^{\bl}_{\bmu\bbm}(Q)
   &=
   \bigg\{
      D_{\mu_1 m_1}^{l_1}(Q)
      D_{\mu_2 m_2}^{l_2}(Q)
   \bigg\}
   \cdot
   \prod_{\alpha = 3}^{N} D_{\mu_\alpha m_\alpha}^{l_\alpha}(Q)
       \\
   &=
      \bigg\{ \sum_{L =  |l_1-l_2|}^{l_1+l_2}
           C_{l_1m_1l_2 m_2}^{L (m_1+m_2)}
           C_{l_1\mu_1l_2 \mu_2}^{L (\mu_1+\mu_2)}
           D_{\mu_1+\mu_2,m_1+m_2}^{L}(Q)
       \bigg\}
       \cdot
    \prod_{\alpha = 3}^{N} D_{\mu_\alpha m_\alpha}^{l_\alpha}(Q).
\end{align*}
That is, we have the recursion formula
\begin{equation} \label{eq:app:Dlkm-recursion}
    D_{\bmu\bbm}^{\bl}(Q)
    = \sum_{L =  |l_1-l_2|}^{l_1+l_2}
    C_{l_1m_1l_2 m_2}^{L (m_1+m_2)}
    C_{l_1\mu_1l_2 \mu_2}^{L (\mu_1+\mu_2)}
    D_{\bmu'\bbm'}^{(L,\bl')}(Q),
\end{equation}
where
\begin{align*}
    \bmu' &=  (\mu_1+\mu_2, \mu_3, \dots, \mu_{N}), \\
    \bbm' &=  (m_1+m_2, m_3, \dots, m_N), \\
    (L,\bl') &=  (L, l_3, \dots, l_N).
\end{align*}
Using this recursion formula $(N-1)$ times, we obtain
\begin{align}
   \label{eq:app:recursion_Dmu_full}
   D^\bl_{\bmu\bbm}(Q)
   =
   \sum_{L_2 = |l_1-l_2|}^{l_1+l_2} \sum_{L_3 = |L_2-l_3|}^{L_2+l_3}\hspace{-.2cm} \ldots
   \hspace{-.2cm}
   \sum_{L_N = |L_{N-1}-l_N|}^{L_{N+1}+l_N}  &
   C_{l_1m_1l_2m_2}^{L_2M_2}C_{L_2M_2l_3m_3}^{L_3M_3} \ldots
   C_{L_{N-1}M_{N-1}l_Nm_N}^{L_NM_N} \\
   & \hspace{-1cm} C_{l_1\mu_1l_2\mu_2}^{L_2\tilde M_2}C_{L_2\tilde M_2l_3\mu_3}^{L_3\tilde M_3} \ldots
   C_{L_{N-1}\tilde M_{N-1}l_N\mu_N}^{L_N\tilde M_N}
   D_{\tilde M_NM_N}^{L_N}(Q), \nonumber
\end{align}
where $\tilde M_i = \sum_{j=1}^i \mu_i$, $M_i = \sum_{j=1}^i m_i$.

\revii{
Note that the integrated coefficients $\revii{\intD^\bl_{\bmu\bbm}}$ defined in~\eqref{eq:defn_intD} can be computed similarly, via the following recursion formula,
\begin{equation}
        \revii{\intD^\bl_{\bmu\bbm}}
    =
    \sum_{L = |l_1-l_2|}^{l_1+l_2}
    C_{l_1m_1l_2 m_2}^{L (m_1+m_2)}
    C_{l_1\mu_1l_2 \mu_2}^{L (\mu_1+\mu_2)}
    \revii{\intD_{\bmu'\bbm'}^{(L,\bl')}},
    \label{eq:Drecursion_appendix}
\end{equation}
with $\bbm',\bmu'$ and $\bl'$ as above.
}

\subsection{Dimensions of the RI and RPI bases}
\label{sec:app:dim_rot_inv}
In Table~\ref{tbl:more_dimensions}   we show the dimension of the space ${\rm span}~\big\{ b_{\bl\bbm} \,|\, \bbm \in \calM_\bl^0 \big\}$ for orders $N = 4, 5$  and different $\bl$, as well as the dimension of the corresponding RPI bases for different $\bn$. (For $N \leq 3$ the dimension is always zero or one.)

\begin{table}
    \begin{center} \footnotesize
\begin{tabular}{cccc}
\toprule
$ \bl $ & \,\#RI\, & $ \bn $ & \#RPI \\
\midrule[0.075em]
\multirow{5}{*}{\lbrak1, 1, 1, 1\rbrak}
                 & \multirow{5}{*}{3}  &  \lbrak$1, 1, 1, 1$\rbrak &  1 \\
  &  &  \lbrak$1, 1, 1, 2$\rbrak &  1 \\
  &  &  \lbrak$1, 1, 2, 2$\rbrak &  2 \\
  &  &  \lbrak$1, 1, 2, 3$\rbrak &  2 \\
  &  &  \lbrak$1, 2, 3, 4$\rbrak &  3 \\
\midrule
\multirow{1}{*}{\lbrak1, 1, 1, 3\rbrak}
                 & \multirow{1}{*}{1}  &  \lbrak$\_, \_, \_, \_ $\rbrak &  1 \\
\midrule
\multirow{4}{*}{\lbrak1, 1, 2, 2\rbrak}
                 & \multirow{4}{*}{3}  &  \lbrak$1, 1, 1', 1'$\rbrak &  2 \\
  &  &   \lbrak$1, 1, 1', 2'$\rbrak&  2 \\
  &  &  \lbrak$1, 2, 1', 1'$\rbrak &  2 \\
  &  &  \lbrak$1, 2, 1', 2'$\rbrak &  3 \\
\midrule
\multirow{1}{*}{\lbrak1, 1, 2, 4\rbrak}
                 & \multirow{1}{*}{1}  &    \lbrak$\_, \_, \_, \_ $\rbrak &  1 \\
\midrule
\multirow{1}{*}{\lbrak1, 1, 3, 5\rbrak}
                 & \multirow{1}{*}{1}  &  \lbrak$\_, \_, \_, \_ $\rbrak &  1 \\
\midrule
\multirow{2}{*}{\lbrak1, 2, 2, 3\rbrak}
                 & \multirow{2}{*}{3}  &  \lbrak$\_, 1, 1, \_ $\rbrak &  2 \\
  &  &  \lbrak$\_, 1, 2, \_ $\rbrak &  3 \\
\midrule
\multirow{3}{*}{\lbrak1, 3, 3, 3\rbrak}
                 & \multirow{3}{*}{3}  &  \lbrak$\_, 1, 1, 1 $\rbrak &  1 \\
  &  &  \lbrak$\_, 1, 1, 2 $\rbrak &  2 \\
  &  &  \lbrak$\_, 1, 2, 3 $\rbrak &  3 \\
\midrule
\multirow{5}{*}{\lbrak2, 2, 2, 2\rbrak}
                 & \multirow{5}{*}{5}  &   \lbrak$1, 1, 1, 1$\rbrak &  1 \\
  &  &  \lbrak$1, 1, 1, 2$\rbrak &  1 \\
  &  &  \lbrak$1, 1, 2, 2$\rbrak &  3 \\
  &  &  \lbrak$1, 1, 2, 3$\rbrak &  3 \\
  &  &  \lbrak$1, 2, 3, 4$\rbrak &  3 \\
\midrule
\multirow{3}{*}{\lbrak2, 2, 3, 3\rbrak}
                 & \multirow{3}{*}{5}  & \lbrak$1, 1, 1', 1'$\rbrak &  3 \\
  &  &  \lbrak$1, 1, 1', 2'$\rbrak &  3 \\
  &  &  \lbrak$1, 2, 1', 1'$\rbrak &  3 \\
  &  &  \lbrak$1, 2, 1', 2'$\rbrak &  5 \\
\bottomrule
\end{tabular}
\begin{tabular}{cccc}
\toprule
$ \bl $ & \,\#RI\, & $ \bn $ & \#RPI \\
\midrule[0.075em]
\multirow{5}{*}{\lbrak1, 1, 1, 1, 2\rbrak}
                 & \multirow{5}{*}{6}  &  \lbrak$1, 1, 1, 1, \_ $\rbrak &  1 \\
  &  &  \lbrak$1, 1, 1, 2, \_ $\rbrak &  2 \\
  &  & \lbrak$1, 1, 2, 2, \_ $\rbrak &  3 \\
  &  &  \lbrak$1, 1, 2, 3, \_ $\rbrak &  4 \\
  &  &  \lbrak$1, 2, 3, 4, \_ $\rbrak &  6 \\
\midrule
\multirow{1}{*}{\lbrak1, 1, 1, 1, 4\rbrak}
                 & \multirow{1}{*}{1}  &  \lbrak$\_, \_, \_, \_, \_ $\rbrak &  1 \\
\midrule
\multirow{3}{*}{\lbrak1, 1, 1, 2, 3\rbrak}
                 & \multirow{3}{*}{6}  &  \lbrak$1, 1, 1, \_, \_ $\rbrak &  2 \\
  &  & \lbrak$1, 1, 2, \_, \_ $\rbrak &  4 \\
  &  & \lbrak$1, 2, 3, \_, \_ $\rbrak &  6 \\
\midrule
\multirow{1}{*}{\lbrak1, 1, 1, 2, 5\rbrak}
                 & \multirow{1}{*}{1}  &   \lbrak$\_, \_, \_, \_, \_ $\rbrak &  1 \\
\midrule
\multirow{6}{*}{\lbrak1, 1, 2, 2, 2\rbrak}
                 & \multirow{6}{*}{9}  &  \lbrak$1,1, 1',1',1'$\rbrak &  2 \\
  &  &  \lbrak$1,1, 1',1',2'$\rbrak &  4 \\
  &  &  \lbrak$1,1, 1',2',3'$\rbrak &  6 \\
  &  &  \lbrak$1,2, 1',1',1'$\rbrak &  2 \\
  &  &  \lbrak$1,2, 1',1',2'$\rbrak &  5 \\
  &  &   \lbrak$1,2, 1',2',3'$\rbrak &  9 \\
\midrule
\multirow{4}{*}{\lbrak1, 1, 2, 2, 4\rbrak}
                 & \multirow{4}{*}{6}  & \lbrak$1,1, 1',1',\_ $\rbrak &  3 \\
  &  &  \lbrak$1,1, 1',2',\_ $\rbrak &  4 \\
  &  &  \lbrak$1,2, 1',1',\_ $\rbrak &  4 \\
  &  & \lbrak$1,2, 1',2',\_ $\rbrak &  6 \\
\midrule
\multirow{4}{*}{\lbrak1, 1, 2, 3, 3\rbrak}
                 & \multirow{4}{*}{9}  &  \lbrak$1,1,\_ , 1',1'$\rbrak &  4 \\
  &  & \lbrak$1,1,\_ , 1',2'$\rbrak &  6 \\
  &  &  \lbrak$1,2,\_ , 1',1'$\rbrak &  5 \\
  &  & \lbrak$1,2,\_ , 1',2'$\rbrak &  9 \\
\midrule
\multirow{3}{*}{\lbrak1, 2, 2, 2, 3\rbrak}
                 & \multirow{3}{*}{12}  &  \lbrak$\_,1,1, 1,\_$\rbrak &  3 \\
  &  &  \lbrak$\_,1,1, 2,\_$\rbrak &  7 \\
  &  & \lbrak$\_,1,2, 3,\_$\rbrak &  7 \\
\bottomrule
\end{tabular}
    \end{center}
    \label{tbl:more_dimensions}
    \vspace{2mm}
    \caption{Dimensions of rotation-invariant (RI) and rotation- and permutation-invariant (RPI) basis groups; cf. Table~\ref{tab:RIvsRPI_small} for additional details. In the third column, the numbers have to be understand as indices, e.g. $1 \neq 2 \neq 3$, and $1 = 1'$ or $1 \neq 1'$. 
    The $\_$ indicates a number that can be taken arbitrarily.}
\end{table}

\section{Related Bases and Descriptors}
\label{app:otherbases}
It is interesting to compare the construction presented in this paper with alternative representations of symmetric functions. Here, we focus on a small selection of bases and descriptors based also on symmetric polynomials. While the polynomial spaces spanned by these bases and descriptors are similar (often identical), the representation of the space is crucial, as this has influence on numerical stability, conditioning of the least-square system, as well as the computational cost of evaluating the potential.

We also briefly review the SOAP descriptor~\cite{Bartok2010-mv} to put it into the context of the present paper. Many other descriptors related to SOAP exist of course, which are listed in the Introduction, and which are also briefly discussed in the context of the ACE construction in~\cite{Drautz2019-er}. Note that we do not add the invariant wavelet scattering transform~\cite{Eickenberg-17} to the comparison because the representation is not atom-centered, although the angular component of the descriptor also relies on spherical harmonics.

\subsection{Atom-centred permutation-invariant potentials (aPIPs)}
The aPIPs construction~\cite{2019-regapips1} applies the invariant theory techniques~\cite{Derksen2015-km} of permutation-invariant potentials (PIPs) pioneered by Bowman and Braams~\cite{Braams2009-wi} to obtain a basis of many-body potentials (cf. \S~\ref{sec:pes}) constructed from permutation invariant polynomials.
There is significant freedom in this construction, but some specific {\em natural} choices lead to a basis that spans the same space of symmetric polynomials that we constructed in the present paper.

The aPIP construction begins with a choice of rotation-invariant (RI) coordinates, which can be e.g. distance variables (typical for PIPs~\cite{Braams2009-wi}) or distance and angles variables~\cite{2019-regapips1},
\[
    X := \big( (\xi(r_j))_{j=1}^N, (\cos \theta_{jj'})_{j < j'} \big),
\]
where $\xi$ is a distance transform as in \S~\ref{sec:radial}.
We focus on the latter choice; for pure distance coordinates the following discussion does not apply.

The number of RI coordinates is $d_N := \frac{N (N+1)}{2}$, i.e., $X \in \R^{d_N}$.
A permutation of the atom indices induces a permutation on the RI coordinates, and all these permutations form a subgroup of $S_{d_N}$ (strict subgroup for $N \ge 4$).
Invariant theory states that for each $N$, there exist polynomials satisfying Assumption~\ref{as:VN}(iii,iv) denoted by $I_1,I_2,\ldots,I_{d_N}$ and $J_1,\ldots,J_M$ (with $M \ge 0$) called primary and secondary invariants such that any polynomial $P_N(X)$ can be uniquely written as
\[
    V_N(\br_1, \dots, \br_{N}) = P_N(X) =
    \sum_{m = 0}^{M} J_m P_m(I_1,\ldots,I_{d_N}),
\]
where $I_n = I_n(X), J_m = J_m(X), J_0 \equiv 1$ and the $P_m$ are multivariate polynomials in $I_1, \dots, I_{d_N}$. The total interaction order-$\calN$ potential energy surface is then given by
\[
    E\big(\{\br_j\}_{j=1}^J\big) =
    \sum_{N = 0}^{\calN}
    \sum_{j_1 < \ldots < j_N} V_N(\br_{j_1}, \dots, \br_{j_N}).
\]
\noindent {\bf Remarks:}
\begin{enumerate}
    \item With suitable choice of basis of the $P_m$ polynomials, a finite aPIP basis spans the same space as the analogous ACE  basis. This can be see from the {\em addition theorem}: for $\br_1, \br_2 \in \bbS^2$, and the Legendre polynomial $P_l$ we can write
    \[
        P_l(\br_1 \cdot \br_2)  =
        C_l \sum_{m = -l}^l Y_l^m(\br_1) Y_l^{-m}(\br_2).
    \]
    Thus any multivariate polynomial in the variables $(\cos\theta_{jj'})_{j < j'}$ can be written as a linear combination of spherical harmonics.

    \item A first downside of the aPIPs representation is that it appears to be more difficult to construct an orthogonal basis than in the spherical harmonics construction.

    \item The second downside of the aPIPs representation is its computational cost. Even if evaluating a single order-$N$ term $V_N(\br_1, \dots, \br_N)$ is of similar or lower cost than for a density-based representation (this is difficult to analyse since the invariants are generated by a computer algebra systems) then the $\sum_{j_1 < \ldots < j_N}$ has $\binom{J}{N}$ terms, which scales poorly with increasing $J, N$. However, for multiple species the cost of the density-based representations increases significantly while the cost of aPIPs remains the same. Furthermore, the poor scaling of aPIPs is primarily with the number of neighbours. For open structures (semiconductors, organic molecules), it may be more favourable than constructions (including the one detailed in this paper) based on the density projection.

    \item Finally, a practical barrier is  that invariants are computed with a computer algebra system, which appears impossible at present for $N \ge 5$. It is currently impossible, in this way, to construct a complete aPIPs basis for a 6-body (5 neighbours) potential where all atoms are of the same species. However, it should be possible to use the constructions presented in this paper to generate these invariants. As far as we are aware this has not yet been explored.
\end{enumerate}

\subsection{Moment Tensor Potential (MTP)}
\label{sec:mtps}
MTPs \cite{Shapeev2016-pd} describe the local atomic environment using rotationally covariant {\em moment tensors}
\begin{equation}
    M_{\mu, \nu} = \sum_j f_{\mu}(r_{j}) r^{\nu}_{j} \underbrace{\hat{\br}_{j} \otimes \hat{\br}_{j} \otimes \dots \otimes \hat{\br}_{j}}_{\nu \text{ times}},
\end{equation}
which are the projection of the atomic density onto a basis of tensors.

Moment tensors are contracted to rotationally invariant scalars which are taken as the basis functions in a linear expansion, the moment tensor potential, e.g.,
\[
    M_{\mu, \nu} \cdot M_{\mu', \nu} =
    \sum_{j, j'} f_{\mu}(r_{j}) f_{\mu'}(r_{j'}) (\hat{\br}_j \cdot \hat{\br}_{j'})^\nu,
\]
but much more general contractions of multiple tensors are possible; cf.~\cite{Shapeev2016-pd}. The contraction of the moment tensors to basis functions is systematic and it is shown in~\cite{Shapeev2016-pd} that the basis functions span the space of all symmetric polynomials.
However, a similar issue occurs as in \S~\ref{sec:symm:rot+perm} that combining rotation and permutation invariance leads to linear dependence which must again be resolved numerically. Our impression is that at this stage spherical harmonics (i.e., ACE) have an advantage since the reduction to a basis can be performed in small blocks, possibly even symbolically.

On the other hand, we note that the $L^0$-regularisation strategy used for MTPs~\cite{Shapeev2016-pd} removes or at least diminishes the need for linear independence.

Alternatively, to construct a {\em basis} for the MTPs we may proceed as follows: The tensor product $\hat{\br} \otimes \dots \otimes \hat{\br}$ ($\nu$  times) can be expanded in spherical harmonics (see \cite[Appendix C]{Drautz20} for the details)
\begin{equation*}
\hat{r}_{n_1}  \hat{r}_{n_2} \dots \hat{r}_{n_{\nu}} = \sum_{l=0}^{\nu} \sum_{m=-l}^l X_{n_1 n_2 \dots n_{\nu}}^{lm} Y^m_l(\hat{\br}) \,,
\end{equation*}
with $n_i = x, y, z$ and the transformation matrix $X$ is obtained from generalized Clebsch-Gordan coefficients; cf.~\S~\ref{sec:rotinv:Drautz}.

Note that the number of spherical harmonics on the right hand side is $(\nu+1)^2$ while the number of matrix elements on the left hand side is $3^{\nu}$, that is, the expansion in spherical harmonics provides a sparse representation of the moment tensors.
We emphasize, however, that the MTP implementation does not store the full tensors $M_{\mu,\nu}$ either but uses a sparse representation; see the paragraph {\em Precomputation} in~\cite[\S~4.1]{Shapeev2016-pd}. It is therefore difficult without a more detailed study to compare the relative efficiency of the two approaches.

To make further contact to ACE, we define radial functions as
\begin{equation*}
P_{\mu \nu}(r_{j})  = f_{\mu}(r_{j}) r_{j}^{\nu} \,,
\end{equation*}
such that \eqref{eq:proj:defn_Anlm} reads
\begin{equation*}
A_{\mu \nu l m} = \sum_j P_{\mu \nu} (r_{j}) Y_l^m(\hat{\br}_j) \,,
\end{equation*}
and the moment tensors are written as
\begin{equation}
M_{\mu,\nu} = \sum_{l=0}^{\nu} \sum_{m=-l}^l X_{n_1 n_2 \dots n_{\nu}}^{lm} A_{\mu \nu l m} \,. \label{eq:AtoM}
\end{equation}
It is now easy to see how a specific MTP parameterization may be expressed in terms of the bases constructed in the present paper: replace the moment tensors in the MTP by \eqref{eq:AtoM} and order the resulting expansion according to powers of $A$. 

In summary, we see that suitably chosen finite sets of MTP and ACE basis functions span the same space, and that an explicit connection can be made. This connection moreover implies that theoretical results can be transferred; for example Shapeev's completeness result~\cite{Shapeev2016-pd} applies to the ACE basis while our result, under somewhat different assumptions, applies to MTPs.

\subsection{SOAP}
\def\calX{\mathcal{X}}
For nonlinear approximation schemes, such as artificial neural networks or Gaussian process regression, one usually encodes the symmetries of the target function (in our case, the site potential $V$) in a {\em descriptor map}.
That is, given an atomic environment $R_\iota := \{\br_{\iota j}\}_j$ one defines a descriptor $\calX(R_\iota) \in \R^M$ for some $M > 0$, which is invariant under permutations and isometries and then aims to construct a {\em nonlinear function} $F(\calX(R_\iota)) \approx V(R_\iota)$.
A particularly succesful construction is the SOAP descriptor~\cite{Bartok13}, which in the notation of this paper can simply be defined as follows: we choose a finite sub-basis ${\bm B} \subset {\bm B}_1 \cup {\bm B}_2$ and define
\[
    \calX^{(2)}(R_\iota) := \big( B(R_\iota) \big)_{B \in {\bm B}}.
\]
(In \revii{practice}, as a form of regularisation, the SOAP descriptor ``mollifies'' the basis functions, replacing the delta-distributions $\delta(\cdot - \br_j)$ in \eqref{eq:proj:defn_Anlm} with Gaussians.)
This descriptor has been used with considerable success for the construction of a variety of highly accurate and transferable interatomic potentials~\cite{Bartok2018-fk,Szlachta2014-qr,Deringer:2017ea,Deringer:2018eh,MAvracic:2018da,Mocanu:2018ck,Bernstein:2019jw,Veit:2019gp,Fujikake:2018ce,Nguyen:2018kh,Rowe:2018ct,Dragoni:2018je,Bartok:2017hz}.

A significant {\em theoretical} shortcoming is that $\calX^{(2)}$ is not injective; that is, one can construct atomic environments $R, R'$ which are distinct up to symmetries, but $\calX^{(2)}(R) = \calX^{(2)}(R')$~\cite{Pozdnyakov2020-gc}. This observation is independent of the basis subset chosen but is due to the fact that only 3-body correlations are accounted for in $\calX^{(2)}$.
A natural question therefore is to ask whether this can be remedied by introducing higher order correlations: for $\calN > 2$, let ${\bm B} \subset \bigcup_{N = 1}^\calN {\bm B}_N$ be a finite but sufficiently rich ACE basis, then this gives a descriptor map
\[
    \calX^{(\calN)}(R_\iota) := \big( B(R_\iota) \big)_{B \in {\bm B}}.
\]
Thus, the symmetric polynomial construction can be used as a natural generalisation of the SOAP descriptor to arbitrary interaction orders. It is currently unknown whether this can entirely overcome the lack of injectivity of $\calX^{(2)}$ (and other descriptors based on 3-body correlations), but it is clear that choosing sufficiently large order $\calN$ this yields at least a possible path to ensuring that all configurations in a given training set become distinguishable.
We refer to~\cite{Pozdnyakov2020-gc} for a more detailed discussion of these issues.

\subsection{SNAP}
A variant of the SOAP descriptor is represented in hyperspherical harmonics. To this end a 3-dimensional vector $\pmb{r} = (x,y,z)$ of length $r$ is mapped onto a 4-dimensional unit sphere,
\begin{equation*}
\varphi = \arctan( x/y), \qquad
\theta = \arccos( z/r), \qquad
\omega = \pi r/r_0.
\end{equation*}
Applying this construction to the bispectrum~\cite{Bartok2010-mv,Bartok13}, thus introducing 4-body correlations, is the basis for the Spectral Neighbor Analysis Potential (SNAP) \cite{Thompson2015-af}. The transformation onto a 4-dimensional unit sphere allows one to write the hyperspherical harmonics $Z^n_{lm}(\omega, \theta, \varphi)$ in the form of a radial function times a spherical harmonics,
\begin{equation}
Z^n_{lm}(\omega, \theta, \varphi) =  \phi_{nlm}(\pmb{r}) = P_{nl}(r)  Y^{m}_{l} (\theta,\varphi) \,, \label{eq:AtoZ}
\end{equation}
with a suitably defined function $P_{nl}$ \cite{Drautz20}. Because SNAP also uses linear regression, one may therefore represent it exactly in the form of an ACE.

\bibliography{ace}
\bibliographystyle{plain}

\end{document}